\documentclass[english,11pt]{smfart}

\usepackage{etex}

\usepackage{amsbsy}
\usepackage{amsmath,amsfonts,amssymb,amsthm,mathrsfs,mathtools}
\usepackage{bbm}

\usepackage{bm}

\usepackage[a4paper,vmargin={3cm,3cm},hmargin={3.5cm,3.5cm}]{geometry}
\usepackage[font=sf, labelfont={sf,bf}, margin=1cm]{caption}
\usepackage{graphicx}
\usepackage{epsfig}%pour les eps
\usepackage{latexsym}%encore des symboles
\usepackage{xcolor}
\usepackage{ae,aecompl}
\usepackage{soul,framed}
\usepackage{comment}

\usepackage{xcolor}
\usepackage[pdfpagemode=UseNone,bookmarksopen=false,colorlinks=true,urlcolor=blue,citecolor=blue,citebordercolor=blue,linkcolor=blue]{hyperref}
\usepackage{smfhyperref}
\usepackage[capitalize]{cleveref}

\usepackage{pstricks}
\usepackage{enumerate}
\usepackage{tikz,animate,media9}						%%%%%%%%%%
\usepackage{todonotes}
\usepackage{pifont}
\usepackage{bm,marvosym}
\usepackage{algorithm}
\usepackage{algorithmic}
\usepackage{microtype}
\usepackage{tikz}
\usepackage{float}
\usetikzlibrary{calc,positioning,arrows.meta,shapes.geometric}

\usepackage{amssymb,amsmath}
\usepackage{tikz}
\usetikzlibrary{arrows.meta,calc,positioning,shapes.geometric}
\tikzset{
  tv/.style  ={circle,fill=black,inner sep=1.7pt},
  rt/.style  ={draw,rectangle,inner sep=2.8pt,fill=white},
  lf/.style  ={draw,circle,inner sep=1.9pt,fill=white},
  ro/.style  ={draw,circle,inner sep=2.1pt,fill=white},
  mk/.style  ={draw,circle,inner sep=3.6pt},
  gad/.style ={draw=black!50,rounded corners=2.5pt,fill=black!6,
               minimum width=11mm,minimum height=7mm,
               font=\footnotesize,inner sep=1pt},
  e/.style   ={-{To[length=1.9mm]},semithick,shorten >=2pt,shorten <=2pt},
  d/.style   ={-{To[length=1.9mm]},thin,dashed,shorten >=2pt,shorten <=2pt},
  lbl/.style ={font=\footnotesize,inner sep=1.5pt},
  slbl/.style={font=\scriptsize,inner sep=1pt},
  note/.style={font=\scriptsize\itshape,inner sep=1pt},
  cell/.pic={
    \node[tv] (-v) at (0,0)     {};
    \node[rt] (-r) at (0,-1.8)  {};
    \node[lf] (-l) at (0,-3.4)  {};
    \node[tv] (-w) at (2.2,-1.8){};
    \draw[e] (-v) -- (-r);
    \draw[e] (-w) -- (-r);
    \draw[e] (-r) -- (-l);
    \draw[d] (-w) -- ($(-w)+(0.8, 0.7)$);
    \draw[d] (-w) -- ($(-w)+(0.8,-0.7)$);
  },
}

\tikzset{
  mkring/.style={draw,circle,inner sep=4.5pt,thick},
  cellM/.pic={
    \node[tv] (-v) at (0,0)     {};
    \node[mkring] at (-v) {};
    \node[rt] (-r) at (0,-1.8)  {};
    \node[lf] (-l) at (0,-3.4)  {};
    \node[tv] (-w) at (2.2,-1.8){};
    \draw[e] (-v) -- (-r);
    \draw[e] (-w) -- (-r);
    \draw[e] (-r) -- (-l);
    \draw[d] (-w) -- ($(-w)+(0.8, 0.7)$);
    \draw[d] (-w) -- ($(-w)+(0.8,-0.7)$);
  },
}

\usetikzlibrary{calc,positioning,arrows.meta,shapes.geometric}

\definecolor{nodecol}{HTML}{2B2B2B}
\definecolor{unarycol}{HTML}{D88C3B}
\definecolor{retcol}{HTML}{6F5A8E}
\definecolor{retedge}{HTML}{A89CC2}
\definecolor{edgecol}{HTML}{555555}
\definecolor{lblcol}{HTML}{4A4A4A}
\definecolor{efaint}{HTML}{9A9A9A}

\tikzset{
  tvD/.style      ={circle,fill=nodecol,inner sep=0pt,minimum size=4.2pt,outer sep=2.5pt},
  rtD/.style      ={draw=retcol,line width=0.5pt,rectangle,inner sep=0pt,minimum size=5pt,fill=white,outer sep=2.5pt},
  lfD/.style      ={draw=nodecol,line width=0.5pt,circle,inner sep=0pt,minimum size=9pt,
                    font=\tiny,text=lblcol,fill=white,outer sep=2.5pt},
  uvD/.style      ={circle,fill=unarycol,inner sep=0pt,minimum size=4.2pt,outer sep=2.5pt},
  eD/.style       ={draw=edgecol,line width=0.5pt,-{Stealth[length=1.6mm,inset=0.3mm]},shorten >=2pt,shorten <=2pt},
  retD/.style     ={draw=retedge,line width=0.5pt,-{Stealth[length=1.6mm,inset=0.3mm]},shorten >=2pt,shorten <=2pt},
  flow/.style     ={draw=lblcol,line width=0.45pt,-{Latex[length=2mm]}},
  panellbl/.style ={font=\small\itshape,inner sep=1pt,text=lblcol},
  vlbl/.style     ={font=\scriptsize,text=lblcol,inner sep=1pt},
  edgelbl/.style  ={font=\scriptsize,text=efaint,inner sep=1pt},
  flowlbl/.style  ={font=\scriptsize\itshape,text=lblcol,inner sep=2pt},
}

\definecolor{mincol}{HTML}{2E86C1}
\definecolor{maxcol}{HTML}{D88C3B}
\tikzset{
  hlmin/.style={draw=mincol,line width=2.6pt,opacity=0.55,line cap=round,line join=round,transform canvas={xshift=-1.05pt}},
  hlmax/.style={draw=maxcol,line width=2.6pt,opacity=0.55,line cap=round,line join=round,transform canvas={xshift=1.05pt}},
}

\definecolor{mincol}{HTML}{2E86C1}
\definecolor{maxcol}{HTML}{D88C3B}
\tikzset{
  hlmin/.style ={draw=mincol,line width=2.6pt,opacity=0.55,line cap=round,line join=round},
  hlmax/.style ={draw=maxcol,line width=2.6pt,opacity=0.55,line cap=round,line join=round},
  hlnone/.style={draw=none},
}
\newcommand{\minmaxnet}[2]{% #1 = estilo banda min, #2 = estilo banda max
  \coordinate (o)  at (1.6, 0);   \coordinate (v1) at (1.6,-0.9);
  \coordinate (vL) at (0.6,-1.7);  \coordinate (vR) at (2.6,-1.7);
  \coordinate (l1) at (-0.1,-2.6); \coordinate (w)  at (2.6,-2.6);
  \coordinate (l4) at (3.4,-2.6);  \coordinate (r)  at (1.6,-2.6);
  \coordinate (l3) at (1.6,-3.6);  \coordinate (l2) at (2.6,-3.6);
  \draw[#1] (o)--(v1)--(vL)--(r)--(l3);
  \draw[#2] (o)--(v1)--(vR)--(w)--(r)--(l3);
  \draw[eD] (o)--(v1);  \draw[eD] (v1)--(vL); \draw[eD] (v1)--(vR);
  \draw[eD] (vL)--(l1); \draw[eD] (vR)--(w);  \draw[eD] (vR)--(l4);
  \draw[eD] (w)--(l2);
  \draw[retD] (vL)--(r); \draw[retD] (w)--(r); \draw[retD] (r)--(l3);
  \node[tvD] at (o) {}; \node[tvD] at (v1) {}; \node[tvD] at (vL) {};
  \node[tvD] at (vR) {}; \node[tvD] at (w) {};
  \node[lfD] at (l1) {1}; \node[lfD] at (l2) {2}; \node[lfD] at (l4) {3};
  \node[lfD] at (l3) {$v$}; \node[rtD] at (r) {};
}

\usepackage{placeins}

\usepackage{bbm}

\usepackage{tcolorbox}

\definecolor{aleacolor}{rgb}{0.16,0.59,0.78}

\hypersetup{
	breaklinks,
	colorlinks=true,
	linkcolor=aleacolor,
	urlcolor=aleacolor,
	citecolor=aleacolor}

\newcount\colveccount
\newcommand*\colvec[1]{
	\global\colveccount#1
	\begin{pmatrix}
		\colvecnext
	}
	\def\colvecnext#1{
		#1
		\global\advance\colveccount-1
		\ifnum\colveccount>0
		\\
		\expandafter\colvecnext
		\else
	\end{pmatrix}
	\fi
}

\newcommand{\ndN}{\mathbb{N}}
\newcommand{\ndZ}{\mathbb{Z}}
\newcommand{\ndQ}{\mathbb{Q}}
\newcommand{\ndR}{\mathbb{R}}

\newcommand{\Prb}[1]{\mathbb{P}\left(#1\right)}

\newcommand{\Ex}[1]{\mathbb{E}[#1]}

\newcommand{\Exb}[1]{\mathbb{E}\left[#1\right]}

\newcommand{\Va}[1]{\mathbb{V}[#1]}

\newcommand{\Vab}[1]{\mathbb{V}\left[#1\right]}

\newcommand{\one}{{\mathbbm{1}}}

\newcommand{\convdis}{\,{\buildrel \mathrm{d} \over \longrightarrow}\,}

\newcommand{\convp}{\,{\buildrel \mathrm{p} \over \longrightarrow}\,}

\newcommand{\eqdist}{\,{\buildrel \mathrm{d} \over =}\,}

\newcommand{\He}{\mathrm{H}}
\newcommand{\he}{\mathrm{h}}

\newcommand{\Si}{\mathrm{S}}

\newcommand{\cA}{\mathcal{A}}
\newcommand{\cB}{\mathcal{B}}

\newcommand{\cF}{\mathcal{F}}

\newcommand{\cL}{\mathcal{L}}

\newcommand{\cN}{\mathcal{N}}
\newcommand{\cO}{\mathcal{O}}

\newcommand{\cR}{\mathcal{R}}

\newcommand{\cT}{\mathcal{T}}

\newcommand{\mL}{\mathsf{L}}

\newcommand{\mN}{\mathsf{N}}

\newcommand{\Cov}{\mathrm{Cov}}

\newtheorem{theorem}{Theorem}[section]

\newtheorem{corollary}[theorem]{Corollary}
\newtheorem{proposition}[theorem]{Proposition}
\newtheorem{lemma}[theorem]{Lemma}

\newtheorem{definition}[theorem]{Definition}

\numberwithin{equation}{section}

\title{\textbf{Limit laws of random simplex tree-child networks}}
\date{}

\author{V\'{\i}ctor J. Maci\'a}
\address[V\'{\i}ctor J. Maci\'a]{Universidad de La Laguna}
\email{victor.macia@ull.edu.es}
\author{Benedikt Stufler}
\address[Benedikt Stufler]{Vienna University of Technology}
\email{benedikt.stufler at tuwien.ac.at}

\begin{document}

\vspace {-0.5cm}

\begin{abstract}
We prove that the longer and shorter Sackin indices of a uniformly random simplex tree-child network with $n$ taxa admit joint distributional limits after rescaling by $n^{-7/4}$. The limiting distributions are described by functionals of a Brownian excursion. We also identify the limiting law of the height after rescaling by $n^{-3/4}$, thereby answering a question of Zhang~(2022). Moreover, we establish sharp tail bounds for the height, which imply convergence of all moments in the above distributional limits. We further obtain a scaling limit for the entire height profile of the leaves. Finally, we determine the local limits of large simplex networks around the fixed root, a uniformly random vertex, and a uniformly random leaf. 
\end{abstract}

\maketitle

%\subjclass{05C80, 60F17, 60J80, 92D15}
%\keywords{Random phylogenetic networks, simplex networks, tree-child networks, Sackin index, Brownian excursion, local weak limit, Benjamini--Schramm convergence}

\section{Introduction and main results}

Phylogenetic networks are a generalisation of phylogenetic trees that allow \emph{reticulation events} such as hybridisation, horizontal gene transfer, and recombination, phenomena that cannot be encoded by a strictly tree-like ancestry. Their combinatorial and probabilistic study has grown substantially in recent years, with significant attention given to various restricted classes such as tree-child, galled, one-component and level-$k$ networks, see for instance \cite{bienvenu2026b2indexgalledtrees,zbMATH07244280,zbMATH07751083,zbMATH07880797}.
\smallskip

Let $X$ be a finite taxon set. A \emph{simplex tree-child} network on $X$ is a planted binary phylogenetic network $N=(V,E)$ with root $o$, whose leaves are bijectively labelled by $X$, and which is both tree-child and one-component. That is, $N$ is a finite directed acyclic graph such that
$
    (d^-(o),d^+(o))=(0,1),
$
every vertex is reachable from the root $o$ via a directed path, every leaf has degree $(1,0)$, and every non-root, non-leaf vertex is either a tree vertex of degree $(1,2)$ or a reticulation vertex of degree $(2,1)$. The tree-child condition means that every non-leaf vertex is required to have at least one child which is either a tree vertex or a leaf. The one-component condition means that every reticulation vertex has a leaf as its unique child. See Figure~\ref{fig:classes} for an illustration.
\smallskip

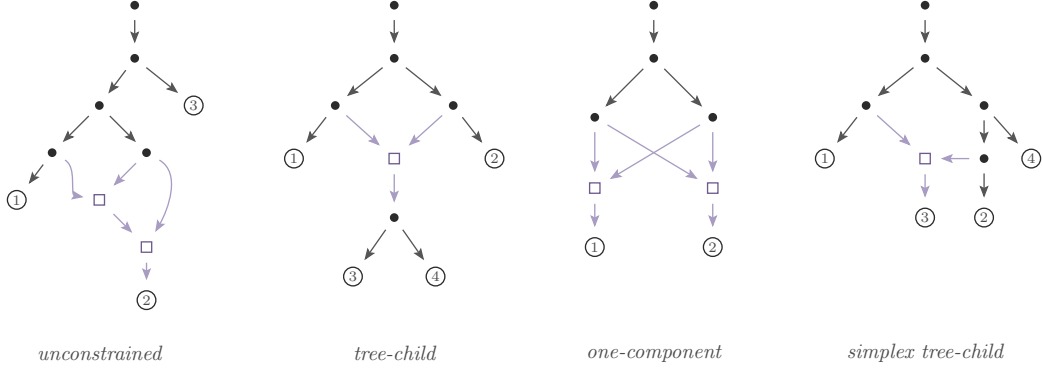
\begin{figure}[H]
\centering
\begin{tikzpicture}[x=1cm,y=1cm,scale=0.78,transform shape]
%% ===== (a) phylogenetic =====
\node[tvD]  (Ao)  at (1.6, 0)    {};
\node[tvD]  (Av1) at (1.6,-0.9)  {};
\draw[eD] (Ao)--(Av1);
\node[tvD]  (AA)  at (1.0,-1.7)  {};
\draw[eD] (Av1)--(AA);
\node[lfD]  (Al3) at (2.6,-1.7)  {3};
\draw[eD] (Av1)--(Al3);
\node[tvD]  (AB)  at (0.2,-2.5)  {};
\draw[eD] (AA)--(AB);
\node[tvD]  (AC)  at (1.8,-2.5)  {};
\draw[eD] (AA)--(AC);
\node[lfD]  (Al1) at (-0.4,-3.3) {1};
\draw[eD] (AB)--(Al1);
\node[rtD]  (Ar1) at (1.0,-3.3)  {};
\draw[retD] (AB) to[out=-30,in=170] (Ar1);
\draw[retD] (AC)--(Ar1);
\node[rtD]  (Ar2) at (1.8,-4.1)  {};
\draw[retD] (Ar1)--(Ar2);
\draw[retD] (AC) to[out=-30,in=60] (Ar2);
\node[lfD]  (Al2) at (1.8,-5.0)  {2};
\draw[retD] (Ar2)--(Al2);
\node[panellbl] at (1.0,-5.9) {unconstrained};
%% ===== (b) tree-child =====
\def\dx{4.4}
\node[tvD]  (Bo)  at (\dx+1.6, 0)    {};
\node[tvD]  (Bv1) at (\dx+1.6,-0.9)  {};
\draw[eD] (Bo)--(Bv1);
\node[tvD]  (BvL) at (\dx+0.6,-1.7)  {};
\draw[eD] (Bv1)--(BvL);
\node[tvD]  (BvR) at (\dx+2.6,-1.7)  {};
\draw[eD] (Bv1)--(BvR);
\node[lfD]  (Bl1) at (\dx-0.1,-2.6)  {1};
\draw[eD] (BvL)--(Bl1);
\node[lfD]  (Bl2) at (\dx+3.3,-2.6)  {2};
\draw[eD] (BvR)--(Bl2);
\node[rtD]  (Br) at (\dx+1.6,-2.6)  {};
\draw[retD] (BvL)--(Br);
\draw[retD] (BvR)--(Br);
\node[tvD]  (Bu) at (\dx+1.6,-3.6)  {};
\draw[retD] (Br)--(Bu);
\node[lfD]  (Bl3) at (\dx+0.9,-4.6)  {3};
\draw[eD] (Bu)--(Bl3);
\node[lfD]  (Bl4) at (\dx+2.3,-4.6)  {4};
\draw[eD] (Bu)--(Bl4);
\node[panellbl] at (\dx+1.6,-5.9) {tree-child};
%% ===== (c) one-component =====
\def\ddx{9.0}
\node[tvD]  (Co)  at (\ddx+1.4, 0)    {};
\node[tvD]  (Cv1) at (\ddx+1.4,-0.9)  {};
\draw[eD] (Co)--(Cv1);
\node[tvD]  (Cv)  at (\ddx+0.4,-1.9)  {};
\draw[eD] (Cv1)--(Cv);
\node[tvD]  (Cz)  at (\ddx+2.4,-1.9)  {};
\draw[eD] (Cv1)--(Cz);
\node[rtD]  (Cr1) at (\ddx+0.4,-3.1)  {};
\node[rtD]  (Cr2) at (\ddx+2.4,-3.1)  {};
\draw[retD] (Cv)--(Cr1);
\draw[retD] (Cz)--(Cr2);
\draw[retD] (Cv)--(Cr2);
\draw[retD] (Cz)--(Cr1);
\node[lfD]  (Cl1) at (\ddx+0.4,-4.1)  {1};
\draw[retD] (Cr1)--(Cl1);
\node[lfD]  (Cl2) at (\ddx+2.4,-4.1)  {2};
\draw[retD] (Cr2)--(Cl2);
\node[panellbl] at (\ddx+1.4,-5.9) {one-component};
%% ===== (d) simplex =====
\def\dddx{13.4}
\node[tvD]  (Do)  at (\dddx+1.6, 0)    {};
\node[tvD]  (Dv1) at (\dddx+1.6,-0.9)  {};
\draw[eD] (Do)--(Dv1);
\node[tvD]  (DvL) at (\dddx+0.6,-1.7)  {};
\draw[eD] (Dv1)--(DvL);
\node[tvD]  (DvR) at (\dddx+2.6,-1.7)  {};
\draw[eD] (Dv1)--(DvR);
\node[lfD]  (Dl1) at (\dddx-0.1,-2.6)  {1};
\draw[eD] (DvL)--(Dl1);
\node[tvD]  (Dw)  at (\dddx+2.6,-2.6)  {};
\draw[eD] (DvR)--(Dw);
\node[lfD]  (Dl4) at (\dddx+3.4,-2.6)  {4};
\draw[eD] (DvR)--(Dl4);
\node[rtD]  (Dr) at (\dddx+1.6,-2.6)  {};
\draw[retD] (DvL)--(Dr);
\draw[retD] (Dw)--(Dr);
\node[lfD]  (Dl3) at (\dddx+1.6,-3.6)  {3};
\draw[retD] (Dr)--(Dl3);
\node[lfD]  (Dl2) at (\dddx+2.6,-3.6)  {2};
\draw[eD] (Dw)--(Dl2);
\node[panellbl] at (\dddx+1.6,-5.9) {simplex tree-child};
\end{tikzpicture}
\caption{Different kinds of phylogenetic networks.}
\label{fig:classes}
\end{figure}
\smallskip

Given a network $N$ and a vertex $v$ of $N$, the shorter height $\he^-_N(v)$ is defined as the minimal length of a directed path from the root to $v$. The longer height $\he^+_N(v)$ is given by the maximal length of such a path. 

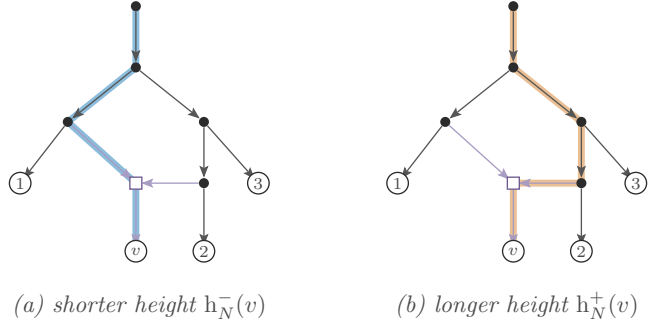
\begin{figure}[H]\centering
\begin{tikzpicture}[x=1cm,y=1cm,scale=0.9,transform shape]
  \minmaxnet{hlmin}{hlnone}
  \node[panellbl] at (1.65,-4.4) {(a)\ shorter height $\he^-_N(v)$};
\end{tikzpicture}\qquad\qquad
\begin{tikzpicture}[x=1cm,y=1cm,scale=0.9,transform shape]
  \minmaxnet{hlnone}{hlmax}
  \node[panellbl] at (1.65,-4.4) {(b)\ longer height $\he^+_N(v)$};
\end{tikzpicture}
\caption{Shorter and longer height of a leaf $v$.}
\label{fig:shortlongheight}
\end{figure}

With $L(N)$ the set of leaves of $N$, this allows us to define the longer and shorter Sackin indices
\[
	\Si^+(N) = \sum_{v \in L(N)} \he^+_N(v) \qquad \text{and} \qquad 	\Si^-(N) = \sum_{v \in L(N)} \he^-_N(v).
\]
The longer and shorter height of $N$ are defined by
\[
	\He^+(N) = \max_{v \in L(N)} \he^+_N(v) \qquad \text{and} \qquad 	\He^-(N) = \max_{v \in L(N)} \he^-_N(v).
\]

The number of simplex networks with a given number of taxa was determined by Cardona and Zhang~\cite{zbMATH07244280}.
Zhang proved in~\cite{zbMATH07570970} that the (longer) Sackin index of a random simplex network $\mN_n$ with $n$ taxa satisfies
\[
	\Ex{\Si^+(\mN_n)} = \Theta(n^{7/4}).
\]
See also recent work~\cite{tap} that determines similar bounds for classes of semi-simplex networks.

The present work introduces a probabilistic approach that significantly sharpens this result. We determine the precise limiting law for the fluctuations at this scale, both for the longer and shorter Sackin index. Let $(e(t), 0 \le t \le 1)$ denote a Brownian excursion of duration $1$ and set
\[
	\phi_+(x,y) = \max(x, y) \qquad \text{and} \qquad \phi_-(x,y) = \min(x, y).
\]
\begin{theorem}[Limit of the Sackin index]
	\label{te:main}
	Jointly for the two signs as $n \to \infty$
	\[
		2^{-3/2} n ^{-7/4} \Si^\pm(\mN_n) \convdis \int_{0}^1\int_{0}^1 \phi_{\pm}(e(x),e(y)) \,\,\mathrm{d}x \mathrm{d}y.
	\]
\end{theorem}

Theorem~\ref{te:main} contrasts with the behavior observed in other classes of phylogenetic networks. For example, for a random level-$k$ network $\mL_{n,k}$ with $n$ taxa and $k \ge 1$ fixed it follows by the main results of~\cite{zbMATH07751083} (see also~\cite{zbMATH08019266}) that there exist constants $c_k^+, c_k^->0$ such that
\[
c_k^+ n^{-3/2} \Si^+(\mL_{n,k}) \convdis \int_{0}^1 e(x)\,\,\mathrm{d}x
\qquad \text{and} \qquad 
c_k^- n^{-3/2} \Si^-(\mL_{n,k}) \convdis \int_{0}^1 e(x)\,\,\mathrm{d}x
\]
in distribution and with convergence of moments of every fixed order, as $n \to \infty$. This behaviour is universal for models of random networks, graphs and trees that belong to the universality class of Aldous' Brownian tree~\cite{MR1085326,MR1166406,MR1207226}. The limiting Brownian excursion area $\int_{0}^1 e(x)\,\,\mathrm{d}x$ has moments given in~\cite{zbMATH05728597}. 

Thus, simplex tree-child networks differ in their behaviour from this universality class since  the order of growth and the limit distributions change.  It is also notable that despite the somewhat unusual order of growth $n^{7/4}$, the limiting distributions are nevertheless functionals of a Brownian excursion.

Our approach also yields that the longer and shorter height jointly converge to the maximum of the same Brownian excursion.
\begin{theorem}[Limit of the height]
	\label{te:main2}
	Jointly as $n \to \infty$
	\begin{align*}
		2^{-3/2} n^{-3/4} (\He^+(\mN_n),\He^-(\mN_n)) \convdis \left(\sup_{0 \le t \le 1} e(t), \sup_{0 \le t \le 1} e(t)\right).
	\end{align*}
\end{theorem}

This answers a question posed by Zhang~\cite[Sec. 4]{zbMATH07570970}. Our third main result establishes sharp tail bounds for the height of these networks:
\begin{theorem}[Tail bounds for the height]
	\label{te:bound}
	There exist constants $C,c>0$ such that uniformly for all $n\ge 1$ and all $x\ge 0$
	\[
		\Prb{\He^-(\mN_n) \ge x n^{3/4}} \le \Prb{\He^+(\mN_n) \ge x n^{3/4}} \le C e^{-c x^2}.
	\]
\end{theorem}
In particular, since $\Si^\pm(\mN_n)\le n\He^+(\mN_n)$, Theorem~\ref{te:bound} implies that for every fixed $p>0$ the families
\[
	\bigl((n^{-7/4}\Si^+(\mN_n))^p\bigr)_{n\ge 1}
	\qquad\text{and}\qquad
	\bigl((n^{-7/4}\Si^-(\mN_n))^p\bigr)_{n\ge 1}
\]
are uniformly integrable. Consequently, all moments in Theorem~\ref{te:main} and Theorem~\ref{te:main2} converge. This yields for all $p \ge 1$
\[
	\Ex{\He^+(\mN_n)^p} \sim \Ex{\He^-(\mN_n)^p} \sim 2^{3p/2} n^{3p/4} \Exb{ \left(\sup_{0 \le t \le 1} e(t)\right)^p}
\]
and
\[
	\Exb{\Si^\pm(\mN_n)^p} \sim 2^{3p/2} n^{7p/4} \Exb{ (L^\pm)^p}
\]
for
\[
	L^\pm := \int_0^1\int_0^1 \phi_\pm(e(s),e(t))\,\mathrm{d}s\,\mathrm{d}t.
\]
The moments of the limiting distribution in Theorem~\ref{te:main2} are known~\cite{zbMATH01663205} and given by
\[
	\Exb{ \sup_{0 \le t \le 1} e(t)} = \sqrt{\pi/2}, \qquad \Exb{ \left(\sup_{0 \le t \le 1} e(t)\right)^p} = 2^{-p/2} p(p-1)\Gamma(p/2) \zeta(p), \qquad p \ge 2.
\] 
The moments of the limiting random variables $L^\pm$ of Theorem~\ref{te:main} do not appear in the literature. We establish a closed expression for them. For $k \ge 2$, let $\cT_k$ denote the set of planted binary phylogenetic trees with leaf set $[k]$ and no reticulation vertices, therefore we have $|\cT_k|=(2k-3)!!$.  If $T \in \cT_k$ and $\bm{x} = (x_e)_{e\in E(T)} \in [0,\infty)^{E(T)}$ with $E(T)$ the edge set of $T$, put
\[
	h_i(T,\bm x) := \sum_{e\in P_T(i)}x_e, \qquad i\in[k],
\]
where $P_T(i)$ is the set of edges on the path from the planted root to
leaf $i$.  For $p \ge 1$ and $T\in\cT_{2p}$, define
\begin{align*}
	F_{T,p}^\pm(\bm x) := \prod_{j=1}^p \phi_\pm\left(h_{2j-1}(T,\bm x),h_{2j}(T,\bm x)\right).
\end{align*}

\begin{theorem}
	\label{te:smoments}
	Let $p \ge 1$ be an integer and for each $T \in \cT_{2p}$ let $\bm{E}_T=(E_e)_{e\in E(T)}$ consist of independent exponential
	random variables with mean one. Then
	\begin{align*}
		\Ex{(L^\pm)^p} = 2^{3p/2-1}\frac{\Gamma(5p/2)}{\Gamma(5p-1)} \sum_{T\in\cT_{2p}} \Ex{F_{T,p}^\pm(\bm E_T)}
	\end{align*}
	with
	\begin{align*}
		\Ex{F_{T,p}^\pm(\bm E_T)}
		=
		\int_{(0,\infty)^{E(T)}}
		F_{T,p}^\pm(\bm x)
		\exp\left(-\sum_{e\in E(T)}x_e\right)
		\prod_{e\in E(T)}\mathrm{d}x_e.
	\end{align*}
	Moreover,
	\[
		\Ex{(L^\pm)^p} \in
			\begin{cases}
				\ndQ, & p\text{ even},\\
				\sqrt{2\pi}\,\ndQ, & p\text{ odd}.
			\end{cases}
	\]
	The first three moments are given by
	\begin{align*}
		\Ex{L^\pm} =
		\frac{(4\pm1)\sqrt{2\pi}}{16},\quad
		\Ex{(L^\pm)^2}
		=
		\frac{149\pm70}{336},\quad
		\Ex{(L^\pm)^3}
		=
		\frac{54968\pm35343}{393216}\sqrt{2\pi}.
	\end{align*}
\end{theorem}
In Lemma~\ref{le:compute} below we describe an explicit procedure using rational arithmetic operations for computing $\Exb{(L^\pm)^p}$ for all $p \ge 1$.
\smallskip

Furthermore, we determine the asymptotic law of the height profile. Let $\mu_e$ denote the occupation measure of $e$, that is
\[
	\mu_e(A) := \int_0^1 \one_{\{e(t)\in A\}}\,\mathrm{d}t, 
\]
where $A$ is a Borel subset of $[0,+\infty)$.
We also define the random probability measures
\begin{align*}
	\Pi_e^\pm :=(\mu_e \otimes \mu_e) \circ \phi_\pm^{-1}.
\end{align*}
Thus, for every bounded continuous function $f: \ndR \to \ndR$,
\[
	\int f(x) \Pi_e^\pm(\mathrm{d}x) = \int_0^1 \int_0^1 f\left(\phi_\pm(e(s),e(t))\right)\,\mathrm{d}s\mathrm{d}t.
\]
We obtain the following scaling limit of the height profile of the leaves of $\mN_n$.
\begin{theorem}[Limit of the height profile]
	\label{te:profile}
	Define
	\[
	\Pi_n^\pm := \frac{1}{n} \sum_{v\in L(\mN_n)} \delta_{ 2^{-3/2}n^{-3/4} \he_{\mN_n}^\pm(v)}.
	\]
	Then, jointly for weak convergence of probability measures on
	$[0,\infty)$,
	\[
	(\Pi_n^+,\Pi_n^-)\convdis(\Pi_e^+,\Pi_e^-).
	\]
	That is, for every bounded continuous function $f \in C_{b}([0,\infty))$, we have
	\[
		\int f\mathrm d\Pi_n^\pm \convdis \int_0^1\!\int_0^1 f\left(\phi_\pm(e(s),e(t))\right)\,\mathrm ds\mathrm dt,
	\]
	jointly for the two signs.
\end{theorem}

While the preceding results give an accurate description of the asymptotic global shape, we also describe the asymptotic local shape near specified points in~$\mN_n$. First, we establish a local weak limit that describes the asymptotic vicinity of the fixed root vertex:
\begin{theorem}[Local convergence near the root]\label{te:local_limit}
	There exists an infinite (deterministic) network $\hat{\mN}$ such that
	\[
		\mN_n \convdis \hat{\mN}
	\]
	in the local weak sense (at the root).
\end{theorem}

We describe the construction of~$\hat{\mN}$ in Section~\ref{ssec:root}. It is remarkable that although $\mN_n$ is random, the limiting network $\hat{\mN}$ is deterministic.
\smallskip

We also determine the local weak limits when the root is replaced by a
uniformly chosen vertex of a given type. These limits are explicit (deterministic) infinite networks whose constructions are described in
Subsections~\ref{ssec:unary}, \ref{ssec:ret}, and \ref{ssec:leaf}; Only the mixture in Theorem~\ref{te:loc-vertex} is random.

\begin{theorem}[Local convergence near tree vertices]\label{te:loc-unary}
On the event that $\mN_n$ has at least one tree vertex, let $u_n$ be
chosen uniformly at random among the tree vertices of $\mN_n$.
Then, there exists a (deterministic) network $(\hat{\cL},p_0)$ such that $(\mN_n,u_n)\convdis(\hat{\cL},p_0)$ in the local weak sense.
\end{theorem}

When $r_n$ is a uniformly chosen reticulation vertex, the limit is the deterministic network $(\hat{\cR},\rho_0)$, in which two disjoint decorated
spines meet at a common child reticulation carrying a pendant leaf, see Subsection \ref{ssec:ret} for further details.

\begin{theorem}[Local convergence near reticulation vertices]\label{te:loc-ret}
On the event that $\mN_n$ has at least one reticulation vertex, let $r_n$ be
chosen uniformly at random among them. Then, there exists a deterministic network $(\hat{\cR},\rho_0)$ such that $(\mN_n,r_n)\convdis(\hat{\cR},\rho_0)$
in the local weak sense.
\end{theorem}

When $w_n$ is a uniformly chosen leaf, the limit is the deterministic network
$(\hat{\mN}^{\mathrm{lf}},\ell_0)$, obtained from $\hat{\cR}$ by re-rooting
at the pendant leaf $\ell_0$ of the marked reticulation, see Subsection \ref{ssec:leaf} for further details.

\begin{theorem}[Local convergence near leaves]\label{te:loc-leaf}
Let $w_n$ be a vertex of $\mN_n$ chosen uniformly at random among its leaves.
Then, there exists a (deterministic) network $(\hat{\mN}^{\mathrm{lf}},\ell_0)$ such that $(\mN_n,w_n)\convdis(\hat{\mN}^{\mathrm{lf}},\ell_0)$ in the local weak
sense.
\end{theorem}

Combining the three cases, and using that tree vertices, reticulations,
and leaves each constitute an asymptotically deterministic proportion of the vertex set, we obtain the local weak limit at a uniformly chosen vertex:

\begin{theorem}[Local convergence near arbitrary vertices]\label{te:loc-vertex}Let $v_n$ be a vertex of $\mN_n$ chosen uniformly at random among all its vertices. Then $(\mN_n,v_n)\convdis(N^*,v^*)$, where
\[
  (N^*,v^*)\;\eqdist\;\begin{cases}
    (\hat{\cL},\,p_0) & \text{with probability } 1/2,\\[2pt]
    (\hat{\cR},\,\rho_0) & \text{with probability } 1/4,\\[2pt]
    (\hat{\mN}^{\mathrm{lf}},\,\ell_0) & \text{with probability } 1/4.
  \end{cases}
\]
The weights are the asymptotic proportions of  tree vertices, reticulation vertices and leaves among the total number of vertices of $\mN_n$.
\end{theorem}

Finally, the annealed limits above upgrade to quenched local convergence,
yielding laws of large numbers for the empirical distributions of
\(R\)-neighbourhoods.

\begin{theorem}[Quenched local convergence by vertex type]
\label{te:loc-quenched}
For every fixed \(R\ge0\) and every finite rooted directed marked graph
\(G\),
\[
  \frac{
    \#\{u\in V(\mN_n):
    u\text{ is a tree vertex and }U_R(\mN_n,u)\cong G\}
  }{
    \#\{\text{tree vertices of }\mN_n\}
  }
  \;\convp\;
  \Prb{U_R(\hat{\cL},p_0)\cong G}.
\]
If the denominator is zero, the ratio is defined to be \(0\).
The analogous statements hold for reticulation vertices, with limit
\((\hat{\cR},\rho_0)\), and for leaves, with limit
\((\hat{\mN}^{\mathrm{lf}},\ell_0)\).
\end{theorem}

As a direct consequence of the preceding type-specific limits, we obtain
quenched local convergence for the empirical neighbourhood distribution
over all vertices.

\begin{corollary}[Quenched local convergence near arbitrary vertices]
\label{co:loc-quenched-vertex}
For every fixed \(R\ge0\) and every finite rooted directed marked graph
\(G\),
\[
  \frac{
    \#\{v\in V(\mN_n):U_R(\mN_n,v)\cong G\}
  }{
    |V(\mN_n)|
  }
  \;\convp\;
  \Prb{U_R(N^*,v^*)\cong G},
\]
where \((N^*,v^*)\) is the mixture appearing in
Theorem~\ref{te:loc-vertex}.
\end{corollary}

Quenched local limits yield laws of large numbers for subgraph counts by
general principles~\cite{zbMATH07589157,zbMATH07696229}. They are also
useful for establishing limits of pattern frequencies
\cite{zbMATH07880797}.

\subsection*{Organisation}

In Section~\ref{sec:global} we study the global shape of a uniform simplex
tree-child network.  In Subsection~\ref{sec:deco} we recall the Cardona--Zhang decomposition of simplex networks. In Subsection~\ref{sec:size} we analyse the number $K_n$ of leaves whose parent is a tree vertex, proving that $K_n$ grows at order $\sqrt{n}$ with Gaussian fluctuations of order $n^{1/4}$. In Subsection~\ref{sec:tail} we derive uniform tail bounds for the longer height and prove Theorem~\ref{te:bound}. In Subsection~\ref{sec:scalheight} we prove distributional convergence of the rescaled longer and shorter heights and prove Theorem~\ref{te:main2}. In Subsection~\ref{sec:profile} we establish
the weak convergence of the rescaled height profile and prove Theorem~\ref{te:profile}. We also deduce the limiting law of the longer and shorter Sackin indices and prove Theorem~\ref{te:main}.  In
Subsection~\ref{sec:sackin-moments} we  compute the moments of the limiting
Sackin functionals and prove Theorem~\ref{te:smoments}. 
\smallskip

Section~\ref{sec:local} studies the asymptotic local shape of uniform
simplex tree-child networks. In Subsection~\ref{ssec:lwc} we introduce the
notion of marked local weak convergence used throughout the section. In
Subsection~\ref{ssec:root} we identify the deterministic local limit seen
from the planted root. In Subsections~\ref{ssec:unary}, \ref{ssec:ret},
and \ref{ssec:leaf} we establish the local limits around uniformly chosen
tree vertices, reticulation vertices, and leaves, respectively. In
Subsection~\ref{ssec:vertex} we combine these results to obtain the local
weak limit around a uniformly chosen vertex of the whole network. Finally,
in Subsection~\ref{ssec:quenched} we establish quenched local convergence
both by vertex type and over all vertices, yielding laws of large numbers
for empirical neighbourhood distributions.

\subsection*{Notation}

We  let $\ndN = \{1, 2, \ldots\}$ and $\ndN_0 = \{0, 1, 2, \ldots\}$ denote the sets of positive and nonnegative integers. For each $n \in \ndN_0$ we set $[n] = \{1, \ldots, n\}$, so that~$[0]=\emptyset$. All random variables considered are defined on a common probability space whose measure is denoted by $\mathbb{P}$. Unspecified limits are taken as $n$ tends to infinity. We use $\convdis$ and $\convp$ to denote convergence in distribution and probability. Equality in distribution is denoted by $\eqdist$. For any $a \in \ndR$ we let $\delta_a$ denote the Dirac measure at~$a$.
For an odd integer $m\ge 1$, the double factorial is $m!!=1\cdot 3\cdots m$, the
product of the odd integers from $1$ to $m$; in particular $1!!=1$. For any rooted tree $T$ and any vertex $v$ of $T$, we write $\he_T(v)$ for the height of $v$, that is, the number of edges on the unique path from the root of $T$ to $v$. We write $\He(T)$ for the maximal height of any vertex in $T$.

\section{Asymptotic study of the global shape}
\label{sec:global}

\subsection{Decomposition of simplex networks}
\label{sec:deco}

We recall structural and enumerative properties of simplex networks which were established by Cardona and Zhang~\cite{zbMATH07244280}.

Let $\cO\cN_n$ denote the class of simplex networks with $n$ taxa. We let $\cO\cN_{k,j}$ denote the subclass of simplex networks with $n=k+j$ taxa, for $k \geq 1$ and $j \geq 0$, in which leaves whose parent is a tree vertex are labelled from $1$ to $k$ and leaves whose parents are reticulation vertices are labelled from $k+1$ to $k+j$, when $j \geq 1$. Through relabelling, this implies
\[
	|\cO \cN_n| = \sum_{k=1}^n \binom{n}{k} |\cO \cN_{k,n-k}|.
\]

Let $\cT_k$ denote the collection of phylogenetic networks with $k$ taxa and no reticulation vertices. Note that any network in $\cT_k$ has $2k-1$ edges. 
\smallskip

The following decomposition reformulates the counting arguments of~\cite{zbMATH07244280}:
\begin{proposition}
	\label{pro:tt}
	There exists a bijection between $\cO\cN_{k,j}$ and triples $(T, Y, M)$ such that:
	\begin{enumerate}
		\item  $T \in \cT_k$.
		\item $Y = (y_i)_{1 \le i \le 2k-1} \in \ndN_0^{2k-1}$ with $\sum_{i=1}^{2k-1} y_i = 2j$.
		\item $M$ is an ordered partition of $[2j]$ into disjoint $2$-element subsets, that is the $2$-element subsets are labelled from $1$ to $j$.
	\end{enumerate}
\end{proposition}
\begin{proof}
	The bijection is established in~\cite{zbMATH07244280}. For completeness,
and to make explicit the construction used below, we recall here one of the maps. Each network from $\cO\cN_{k,j}$ is constructed from such a triple  $(T, Y, M)$ in a unique way  as follows. Fix a canonical ordering of the  $2k-1$ edges of $T$. For each $1 \le i \le 2k-1$ blow up the $i$th edge by a path of length $y_i+1$, thereby creating $y_i$ additional unary vertices. Hence in total we create $2j$ unary vertices. Fix a canonical ordering of these vertices.  For each $1 \le r \le j$ select the $r$th $2$-element subset $\{a,b\}$ from $M$, create a new reticulation vertex whose parents are the $a$th and $b$th unary vertices, and label its unique leaf child by $k+r$. The inverse construction, and hence the bijectivity of this
correspondence, follows from the decomposition in
\cite{zbMATH07244280}.
\end{proof}
The following picture illustrates the bijection given by Proposition \ref{pro:tt}:
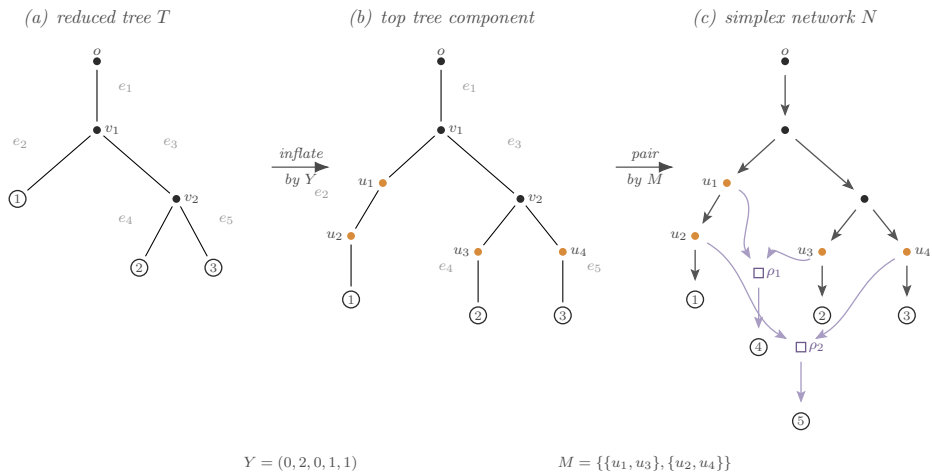
\begin{figure}[h]
\centering
\begin{tikzpicture}[x=1cm,y=1cm,scale=0.7,transform shape]

%% PANEL (a): T
\node[panellbl] at (1.5, 0.8) {(a)\;\,reduced tree $T$};

\node[tvD]  (Ao)  at (1.5,  0)   {};
\node[vlbl,above=2pt] at (Ao) {$o$};

\node[tvD]  (Av1) at (1.5,-1.3) {};
\node[vlbl,right=3pt] at (Av1) {$v_1$};
\draw (Ao)--(Av1);
\node[edgelbl] at (2.05,-0.50) {$e_1$};

\node[lfD]  (Al1) at (0.0,-2.6) {1};
\draw (Av1)--(Al1);
\node[edgelbl] at (0.05,-1.55) {$e_2$};

\node[tvD]  (Av2) at (3.0,-2.6) {};
\node[vlbl,right=3pt] at (Av2) {$v_2$};
\draw (Av1)--(Av2);
\node[edgelbl] at (2.90,-1.55) {$e_3$};

\node[lfD]  (Al2) at (2.3,-3.9) {2};
\draw (Av2)--(Al2);
\node[edgelbl] at (2.05,-3.00) {$e_4$};

\node[lfD]  (Al3) at (3.7,-3.9) {3};
\draw (Av2)--(Al3);
\node[edgelbl] at (3.95,-3.00) {$e_5$};

\draw[flow] (4.8,-2)--(5.9,-2);
\node[flowlbl,above] at (5.35,-1.95) {inflate};
\node[flowlbl,below] at (5.35,-2.05) {by $Y$};

%% PANEL (b): top tree component
\def\dx{6.5}
\node[panellbl] at (\dx+1.5, 0.8) {(b)\;\,top tree component};

\node[tvD]  (Bo)  at (\dx+1.5, 0)   {};
\node[vlbl,above=2pt] at (Bo) {$o$};

\node[tvD]  (Bv1) at (\dx+1.5,-1.3) {};
\node[vlbl,right=3pt] at (Bv1) {$v_1$};
\draw (Bo)--(Bv1);
\node[edgelbl] at (\dx+2.05,-0.50) {$e_1$};

\node[uvD]  (Bu1) at (\dx+0.4,-2.3) {};
\node[vlbl,left=3pt] at (Bu1) {$u_1$};
\draw (Bv1)--(Bu1);

\node[uvD]  (Bu2) at (\dx-0.2,-3.3) {};
\node[vlbl,left=3pt] at (Bu2) {$u_2$};
\draw (Bu1)--(Bu2);

\node[lfD]  (Bl1) at (\dx-0.2,-4.5) {1};
\draw (Bu2)--(Bl1);
\node[edgelbl] at (\dx-0.75,-2.50) {$e_2$};

\node[tvD]  (Bv2) at (\dx+3.0,-2.6) {};
\node[vlbl,right=3pt] at (Bv2) {$v_2$};
\draw (Bv1)--(Bv2);
\node[edgelbl] at (\dx+2.90,-1.55) {$e_3$};

\node[uvD]  (Bu3) at (\dx+2.2,-3.6) {};
\node[vlbl,left=3pt] at (Bu3) {$u_3$};
\draw (Bv2)--(Bu3);

\node[lfD]  (Bl2) at (\dx+2.2,-4.8) {2};
\draw (Bu3)--(Bl2);
\node[edgelbl] at (\dx+1.60,-3.90) {$e_4$};

\node[uvD]  (Bu4) at (\dx+3.8,-3.6) {};
\node[vlbl,right=3pt] at (Bu4) {$u_4$};
\draw (Bv2)--(Bu4);

\node[lfD]  (Bl3) at (\dx+3.8,-4.8) {3};
\draw (Bu4)--(Bl3);
\node[edgelbl] at (\dx+4.40,-3.90) {$e_5$};

\draw[flow] (\dx+4.8,-2)--(\dx+5.9,-2);
\node[flowlbl,above] at (\dx+5.35,-1.95) {pair};
\node[flowlbl,below] at (\dx+5.35,-2.05) {by $M$};

%% PANEL (c): N
\def\ddx{13.0}
\node[panellbl] at (\ddx+1.5, 0.8) {(c)\;\,simplex network $N$};

\node[tvD]  (Co)  at (\ddx+1.5, 0)   {};
\node[vlbl,above=2pt] at (Co) {$o$};

\node[tvD]  (Cv1) at (\ddx+1.5,-1.3) {};
\draw[eD] (Co)--(Cv1);

\node[uvD]  (Cu1) at (\ddx+0.4,-2.3) {};
\node[vlbl,left=3pt] at (Cu1) {$u_1$};
\draw[eD] (Cv1)--(Cu1);

\node[uvD]  (Cu2) at (\ddx-0.2,-3.3) {};
\node[vlbl,left=3pt] at (Cu2) {$u_2$};
\draw[eD] (Cu1)--(Cu2);

\node[lfD]  (Cl1) at (\ddx-0.2,-4.5) {1};
\draw[eD] (Cu2)--(Cl1);

\node[tvD]  (Cv2) at (\ddx+3.0,-2.6) {};
\draw[eD] (Cv1)--(Cv2);

\node[uvD]  (Cu3) at (\ddx+2.2,-3.6) {};
\node[vlbl,left=3pt] at (Cu3) {$u_3$};
\draw[eD] (Cv2)--(Cu3);

\node[lfD]  (Cl2) at (\ddx+2.2,-4.8) {2};
\draw[eD] (Cu3)--(Cl2);

\node[uvD]  (Cu4) at (\ddx+3.8,-3.6) {};
\node[vlbl,right=3pt] at (Cu4) {$u_4$};
\draw[eD] (Cv2)--(Cu4);

\node[lfD]  (Cl3) at (\ddx+3.8,-4.8) {3};
\draw[eD] (Cu4)--(Cl3);

\node[rtD]  (Crho1) at (\ddx+1.0,-4.0) {};
\node[vlbl,right=3pt,text=retcol] at (Crho1) {$\rho_1$};
\draw[retD] (Cu1) to[out=-30,in=120] (Crho1);
\draw[retD] (Cu3) to[out=-150,in=60] (Crho1);

\node[lfD]  (Cl4) at (\ddx+1.0,-5.4) {4};
\draw[retD] (Crho1)--(Cl4);

\node[rtD]  (Crho2) at (\ddx+1.8,-5.4) {};
\node[vlbl,right=3pt,text=retcol] at (Crho2) {$\rho_2$};
\draw[retD] (Cu2) to[out=-30,in=150] (Crho2);
\draw[retD] (Cu4) to[out=-150,in=30] (Crho2);

\node[lfD]  (Cl5) at (\ddx+1.8,-6.8) {5};
\draw[retD] (Crho2)--(Cl5);

\node[flowlbl] at (5.35,-7.6) {$Y=(0,2,0,1,1)$};
\node[flowlbl] at (\dx+5.35,-7.6) {$M=\{\{u_1,u_3\},\{u_2,u_4\}\}$};

\end{tikzpicture}
\caption{\footnotesize The bijection $(T,Y,M)\leftrightarrow N$ for $N\in\cO\cN_{3,2}$
with $Y=(0,2,0,1,1)$ and $M=\{\{u_1,u_3\},\{u_2,u_4\}\}$. Inflating the $i$-th
edge of $T$ into a path of length $y_i+1$ produces the top tree component;
pairing the unary vertices according to $M$ produces $N$.}
\label{fig:bijection}
\end{figure}

There are 
\[
|\cT_k| = \prod_{i=1}^{k-1}(2i-1) = \frac{(2k-2)!}{2^{k-1} (k-1)!}.
\]
choices for $T$,
\[
	\binom{2(k-1+j)}{2(k-1)}
\]
choices for $Y$, and
\[
	\prod_{i=0}^{j-1} \binom{2j - 2i}{2} = (2j)! 2^{-j} 
\]
choices for $M$. Hence, with $n=k+j$, the counting formula from~\cite{zbMATH07244280} follows:
\[
	|\cO \cN_{k,n-k}| = \frac{(2n-2)!}{2^{n-1} (k-1)!}  .
\]

Given a simplex network $N$ from $\cO \cN_{k,j}$  that corresponds to $(T, Y, M)$, we call $T(N) = T$ the associated \emph{reduced tree}. The tree obtained from $T$ by blowing up edges into paths according to $Y(N) = Y$ is called the \emph{top tree component}. We let $\tau(N)$ denote the rooted plane tree obtained from the top tree component by deleting the planted root, adding independent uniform left-right ordering at binary vertices, and forgetting the leaf labels.

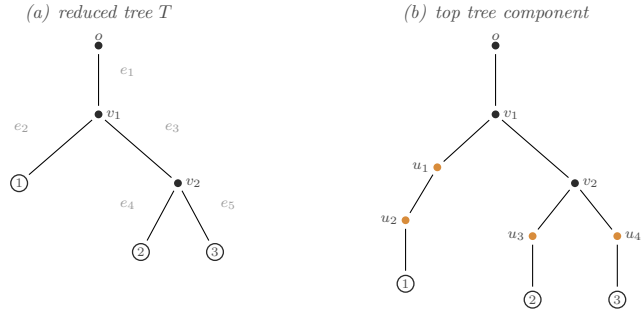
\begin{figure}[H]
\centering
\begin{tikzpicture}[x=1cm,y=1cm,scale=0.7,transform shape]
\node[panellbl] at (1.5, 0.6) {(a)\;\,reduced tree $T$};
\node[tvD]  (Ao)  at (1.5, 0)   {};
\node[vlbl,above=2pt] at (Ao) {$o$};
\node[tvD]  (Av1) at (1.5,-1.3) {};
\node[vlbl,right=3pt] at (Av1) {$v_1$};
\draw (Ao)--(Av1);
\node[edgelbl] at (2.05,-0.50) {$e_1$};
\node[lfD]  (Al1) at (0.0,-2.6) {1};
\draw (Av1)--(Al1);
\node[edgelbl] at (0.05,-1.55) {$e_2$};
\node[tvD]  (Av2) at (3.0,-2.6) {};
\node[vlbl,right=3pt] at (Av2) {$v_2$};
\draw (Av1)--(Av2);
\node[edgelbl] at (2.90,-1.55) {$e_3$};
\node[lfD]  (Al2) at (2.3,-3.9) {2};
\draw (Av2)--(Al2);
\node[edgelbl] at (2.05,-3.00) {$e_4$};
\node[lfD]  (Al3) at (3.7,-3.9) {3};
\draw (Av2)--(Al3);
\node[edgelbl] at (3.95,-3.00) {$e_5$};
\def\dx{7.5}
\node[panellbl] at (\dx+1.5, 0.6) {(b)\;\,top tree component};
\node[tvD]  (Bo)  at (\dx+1.5, 0)   {};
\node[vlbl,above=2pt] at (Bo) {$o$};
\node[tvD]  (Bv1) at (\dx+1.5,-1.3) {};
\node[vlbl,right=3pt] at (Bv1) {$v_1$};
\draw (Bo)--(Bv1);
\node[uvD]  (Bu1) at (\dx+0.4,-2.3) {};
\node[vlbl,left=3pt] at (Bu1) {$u_1$};
\draw (Bv1)--(Bu1);
\node[uvD]  (Bu2) at (\dx-0.2,-3.3) {};
\node[vlbl,left=3pt] at (Bu2) {$u_2$};
\draw (Bu1)--(Bu2);
\node[lfD]  (Bl1) at (\dx-0.2,-4.5) {1};
\draw (Bu2)--(Bl1);
\node[tvD]  (Bv2) at (\dx+3.0,-2.6) {};
\node[vlbl,right=3pt] at (Bv2) {$v_2$};
\draw (Bv1)--(Bv2);
\node[uvD]  (Bu3) at (\dx+2.2,-3.6) {};
\node[vlbl,left=3pt] at (Bu3) {$u_3$};
\draw (Bv2)--(Bu3);
\node[lfD]  (Bl2) at (\dx+2.2,-4.8) {2};
\draw (Bu3)--(Bl2);
\node[uvD]  (Bu4) at (\dx+3.8,-3.6) {};
\node[vlbl,right=3pt] at (Bu4) {$u_4$};
\draw (Bv2)--(Bu4);
\node[lfD]  (Bl3) at (\dx+3.8,-4.8) {3};
\draw (Bu4)--(Bl3);
\end{tikzpicture}
\caption{\footnotesize The two trees underlying a simplex network. The reduced tree $T$ has
$2k-1$ edges. The top tree component is obtained by replacing the $i$-th edge
of $T$ by a path of length $y_i+1$, introducing $\sum_i y_i = 2j$ unary
vertices. Here $k=3$, $j=2$, $Y=(0,2,0,1,1)$.}
\label{fig:trees}
\end{figure}

\subsection{Typical size of the reduced tree}
\label{sec:size}

Recall that $\mN_n$ is drawn uniformly at random from $\cO \cN_n$. Let $K_n$ denote the number of its leaves whose parent is a tree vertex. Thus, conditionally on $K_n=k$, the network $\mN_n$ is obtained from a uniform element of $\cO \cN_{k,n-k}$ by a uniform relabelling of its leaves.
\smallskip

Before proving the main result of this section, we introduce some notation and a few
preliminary lemmas. For $1\le k\le n$ write
\begin{equation}\label{eq: Kweights}
	\Prb{K_n=k}=\frac{w_{n,k}}{W_n},
	\qquad
	w_{n,k}:=\frac{\binom{n}{k}}{(k-1)!},
	\qquad
	W_n:=\sum_{j=1}^{n}\frac{\binom{n}{j}}{(j-1)!},
\end{equation}
with the convention $w_{n,k}:=0$ for $k\notin\{1,\dots,n\}$, so that $K_n$ is supported
on $\{1,\dots,n\}$, and set
\[
	m_n:=\bigl\lfloor\sqrt{n+1}\,\bigr\rfloor,
\]
so that $\sqrt{n+1}-1<m_n\le\sqrt{n+1}$. In fact $m_n-\sqrt{n+1} = O(1)$. The successive ratio of the weights,
\[
	r_{n,k}:=\frac{w_{n,k+1}}{w_{n,k}}=\frac{n-k}{k(k+1)}\qquad(1\le k\le n-1),
\]
is strictly decreasing in $k$, since its numerator decreases while its denominator
increases. As $r_{n,k}\ge1$ if and only if $(k+1)^2\le n+1$, equivalently
$k\le\sqrt{n+1}-1$, the choice of $m_n$ yields $r_{n,m_n-1}\ge1>r_{n,m_n}$; hence
$(w_{n,k})_{1\le k\le n}$ increases for $k\le m_n$, decreases for $k\ge m_n$, and has a
mode at $k=m_n$. We further write
\begin{equation}\label{eq:Ln}
  L_n(t)
  \;:=\;
  \log\left(\frac{w_{n,\,m_n+t}}{w_{n,\,m_n}}\right)
  \;=\;
  \begin{dcases}
    \phantom{-}\sum_{s=0}^{t-1}\log r_{n,\,m_n+s}, & t\ge 0,\\[1ex]
    -\sum_{s=1}^{-t}\log r_{n,\,m_n-s}, & t<0,
  \end{dcases}
\end{equation}
with the convention that the empty sum equals zero, so that \(L_n(0)=0\).
For integers \(t\) outside the range
\(1-m_n\le t\le n-m_n\), we set \(L_n(t):=-\infty\), and hence
\(e^{L_n(t)}=0\). A direct computation, completing $k^2+2k-n$ to $(k+1)^2-(n+1)$, factorises:
\begin{equation}\label{eq:ratio-id}
  \begin{aligned}
    1-r_{n,k}
    &=\frac{(k+1)^2-(n+1)}{k(k+1)}
    =\frac{\bigl(k+1-\sqrt{n+1}\bigr)\bigl(k+1+\sqrt{n+1}\bigr)}{k(k+1)},\\[0.5ex]
    \frac1{r_{n,k}}-1
    &=\frac{\bigl(k+1-\sqrt{n+1}\bigr)\bigl(k+1+\sqrt{n+1}\bigr)}{n-k}.
  \end{aligned}
\end{equation}
We now establish a uniform bound for \(L_n(t)\) throughout the central
window \(|t|\le M n^{1/4}\). This estimate will be used to control the
contribution of the typical fluctuations of \(K_n\) around its mean scale.

\begin{lemma}\label{le:asymptotic_estimate_L_n(t)}
For every fixed $M>0$, uniformly for $|t|\le M\,n^{1/4}$,
\[
  L_n(t)\;=\;\log\left(\frac{w_{n,m_n+t}}{w_{n,m_n}}\right)\;=\;-\frac{t^{2}}{\sqrt n}+o(1).
\]
In particular, whenever $t/n^{1/4}\to x$ one has $e^{L_n(t)}\to e^{-x^{2}}$.
\end{lemma}
\begin{proof}
For $|s|\le M n^{1/4}$, equation~\eqref{eq:ratio-id} gives, uniformly in $s$,
\begin{align*}
1-r_{n,m_n+s}
&=\frac{(\rho_n+s)(\rho_n+s+2\sqrt{n+1})}{(m_n+s)(m_n+s+1)}  \\
&=\frac{2(\rho_n+s)}{\sqrt n}\bigl(1+O(n^{-1/4})\bigr)
 =\frac{2s}{\sqrt n}+O(n^{-1/2}),
\end{align*}
where $\rho_n:=m_n+1-\sqrt{n+1}\in(0,1]$. Since
$1-r_{n,m_n+s}=O(n^{-1/4})$, we also have, uniformly for
$|s|\le M n^{1/4}$,
\[
\log r_{n,m_n+s}
=
\log\bigl(1-(1-r_{n,m_n+s})\bigr)
=
-\frac{2s}{\sqrt n}+O(n^{-1/2}).
\]

Using the definition of $L_n(t)$, with empty sums interpreted as zero, we obtain for
$|t|\le M n^{1/4}$
\[
L_n(t)
=
\begin{dcases}
\displaystyle
-\frac{2}{\sqrt n}\sum_{s=0}^{t-1}s
+O\!\left(\frac{t}{\sqrt n}\right),
& t\ge0,\\[2ex]
\displaystyle
-\frac{2}{\sqrt n}\sum_{s=1}^{-t}s
+O\!\left(\frac{|t|}{\sqrt n}\right),
& t<0.
\end{dcases}
\]
\smallskip

The same estimate includes the cases $t=0$ and $|t|=1$, with empty sums
interpreted as zero. For $t\ge0$,
\[
\sum_{s=0}^{t-1}s=\frac{t(t-1)}2,
\]
whereas for $t<0$,
\[
\sum_{s=1}^{-t}s=\frac{(-t)(-t+1)}2.
\]
In either case, the sum equals $t^2/2+O(|t|)$. Hence, uniformly for
$|t|\le M n^{1/4}$,
\[
L_n(t)
=
-\frac{t^2}{\sqrt n}
+O\!\left(\frac{|t|}{\sqrt n}\right)
=
-\frac{t^2}{\sqrt n}+o(1).
\]
This concludes the proof.
\end{proof}

 The central limit theorem for $K_n$ reduces to the convergence of a rescaled lattice sum to an integral. This is the  content of the following proposition. 

\begin{proposition}
	\label{prop:riemann}
	Let $h_n>0$ with $h_n\to0$, and for each $n$ let $g_n:\ndZ\to[0,\infty)$. Assume:
	\begin{enumerate}
		\item\label{it:rm-point} there is $g:\ndR\to[0,\infty)$, continuous almost everywhere, such that $g_n(t_n)\to g(x)$ whenever $t_n\in\ndZ$ and $h_n t_n\to x$ with $x$ a continuity point of $g$;
		\item\label{it:rm-dom} there are absolute constants $C,c>0$ with $g_n(t)\le C\,e^{-c\,(h_n t)^2}$ for all $n$ and all $t\in\ndZ$.
	\end{enumerate}
	Then
	\[
		h_n\sum_{t\in\ndZ}g_n(t)\xrightarrow[n\to\infty]{}\int_{\ndR}g(x)\,\mathrm{dx} .
	\]
\end{proposition}
\begin{proof}
	Define the step function $\Phi_n(x):=g_n(\lfloor x/h_n\rfloor)$ for $x\in\ndR$. Since $\Phi_n$ is constant equal to $g_n(t)$ on each interval $[h_n t,h_n(t+1))$, of length $h_n$,
	\[
		\int_{\ndR}\Phi_n(x)\, \mathrm{dx}=h_n\sum_{t\in\ndZ}g_n(t).
	\]
	Fix a continuity point $x$ of $g$ and put $t_n:=\lfloor x/h_n\rfloor$, so that $h_n t_n\in(x-h_n,x]$ and hence $h_n t_n\to x$. By~(\ref{it:rm-point}), $\Phi_n(x)=g_n(t_n)\to g(x)$; thus $\Phi_n\to g$ at every continuity point of $g$, that is, almost everywhere.
\medskip
    
	By~(\ref{it:rm-dom}), we have $\Phi_n(x)\le C\,e^{-c(h_n\lfloor x/h_n\rfloor)^2}$. For $|x|\ge h_n$ we have $|h_n\lfloor x/h_n\rfloor|\ge|x|-h_n$, whereas for $|x|<h_n$ the exponential is at most $1$. Hence, once $h_n\le1$,
	\[
		\Phi_n(x)\le C\Bigl(\mathbf 1_{\{|x|<1\}}+e^{-c(|x|-1)^2}\Bigr)=: G(x),
	\]
	and $G$ is integrable on $\ndR$. Dominated convergence gives $\int_{\ndR}\Phi_n\stackrel{n \rightarrow \infty}{\longrightarrow}\int_{\ndR} g$, which is the claim.
\end{proof}

\begin{lemma}
\label{le:domination}
There are absolute constants $C_1\ge1$ and $c_1,c_2>0$ such that, for all
$n\ge1$,
\[
\frac{w_{n,m_n+t}}{w_{n,m_n}}
\le C_1\,e^{-c_1 t^2/\sqrt n},
\qquad 1-m_n \le t \le m_n,
\]
and
\[
\sum_{t>m_n}\frac{w_{n,m_n+t}}{w_{n,m_n}}\le C_1\,e^{-c_2\sqrt n}.
\]
\end{lemma}
\begin{proof}
Write $\delta_n:=n+1-m_n^2\in[0,2m_n+1]$ and fix $0\le s\le m_n$. By equation \eqref{eq:ratio-id}, we have
\[
  1-r_{n,m_n+s}
  =\frac{2m_n+1-\delta_n+s(2m_n+s+2)}{(m_n+s)(m_n+s+1)} .
\]
Since $\delta_n\le 2m_n+1$, the numerator is at least $s(2m_n+s+2)\ge 2m_n s$,
while the denominator satisfies $(m_n+s)(m_n+s+1)\le(2m_n)(2m_n+1)\le 6m_n^2$.
Hence there is an absolute constant $c>0$ with
\[
  1-r_{n,m_n+s}\ge\frac{2m_n s}{6m_n^2}=\frac{s}{3m_n}\ge c\,\frac{s}{m_n},
  \qquad 0\le s\le m_n .
\]
The case $t=0$ is trivial, since $L_n(0)=0$. For $1\le t\le m_n$, as $r_{n,k}$ is
non-increasing in $k$ and $r_{n,m_n}\le1$, we have $r_{n,m_n+s}\le1$ for all
$0\le s\le t-1$; thus the bound $\log(1-y)\le-y$ (valid for $y<1$) gives
$\log r_{n,m_n+s}\le -c\,s/m_n$ for each such $s$, and summation yields
\[
  L_n(t)=\sum_{s=0}^{t-1}\log r_{n,m_n+s}
  \le -\frac{c}{m_n}\sum_{s=0}^{t-1}s
  = -\frac{c\,t(t-1)}{2m_n}
  \le C-\frac{c\,t^2}{m_n},
\]
where the last inequality absorbs the case $t=1$ into the absolute constant $C$.
\smallskip

Now fix $1\le u\le m_n-1$. Then $k=m_n-s$ ranges over
$1\le k\le m_n-1$ for $1\le s\le u$, below the mode, where $r_{n,k}\ge1$.
The reciprocal form of~\eqref{eq:ratio-id} gives, for each $1\le s\le u$,
\[
  1-\frac1{r_{n,m_n-s}}
  =\frac{\delta_n+(s-1)(2m_n-s+1)}{n-m_n+s}.
\]
Bounding the numerator from below by $\delta_n\ge0$ and $2m_n-s+1\ge m_n+2$, and
the denominator from above by $n-m_n+s\le n\le 4m_n^2$, here we use that
$\lfloor\sqrt{n+1}\rfloor \geq \sqrt{n}/2$, for any $n \geq 1$, we obtain
\[
  1-\frac1{r_{n,m_n-s}}\ge\frac{(s-1)(m_n+2)}{4m_n^2}
  \geq \frac{s-1}{4m_n}
  \ge c\,\frac{s-1}{m_n},
  \qquad 1\le s\le u.
\]
Since $r_{n,m_n-s}\ge1$, we have
\[
  \log r_{n,m_n-s}
  =
  -\log(1-(1-{1}/{r_{n,m_n-s}}))
  \ge c\,\frac{s-1}{m_n}.
\]
Therefore
\[
  L_n(-u)
  =-\sum_{s=1}^{u}\log r_{n,m_n-s}
  \le -\frac{c}{m_n}\sum_{s=1}^{u}(s-1)
  =-\frac{c\,u(u-1)}{2m_n}
  \le C-\frac{c\,u^2}{m_n},
\]
where $C$ absorbs the case $u=1$.
\smallskip

Combining the two halves, and using $m_n\le\sqrt{n+1}$, we obtain that
\begin{equation}\label{eq:central-bound}
  \frac{w_{n,m_n+t}}{w_{n,m_n}}\le C\exp\!\left(-c\,\frac{t^2}{\sqrt n}\right),
  \qquad 1-m_n\le t\le m_n .
\end{equation}

Now we study the tail $t>m_n$, that is, $k>2m_n$. For $k\ge 2m_n$, using
$m_n^2\ge n/2$, which follows from $m_n=\lfloor\sqrt{n+1}\rfloor$, we have
\[
  r_{n,k}
  = \frac{w_{n,k+1}}{w_{n,k}}
  = \frac{n-k}{k(k+1)}
  \le\frac{n}{k^2}
  \le\frac{n}{4m_n^2}
  \le\tfrac12 .
\]
Thus the weights decay geometrically beyond $k=2m_n$: for every $k\ge 2m_n$ and
$j\ge0$ with $k+j\le n$,
\begin{equation}\label{eq:geometric-tail}
  w_{n,k+j}\le 2^{-j}\,w_{n,k}.
\end{equation}
Summing~\eqref{eq:geometric-tail} from $k=2m_n$ gives
\[
  \sum_{t>m_n}\frac{w_{n,m_n+t}}{w_{n,m_n}}
  =\sum_{j\ge1}\frac{w_{n,2m_n+j}}{w_{n,m_n}}
  \le\frac{w_{n,2m_n}}{w_{n,m_n}}\sum_{j\ge1}2^{-j}
  =\frac{w_{n,2m_n}}{w_{n,m_n}} .
\]
It remains to bound the right-hand side. By the central estimate at $t=m_n$,
see equation~\eqref{eq:central-bound}, and using that
$m_n^2/\sqrt n\ge\sqrt n/2$, we conclude that
\[
  \frac{w_{n,2m_n}}{w_{n,m_n}}\le Ce^{-c\,{m_n^2}/{\sqrt n}}
  \le C\,e^{-c\sqrt n/2}.
\]
Combining both estimates yields
\[
  \sum_{t>m_n}\frac{w_{n,m_n+t}}{w_{n,m_n}}\le C\,e^{-c\sqrt n/2}\,.
\]
This gives the tail bound with $c_2:=c/2$, and concludes the proof.
\end{proof}

We now prove a central limit theorem for the random variable $K_n$.
\begin{corollary}
	\label{co:ttsize}
	Let $K_n$ denote the number of leaves of $\mN_n$ whose parent is a tree vertex.
	Then
	\[
		\frac{K_n - \sqrt{n}}{n^{1/4}} \convdis \cN(0, 1/2).
	\]
\end{corollary}
\begin{proof}
We realize the rescaled weight sums as Riemann sums and pass to the limit via
Proposition~\ref{prop:riemann}, both of whose hypotheses are already available in the proof of previous results.
Set $b_n(t):=w_{n,m_n+t}/w_{n,m_n}$ and $h_n:=n^{-1/4}$. Equation~\eqref{eq: Kweights} gives
\[
  \Prb{K_n=m_n+t}=\frac{w_{n,m_n+t}}{W_n}=\frac{w_{n,m_n}}{W_n}\,b_n(t),
\]
so that, summing over $t$,
\[
  \sum_{t}b_n(t)=\frac{1}{w_{n,m_n}}\sum_{t}w_{n,m_n+t}=\frac{W_n}{w_{n,m_n}},
\]
and hence the normalising constant is the rescaled total weight,
\[
  A_n:=\frac{W_n}{w_{n,m_n}\,n^{1/4}}
  =n^{-1/4}\,\frac{W_n}{w_{n,m_n}}
  =h_n\sum_{t}b_n(t).
\]
so that
$\Prb{a\le n^{-1/4}(K_n-m_n)\le b}=A_n^{-1}\,h_n\sum_{a\le h_n t\le b}b_n(t)$ for
every $a<b$. The whole proof reduces to computing the limits of these sums. The
pointwise limit $b_n(t_n)\to e^{-x^2}$ (whenever $h_n t_n\to x$)
is Lemma~\ref{le:asymptotic_estimate_L_n(t)}, while the domination
$b_n(t)\le C_1 e^{-c_1(h_n t)^2}$ on the central range, together with the tail
bound $\sum_{t>m_n}b_n(t)\le C_1 e^{-c_2\sqrt n}$, is Lemma~\ref{le:domination}.
\medskip

We first show that $A_n\to\sqrt\pi$. For fixed $R>0$,
Proposition~\ref{prop:riemann} applied to $g_n(t) = b_n(t)\mathbf 1_{[-R,R]}(h_n t)$
gives
\[
  h_n\!\!\sum_{|h_n t|\le R}\!\! b_n(t)
  =\int_{\ndR}b_n(\lfloor x/h_n\rfloor)\,
   \mathbf 1_{[-R,R]}\!\bigl(h_n\lfloor x/h_n\rfloor\bigr)\, \mathrm{dx}
  \xrightarrow[n\to\infty]{}\int_{-R}^{R}e^{-x^2}\, \mathrm{dx} .
\]
The remaining mass is controlled uniformly in $n$: by the central bound of Lemma~\ref{le:domination},
\[
  h_n\!\!\sum_{R<|h_n t|,\ t\le m_n}\!\! b_n(t)
  \le C_1\,h_n\!\!\sum_{|h_n t|>R}\!\! e^{-c_1(h_n t)^2}
  \le 2C_1\!\int_{R-h_n}^{\infty}\!\!e^{-c_1 x^2}\, \mathrm{dx}=:\varepsilon_n(R),
\]
with $\limsup_n\varepsilon_n(R)\le 2C_1\int_{R}^{\infty}e^{-c_1 x^2}\, \mathrm{dx}=:\varepsilon(R)\to0$
as $R\to\infty$, while the tail $t>m_n$ contributes at most
$h_n C_1 e^{-c_2\sqrt n}\to0$. Hence
\[
  \limsup_{n\to\infty}\Bigl|A_n-\int_{-R}^{R}e^{-x^2}\, \mathrm{dx}\Bigr|\le\varepsilon(R),
\]
and letting $R\to\infty$ yields $A_n\to\int_{\mathbb R}e^{-x^2}\, \mathrm{dx}=\sqrt\pi$.
\smallskip

Now fix reals $a<b$. The indicator $\mathbf 1_{[a,b]}$ has compact support, so for
large $n$ the set $\{a\le h_n t\le b\}$ lies entirely in the central range
$t\le m_n$; applying Proposition~\ref{prop:riemann} to
$g_n(t):=b_n(t)\mathbf 1_{\{a\le h_n t\le b\}}$, whose pointwise limit is
$e^{-x^2}\mathbf 1_{[a,b]}(x)$, gives
\[
  h_n\!\!\sum_{a\le h_n t\le b}\!\! b_n(t)
  \xrightarrow[n\to\infty]{}\int_a^b e^{-x^2}\,\mathrm dx .
\]
Dividing by the normalisation,
\[
  \Prb{a\le\frac{K_n-m_n}{n^{1/4}}\le b}
  =\frac{h_n\sum_{a\le h_n t\le b}b_n(t)}{A_n}
  \xrightarrow[n\to\infty]{}\frac{1}{\sqrt\pi}\int_a^b e^{-x^2}\,\mathrm dx .
\]
Hence $n^{-1/4}(K_n-m_n)\convdis\cN(0,1/2)$. Since $m_n-\sqrt n=o(n^{1/4})$,
the same limit holds with $\sqrt n$ in place of $m_n$.
\end{proof}

\begin{corollary}
\label{co:Knmoments}
Let $Z\sim\cN(0,1)$ and let $Z_n:=\sqrt2\,(K_n-\sqrt n)/n^{1/4}$ denote the
standardised variable. Then, for every fixed integer $p\ge0$, we have
\[
  \Exb{Z_n^{p}}
  \xrightarrow[n\to\infty]{}\Exb{Z^{p}} .
\]
Consequently, for every fixed integer $p\ge1$,
\[
  \Exb{K_n^{p}}=n^{p/2}+\frac{p(p-2)}{4}\,n^{(p-1)/2}+o\bigl(n^{(p-1)/2}\bigr).zw
\]
In particular $\Exb{K_n}=\sqrt n-\tfrac14+o(1)$, $\Exb{K_n^2} = n$,
$\Vab{K_n}=\tfrac{\sqrt n}{2}+o(\sqrt n)$, and $K_n/\sqrt n\to1$ in probability.
\end{corollary}

\begin{proof}
Set $Y_n:=(K_n-\sqrt n)/n^{1/4}$, so that $Z_n=\sqrt2\,Y_n$, and fix $p\ge0$. By
Corollary~\ref{co:ttsize}, we have $Y_n\convdis Y$ with $Y\sim\cN(0,\tfrac12)$. To upgrade
this to convergence of moments it suffices to bound $\Exb{|Y_n|^{p+1}}$ uniformly
in $n$.
\smallskip

Set $\tilde Y_n:=(K_n-m_n)/n^{1/4}$, so that
$Y_n=\tilde Y_n+(m_n-\sqrt n)/n^{1/4}=\tilde Y_n+O(n^{-1/4})$. It suffices to bound $\Ex{|\tilde Y_n|^{p+1}}$ uniformly in $n$.
\smallskip

With $b_n(t):=w_{n,m_n+t}/w_{n,m_n}$, 
$\Prb{K_n=m_n+t}=w_{n,m_n}b_n(t)/W_n$, and the normalisation
$W_n\ge c_0 w_{n,m_n}n^{1/4}$ established in the proof of
Corollary~\ref{co:ttsize} (where we proved that
$A_n:=W_n/(w_{n,m_n}n^{1/4})\to\sqrt{\pi}$), we find that 
\[
  \Ex{|\tilde Y_n|^{p+1}}
  =\frac{w_{n,m_n}}{W_n}\sum_{t}\Bigl|\tfrac{t}{n^{1/4}}\Bigr|^{p+1}b_n(t)
  \le\frac{1}{c_0\,n^{1/4}}\sum_{t}\Bigl|\tfrac{t}{n^{1/4}}\Bigr|^{p+1}b_n(t).
\]
On the central range, Lemma~\ref{le:domination} gives
$b_n(t)\le C_1 e^{-c_1(t/n^{1/4})^2}$, and Proposition~\ref{prop:riemann}
applied to $|t/n^{1/4}|^{p+1}e^{-c_1(t/n^{1/4})^2}$, which is dominated by
$C_p e^{-c_1(t/n^{1/4})^2/2}$, yields
\[
  \frac{C_1}{c_0}\,n^{-1/4}\sum_{t}\Bigl|\tfrac{t}{n^{1/4}}\Bigr|^{p+1}
  e^{-c_1(t/n^{1/4})^2}
  \xrightarrow[n\to\infty]{}
  \frac{C_1}{c_0}\int_{\mathbb R}|x|^{p+1}e^{-c_1 x^2}\,\mathrm dx<\infty\,.
\]
On the tail $t>m_n$, the deterministic bound $|t/n^{1/4}|\le n^{3/4}$
(which follows from $K_n\le n$) and $\sum_{t>m_n}b_n(t)\le C_1 e^{-c_2\sqrt n}$
of Lemma~\ref{le:domination} make that part at most
\[
  c_0^{-1}n^{-1/4}n^{3(p+1)/4}\,C_1 e^{-c_2\sqrt n}\to0.
\]
Hence $\sup_n\Ex{|\tilde Y_n|^{p+1}}<\infty$, and therefore
$\sup_n\Exb{|Y_n|^{p+1}}<\infty$; the family $\{|Y_n|^{p}\}_n$ is uniformly
integrable, and convergence in distribution yields
\[
  \Exb{Y_n^{p}}\xrightarrow[n\to\infty]{}
  \frac{1}{\sqrt\pi}\int_{\mathbb R}x^{p}\,e^{-x^2}\,\mathrm dx\,.
\]
In particular $\Exb{Y_n^{2}}\to\tfrac12$ and $\Exb{Y_n}\to0$. Since
$Z_n=\sqrt2\,Y_n$, this gives
$\Exb{Z_n^{p}}=2^{p/2}\Exb{Y_n^{p}}\to\Exb{Z^{p}}$, the moments of the standard
normal.
\smallskip

We sharpen the first moment using the exact identity $\Exb{K_n^{2}}=n$, which
follows from the equality $k(k+1)w_{n,k+1}=(n-k)w_{n,k}$, or equivalently
$r_{n,k}=w_{n,k+1}/w_{n,k}=(n-k)/(k(k+1))$. 
\smallskip

Summing over $k$ and reindexing gives
$\Exb{K_n(K_n-1)}=n-\Exb{K_n}$, that is, $\Exb{K_n^2}=n$. Since
$K_n^{2}=n+2\sqrt n\,n^{1/4}Y_n+\sqrt n\,Y_n^{2}$, taking expectations gives
$n=n+2n^{3/4}\Exb{Y_n}+\sqrt n\,\Exb{Y_n^{2}}$, that is
\[
  n^{1/4}\Exb{Y_n}=-\tfrac12\Exb{Y_n^{2}}\xrightarrow[n\to\infty]{}-\tfrac14 .
\]
\smallskip

Finally, expanding $K_n^{p}=(\sqrt n+n^{1/4}Y_n)^{p}=\sum_{i=0}^{p}\binom{p}{i}
n^{(2p-i)/4}Y_n^{i}$ and taking expectations,
\[
  \Exb{K_n^{p}}=\sum_{i=0}^{p}\binom{p}{i}\,n^{(2p-i)/4}\,\Exb{Y_n^{i}}\,.
\]
The term $i=0$ contributes $n^{p/2}$; the terms $i=1,2$ are both of order
$n^{(p-1)/2}$, with combined coefficient
\[
  \binom{p}{1}\bigl(n^{1/4}\Exb{Y_n}\bigr)+\binom{p}{2}\Exb{Y_n^{2}}
  \xrightarrow[n\to\infty]{}
  -\frac{p}{4}+\frac{p(p-1)}{4}=\frac{p(p-2)}{4}\,.
\]
For $3\le i\le p$, the uniform $(p+1)$-moment bound above implies
$\sup_n\Exb{|Y_n|^i}<\infty$. Hence the corresponding terms are at most of order
\[
  n^{(2p-3)/4}=o\bigl(n^{(p-1)/2}\bigr).
\]
This gives the stated expansion, and the case $p=1$ yields
$\Exb{K_n}=\sqrt n-\tfrac14+o(1)$. The variance then follows from
$\Exb{K_n^2}=n$:
\[
  \Vab{K_n}=\Exb{K_n^2}-\Exb{K_n}^2
  =n-\bigl(\sqrt n-\tfrac14+o(1)\bigr)^2=\tfrac{\sqrt n}{2}+o(\sqrt n).
\]
Finally, since $\Exb{K_n}=\sqrt n+o(\sqrt n)$ and
$\Vab{K_n}=\tfrac{\sqrt n}{2}+o(\sqrt n)=o(n)$, Chebyshev's inequality gives, for
every $\varepsilon>0$,
\[
  \Prb{\Bigl|\frac{K_n}{\sqrt n}-1\Bigr|\ge\varepsilon}
  \le\frac{\Vab{K_n}+(\Exb{K_n}-\sqrt n)^2}{\varepsilon^2 n}
  \xrightarrow[n\to\infty]{}0,
\]
so $K_n/\sqrt n\to1$ in probability.
\end{proof}

\begin{corollary}
\label{co:tttail}
There exist constants $C,c>0$ such that for all integers $n\ge1$ and all real
$x$ with $0\le x\le n^{1/4}$,
\[
  \Prb{|K_n-\sqrt n|\ge x n^{1/4}}\le C e^{-c x^2}.
\]
\end{corollary}

\begin{proof}
It suffices to treat $x\ge2$: for $0\le x<2$ one has $Ce^{-cx^2}\ge Ce^{-4c}\ge1$
as soon as $C\ge e^{4c}$, and the bound holds trivially since the left-hand side
is a probability. Using that $|m_n-\sqrt n|<1$, for $x\ge2$ we have
\begin{align}\label{eq:proof_tail_K_n}
  \{|K_n-\sqrt n|\ge xn^{1/4}\}\subseteq\{|K_n-m_n|\ge\tfrac12 xn^{1/4}\}.
\end{align}
Indeed, since $|m_n-\sqrt n|<1$, the triangle inequality gives
$|K_n-m_n|>xn^{1/4}-1\ge\tfrac12 xn^{1/4}$
on that event. Here we use that $x \geq 2$ and $xn^{1/4}\ge2$. Thus it suffices to bound the right-hand event.
\smallskip

Writing $v:= xn^{1/4}/2$ and using that
$\Prb{K_n=m_n+t}=w_{n,m_n}b_n(t)/W_n$ with $b_n(t):=w_{n,m_n+t}/w_{n,m_n}$,
\[
  \Prb{|K_n-m_n|\ge v}
  =\frac{w_{n,m_n}}{W_n}
  \sum_{\substack{|t|\ge v\\ 1-m_n\le t\le n-m_n}}b_n(t).
\]
By Lemma~\ref{le:domination}, on the central range $b_n(t)\le C_1 e^{-c_1
t^2/\sqrt n}$. Since we are in the case $x\ge2$ and $x\le n^{1/4}$, we have
$n^{1/4}\ge2$, and therefore $v=xn^{1/4}/2\ge2$. By monotonicity,
$\sum_{t\ge v}f(t)\le\int_{v-1}^\infty f(s)\,ds$ for decreasing $f\ge0$; since
$v-1\ge v/2$, applying this to $f(t)=e^{-c_1(t/n^{1/4})^2}$ and using the
Gaussian tail estimate
\[
  \int_x^\infty e^{-cu^2}\,\mathrm du
  \le\frac{1}{2cx}\,e^{-cx^2}, \qquad x>0,
\]
which follows from $u/x\ge1$, we obtain
\begin{align*}
\sum_{v\le t\le m_n}e^{-c_1(t/n^{1/4})^2}
&\le \int_{v/2}^\infty e^{-c_1(s/n^{1/4})^2}\,\mathrm ds\\
&= n^{1/4}\int_{x/4}^\infty e^{-c_1u^2}\,\mathrm du\\
&\le \frac{n^{1/4}}{2c_1(x/4)}e^{-c_1x^2/16}.
\end{align*}
The same estimate, with a possibly different constant, applies to the
negative range \(1-m_n\le t\le -v\). It remains to consider the range
\(t>m_n\). By Lemma~\ref{le:domination},
\[
\sum_{t>m_n} b_n(t)\le C_2 e^{-c_2\sqrt n}.
\]
Since \(x\le n^{1/4}\), we have \(x^2\le\sqrt n\), and therefore
\[
e^{-c_2\sqrt n}\le e^{-c_2x^2}.
\]
Combining the positive and negative central ranges with this remaining
tail, and absorbing the factor \(1/x\) into the constant since \(x\ge2\),
we obtain
\[
\sum_{\substack{|t|\ge v\\ 1-m_n\le t\le n-m_n}}b_n(t)
\le C_3\,n^{1/4}e^{-c_3x^2}.
\]
Finally, since $A_n:=W_n/(w_{n,m_n}n^{1/4})\to\sqrt\pi$, see the proof of
Corollary~\ref{co:ttsize}, the sequence $A_n$ is bounded below: there is $c_0>0$
with $W_n\ge c_0 w_{n,m_n}n^{1/4}$, hence
\begin{align*}
  \Prb{|K_n-\sqrt n|\ge x n^{1/4}} &\leq \Prb{|K_n-m_n|\ge x n^{1/4}/2} \\
  &= \frac{w_{n,m_n}}{W_n} \sum_{\substack{|t|\ge v\\ 1-m_n\le t\le n-m_n}}b_n(t) \\
  &\le\frac{C_3\,n^{1/4}e^{-c_3 x^2}}{c_0\,n^{1/4}}
  =\frac{C_3}{c_0}\,e^{-c_3 x^2}.
\end{align*}
This concludes the proof.
\end{proof}

\subsection{Tail bounds for the height}
\label{sec:tail}

For $n= k+j$ with  $k \ge 1$ and $j \ge 0$ we let $\cT_{k,j}$ denote the set of rooted plane trees whose leaves are labelled by $1, \ldots k$ and whose numbers $N_i$ of vertices with outdegree $i\ge 0$ are given by
\begin{align}
	\label{eq:degreestat1}
	N_0 = k,
	\qquad
	N_1 = 2j,
	\qquad
	N_2 = k-1,
	\qquad  \text{ and } \quad
	N_i = 0 \text{ for } i \ge 3.
\end{align}

\begin{proposition}
	\label{pro:correspondence}
	There is a bijection between $\cT_{k,j}$ and triples $(T,Y, O)$  with $T \in \cT_k$, $O \in \{+,-\}^{k-1}$ and
	$
	Y=(y_i)_{1\le i\le 2k-1}\in \ndN_0^{2k-1}
	$
	satisfying $\sum_{i=1}^{2k-1} y_i = 2j$.
\end{proposition}
\begin{proof}
	Let $(T,Y,O)$ be given. Note that $T$ consists of $k-1$ binary vertices, $k$~leaves, and a root vertex with outdegree $1$. Since the leaves of $T$ are labelled, we may canonically list the vertices and edges of $T$.
    \smallskip
    
	This allows us to canonically map the $k-1$ coordinates of $O$ to the $k-1$ binary vertices of $T$, and interpret $O$ as a choice of left-right ordering of the two children of each of the binary vertices of $T$. Likewise, we may canonically map the $2k-1$ coordinates of $Y$ to the $2k-1$ edges of $T$, allowing us to replace the $i$th edge of $T$ by a path of length $y_i +1$ for each $1\le i \le 2k-1$. This creates $\sum_{i=1}^{2k-1} y_i = 2j$ new unary vertices.
	\smallskip
    
	Thus, after ordering $T$ according to $O$, subdividing its edges according to $Y$, and deleting the root of $T$, we obtain a tree from $\cT_{k,j}$.
	\smallskip
    
Conversely, first add a planted root above the root of the plane tree.
Then contract each maximal path of unary vertices lying on an edge of the
resulting planted tree, including the path starting at the planted root.
The plane order of the two children at each binary vertex recovers $O$;
after recording it, forgetting the plane order yields the underlying tree
$T \in \cT_k$, while the lengths of the contracted unary paths recover the
vector $Y$.
\end{proof}

For integers $1\le k\le n$, let $\tau_{n,k}$ denote a rooted plane tree sampled uniformly from the rooted plane trees with outdegree statistics given by~\eqref{eq:degreestat1} for $j=n-k$. We verify that the top tree component of $\mN_n$ is a mixture of $(\tau_{n,k})_{1 \le k \le n}$.

\begin{lemma}
	\label{le:conditioned-tree}
	Fix an integer $1 \le k \le n$. Conditionally on the event $K_n = k$, take the top tree component of $\mN_n$, delete the root vertex, equip the two children of each binary vertex with an independent uniform left-right ordering, and then forget the leaf labels. The resulting rooted plane tree $\tau(\mN_n \mid K_n =k)$ is distributed like $\tau_{n,k}$.
\end{lemma}

\begin{proof}
	Conditionally on the event $K_n=k$, the network $\mN_n$ is obtained from a uniform element of $\cO\cN_{k,j}$ by a uniform relabelling of its leaves. Since leaf relabelling does not affect the top tree once the labels are forgotten, it suffices to work with a uniform random network $\mN_{k,j}$ in $\cO\cN_{k,j}$.
	\smallskip
    
	By Proposition~\ref{pro:tt}, any network $N$ from $\cO\cN_{k,j}$  corresponds  to a triple $(T,Y,M)$, where $T\in\cT_k$, 
	\[
		Y=(y_i)_{1\le i\le 2k-1}\in \ndN_0^{2k-1}, \qquad \sum_{i=1}^{2k-1} y_i = 2j,
	\]
	and  $M$ is an ordered partition of $[2j]$ into $j$ two-element blocks. The top tree component of $N$ depends only on $(T,Y)$. It is obtained from $T$ by ordering its edges in a canonical way and replacing the $i$th edge by a path of length $y_i + 1$.
	\smallskip

	Hence the pair $(T(\mN_{k,j}),Y(\mN_{k,j}))$ corresponding to the random network $\mN_{k,j}$ is uniform on $\cT_{k} \times \{(y_i)_{1 \le i \le 2k-1} \in \ndN_0^{2k-1} \mid \sum_{i=1}^{2k-1} y_i = 2j\}$. Hence by the bijection in Proposition~\ref{pro:correspondence} it follows that if we uniformly select one of the $2^{k-1}$ left-right orderings of the top tree component of $\mN_{k,j}$ and delete its root vertex, then the result $\tau(\mN_{k,j})$ is uniformly distributed on~$\cT_{k,j}$.
	\smallskip
    
	Finally, every underlying unlabelled plane tree of an element of $\cT_{k,j}$ has exactly $k!$ leaf-labellings by $1, \dots, k$. Therefore, after forgetting the leaf labels of $\tau(\mN_{k,j})$, the distribution remains uniform on the class of unlabelled plane trees corresponding to~$\cT_{k,j}$, which is precisely the law of $\tau_{n,k}$.
\end{proof}

\begin{lemma}
	\label{le:heightk}
	For all $n \ge 2$ and all integers $1\le k\le n$ and all $x> 0$,
	\[
		\Prb{\He(\tau_{n,k}) > x n^{3/4}} \le 5\exp\left(-\frac{k x^2}{2^{14} \sqrt{n} } \right).
	\]
\end{lemma}
\begin{proof}
	Fix $1\le k\le n$ and write $j=n-k$. By definition, $\tau_{n,k}$ is uniform over the rooted plane trees with outdegree statistics given in~\eqref{eq:degreestat1}.
	\smallskip
    
	Choose a uniform random labelling of its $2n-1$ vertices by $1, \ldots, 2n-1$ and then forget the plane order. Conditional on the resulting outdegree sequence $d=(d_1,\dots,d_{2n-1})$ (with the order determined by the labels of the vertices), this labelled rooted tree is uniform on the class $T_d$ of $(2n-1)$-vertex labelled trees where the vertex with label $i$ is required to have exactly $d_i$ children for all $1 \le i \le 2n-1$. Indeed, every rooted labelled tree in $T_d$ has exactly
	\[
		\prod_{i=1}^{2n-1} d_i! = 2^{k-1}
	\]
	compatible plane orderings, and this number depends only on $d$.
	\smallskip
    
	Since every such degree sequence $d$ has exactly $2j$ entries equal to $1$, it follows by~\cite[Thm. 6]{zbMATH07940224} that for every $y>0$,
	\begin{align*}
		\Prb{\He(\tau_{n,k})> y(2n-1)^{1/2} \,\,\Big|\,\, d} \le 5\exp\left(-\frac{n-j}{n}\frac{y^2}{2^{13}}\right).
	\end{align*}
	The bound does not depend on $d$, hence the same inequality holds unconditionally. Setting $y = x n^{3/4} / (2n-1)^{1/2}$ we obtain
	\[
		\Prb{\He(\tau_{n,k}) > x n^{3/4}} \le 5\exp\left(-\frac{k x^2}{\sqrt{n} 2^{13}} \frac{1}{2-1/n}\right) \le 5\exp\left(-\frac{k x^2}{2^{14} \sqrt{n}} \right),
		\]
	as claimed.
\end{proof}

\begin{proof}[Proof of Theorem~\ref{te:bound}]
	The inequality
	\[
		\Prb{\He^-(\mN_n) \ge x n^{3/4}} \le \Prb{\He^+(\mN_n) \ge x n^{3/4}}
	\]
	is immediate from the definitions. It therefore remains to bound the upper tail of~$\He^+(\mN_n)$.
	\smallskip
    
	Let $\tau_n = \tau(\mN_n)$. Every leaf of $\tau_n$ is also a leaf of $\mN_n$. If $v$ is a reticulation leaf of $\mN_n$ and $u,w$ are its two parents in $\tau_n$, then $u$ and $w$ are not leaves of $\tau_n$, hence have depth at most $\He(\tau_n)-1$. Consequently,
	\[
	\he^+_{\mN_n}(v)=\max\{ \he_{\tau_n}(u), \he_{\tau_n}(w)\} + 3 \le \He(\tau_n)+2.
	\]
	Taking the maximum over all leaves $v$ of $\mN_n$ yields
	\[
	\He^+(\mN_n)\le \He(\tau_n)+2.
	\]
	
	Fix $x \ge 0$. If $x > 2n^{1/4}+1$, then $x n^{3/4}>2n+1$, whereas every simplex network with $n$ leaves has height at most $2n+1$. Hence in this case $\Prb{\He^+(\mN_n) \ge x n^{3/4}} = 0$. We may therefore assume that \[
	x\le 2n^{1/4}+1.
	\]
	
	For \(x\le4\), the desired estimate follows after increasing the constant
\(C\), since probabilities are at most \(1\). Assume now that \(x>4\),
and let
\[
E_n:=\{K_n\ge\tfrac12\sqrt n\}.
\]
On the event \(\{\He^+(\mN_n)\ge xn^{3/4}\}\), the inequality
\(\He^+(\mN_n)\le\He(\tau_n)+2\) gives
\[
\He(\tau_n)\ge xn^{3/4}-2.
\]
Since \(x>4\) and \(n\ge1\), we have
\[
xn^{3/4}-2>\frac12xn^{3/4}.
\]
Consequently,
\[
\He(\tau_n)>\frac12xn^{3/4}.
\]
Furthermore, \((\tau_n\mid K_n=k)\eqdist\tau_{n,k}\). Therefore
	\[
	\Prb{\He^+(\mN_n)\ge x n^{3/4}}
	\le
	\Prb{E_n^c}
	+
	\sum_{k\ge \sqrt n/2}\Prb{K_n=k}\Prb{\He(\tau_{n,k})> \tfrac12 x n^{3/4}}.
	\]
	By Corollary~\ref{co:tttail}, applied with $x= n^{1/4}/2$,
	\[
	\Prb{E_n^c}
	=
	\Prb{K_n<\tfrac12\sqrt n}
	\le C e^{-c n^{1/2}}.
	\]
Since $n^{1/4}\ge 1$, the inequality $x\le 2n^{1/4}+1$ implies $x\le 3n^{1/4}$, and therefore $n^{1/2}\ge x^2/9$. After decreasing $c$ if necessary we may thus bound this by $C e^{-c x^2}$. Furthermore, for every $k\ge \sqrt n/2$, Lemma~\ref{le:heightk} yields
	\[
	\Prb{\He(\tau_{n,k})> \tfrac12 x n^{3/4}}
	\le 5\exp\!\left(-\frac{k}{\sqrt{n}}\frac{x^2}{2^{16}}\right)
	\le 5\exp\!\left(-\frac{x^2}{2^{17}}\right)\,.
	\]
    
	Combining the last three displays and absorbing constants, we obtain constants $C,c>0$ such that for all $n\ge 1$ and all $x\ge 0$,
	\[
	\Prb{\He^+(\mN_n)\ge x n^{3/4}}\le C e^{-c x^2}.
	\]
	This finishes the proof.
\end{proof}
\subsection{Convergence of the rescaled height}

\label{sec:scalheight}

Recall that we write $\he_{\tau(N)}(v)$ for the height of a vertex $v$ in the tree $\tau(N)$. Every leaf of $\tau(N)$ is a leaf of $N$. Hence a highest leaf of $\tau(N)$ has both shorter and longer network height equal to its depth in $\tau(N)$ plus the edge from the planted root. If $v$ is a reticulation leaf of $N$ and $u,w$ are its two parents. Consequently $\he_{\tau(N)}(u), \he_{\tau(N)}(w) \le \He(\tau(N))- 1$, and the two directed paths from the planted root of $N$ to $v$ have lengths $\he_{\tau(N)}(u)+3$ and $\he_{\tau(N)}(w) + 3$. It follows that, for every simplex network $N$,
\begin{align}
	\label{eq:heightcompare}
	\He(\tau(N)) \le \He^-(N) \le \He^+(N) \le \He(\tau(N)) + 2.
\end{align}

We shall use the following binomial tail estimate. By Bernstein's inequality~\cite[Eq. (2.10)]{zbMATH06586491} there is a constant $c>0$ such that for all $s \ge 0$ and $0 \le p \le 1$
\begin{equation}
	\label{eq:bbound}
	\Prb{|\mathrm{Bin}(n,p) - n p|\ge s} \le 2 \exp\left(-c \frac{s^2}{np + s}\right).
\end{equation}

\begin{lemma}
	\label{le:blowup}
	Let $k = k_n$ satisfy $k / \sqrt n \to 1$ and set $\ell = 2k-1$ and $q=2(n-k)$. Let $Y = (Y_1,\ldots,Y_\ell)$ be uniformly distributed over
	\[
		\left\{(y_1,\ldots,y_\ell) \in \ndN_0^\ell \,\Bigg|\, \sum_{i=1}^\ell y_i=q\right\}.
	\]
	Let $g_n$ denote a sequence of positive integers satisfying $\log g_n = o(n^{1/4})$. Then, for every $R>0$ and every $\epsilon>0$,
	\[
		\sup_{\mathcal A} \Prb{\max_{A \in \mathcal A} \left|\sum_{i\in A}(Y_i+1) - |A| \frac{2n-1}{2k-1}\right| > \epsilon \frac{n}{\sqrt{k}}} \to 0,
	\]
	where the supremum is over all collections $\mathcal A$ of at most $g_n$ subsets of
	$[\ell]$, each of cardinality at most $R \sqrt{k}$.
\end{lemma}
\begin{proof}
	We prove a bound for one fixed subset and then take a union bound. Let $A \subset [\ell]$ have cardinality $a \le R \sqrt{k}$, and set $S_A := \sum_{i\in A} Y_i$. The case $a=0$ is trivial. Since $\ell = 2k - 1$ and $a \le R \sqrt{k}$, we have $1 \le a < \ell$ for all sufficiently large $k$. Let
	\[
		\Delta_\ell := \left\{(p_1,\ldots,p_\ell) \in [0,1]^\ell \,\Bigg|\, \sum_{i=1}^\ell p_i = 1 \right\}.
	\]
	Let $(P_1, \ldots, P_\ell)$ have the Dirichlet distribution with all parameters equal to one, whose density on $\Delta_\ell$ is $(\ell-1)!$, and, conditionally on this vector, let $(Y_1, \ldots, Y_\ell)$ have multinomial distribution with $q$ trials and cell probabilities $(P_1, \ldots, P_\ell)$. For every $(y_1, \ldots, y_\ell)\in\ndN_0^\ell$ with sum $q$ we  have
	\begin{align*}
		\Prb{Y_1 = y_1, \ldots, Y_\ell = y_\ell} 
		&= \frac{q!(\ell-1)!}{\prod_{i=1}^\ell y_i!} \int_{\Delta_\ell}\prod_{i=1}^\ell p_i^{y_i}\,\mathrm{d}(p_1, \ldots, p_\ell) \\ 
		&= \frac{q!\,(\ell-1)!}{\prod_{i=1}^\ell y_i!} \int_{\Delta_{\ell-1}} \prod_{i=1}^{\ell-1} p_i^{y_i} \Big(1 - \sum_{i=1}^{\ell-1} p_i\Big)^{y_\ell}
		\,\mathrm{d}p_1\cdots\mathrm{d}p_{\ell-1} \\
		&=  \frac{q!\,(\ell-1)!}{(q+\ell-1)!},
\end{align*}
where we have used $p_\ell = 1 - \sum_{i=1}^{\ell-1} p_i$ to reduce the integral to
the $\ell-1$ free variables. The last equality is the evaluation of the multivariate Dirichlet integral, see~\cite{zbMATH01574099}. This value is independent of $(y_1, \ldots, y_\ell)$, and therefore the resulting law is uniform over the compositions of $q$ into $\ell$ non-negative integer parts. With $P_A := \sum_{i \in A} P_i$, the variable $P_A$ has distribution $\mathrm{Beta}(a, \ell - a)$, and conditionally on $P_A$ the random variable $S_A$ has distribution $\mathrm{Bin}(q, P_A)$.
	
	The $a$-th smallest value of $\ell - 1$ independent uniform random variables on $[0,1]$ has distribution $\mathrm{Beta}(a, \ell - a)$. That is, it is distributed like $P_A$. Hence, for $0<p<1$,
	\[
		\Prb{P_A \ge p} = \Prb{\mathrm{Bin}(\ell - 1, p) \le a-1}, \qquad \Prb{P_A \le p} = \Prb{\mathrm{Bin}(\ell - 1, p) \ge a}.
	\]
	By~\eqref{eq:bbound} it follows that there exists $c > 0$ such that, for all $\delta>0$,
	\begin{align}
		\label{eq:bt}
		\Prb{|P_A-a/\ell| > \delta} \le 4 \exp\left(-c\frac{\ell^2 \delta^2}{a + \ell\delta}\right).
	\end{align}
	Indeed, if $p = a/\ell + \delta < 1$, then by $a \le \ell$ we have
	\begin{align*}
		(\ell - 1)p - (a-1) = (\ell - 1)(a/\ell + \delta) - (a-1) \ge (\ell - 1)\delta,
	\end{align*}
	and
	\begin{align*}
		(\ell-1)p + (\ell-1)\delta \le \ell p + \ell \delta = a + 2\ell \delta \le 2(a + \ell \delta).
	\end{align*}
	Thus, by \eqref{eq:bbound} for some constant $c'>0$
	\begin{align*}
		\Prb{P_A > a/\ell + \delta} 	&\le \Prb{\mathrm{Bin}(\ell - 1, p) \le a-1} \\
									&\le \Prb{ \mathrm{Bin}(\ell - 1, p) - (\ell -1)p \le -(\ell-1) \delta} \\
									&\le  2\exp\left(-c'\frac{(\ell- 1)^2 \delta^2}{2(a + \ell\delta)}\right)
	\end{align*}
	By analogous arguments we also obtain that there exists a constant $c''>0$ such that
	\[
		\Prb{P_A < a/\ell - \delta} \le 2 \exp\left(-c'' \frac{(\ell- 1)^2 \delta^2}{a + \ell\delta}\right)
	\]
	We may replace $c'$ and $c''$ by lower constants and combine the bounds for the lower and upper tail to obtain~\eqref{eq:bt}.
	\smallskip
    
	Set $t := \epsilon n / \sqrt{k}$. Note that 
	\[
		S_A - qa/\ell =  q(P_A - a/\ell) + (S_A - q P_A) .
	\]
	Using $(S_A \mid P_A) \eqdist (\mathrm{Bin}(q, P_A) \mid P_A)$ it follows from~\eqref{eq:bt} and~\eqref{eq:bbound} that there exists a constant $c_1>0$ with
	\begin{align*}
		\Prb{|S_A - qa/\ell|>t} &\le \Prb{|P_A - a/\ell| > t/(2q)} + \sup_{0\le p \le 1} \Prb{|\mathrm{Bin}(q,p) - qp| > t/2}  \\ 
		&\le 4\exp\left(-c_1 \frac{\ell^2(t/q)^2}{a+\ell t/q}\right) + 2\exp\left(-c_1 \frac{t^2}{q+t}\right).
	\end{align*}
	Since $k / \sqrt{n} \to 1$, for all sufficiently large $n$ we have $k \le 2 \sqrt{n}$, $\ell \ge k$, $\ell \le 2k$, $n \le q \le 2n$, and $n/k \ge k/2$. Hence
	\[
		\frac{\epsilon}{2 \sqrt{k}} \le \frac{t}{q} \le \frac{\epsilon}{\sqrt{k}}
	\]
	Using that $x \mapsto x/(a+x)$ is increasing and $a \le R \sqrt{k}$ it follows that
	\begin{align*}
		\frac{\ell^2 (t/q)^2}{a + \ell t/q} \ge \frac{k^2 \epsilon^2/(4k)}{R \sqrt{k} + k \epsilon/(2\sqrt{k})} = \frac{\epsilon^2}{4R + 2\epsilon} \sqrt{k}.
	\end{align*}
	Moreover,
	\begin{align*}
		\frac{t^2}{q + t} \ge \frac{\epsilon^2 n^2 / k}{2n + \epsilon n/\sqrt{k}} \ge \frac{\epsilon^2}{2 + \epsilon} \frac{n}{k} \ge \frac{\epsilon^2}{4 + 2\epsilon} k.
	\end{align*}
	Consequently, there exist constants $c_2 = c_2(R,\epsilon) > 0$ and $c_3 = c_3(\epsilon) > 0$ such that, uniformly over all $A \subset [\ell]$ with $|A| \le R \sqrt{k}$,
	\[
		\Prb{|S_A - qa/\ell| > t} \le  \exp(-c_2 \sqrt{k}) + \exp(-c_3 k)
	\]
	for all sufficiently large $n$. Note that
	\begin{align*}
		\sum_{i \in A}(Y_i + 1) - |A| \frac{2n - 1}{2k - 1} = S_A + a - \frac{q + \ell}{\ell}a = S_A - \frac{q}{\ell}a.
	\end{align*}
	Hence, for every collection $\mathcal{A}$ of at most $g_n$ subsets of $[\ell]$, each of cardinality at most $R \sqrt{k}$, we obtain
	\begin{align*}
		\Prb{\max_{A \in \mathcal A} \left|\sum_{i \in A}(Y_i + 1) - |A| \frac{2n - 1}{2k - 1}\right| > \epsilon \frac{n}{\sqrt{k}}} \le  g_n \left( \exp(-c_2 \sqrt{k}) +  \exp(-c_3 k)\right).
	\end{align*}
	Since $k / \sqrt{n} \to 1$ and $g_n = \exp(o(n^{1/4}))$ we have
	\[
		g_n \left( \exp(-c_2 \sqrt{k}) +  \exp(-c_3 k)\right) \to 0.
	\]
	The bound is independent of the choice of $\mathcal A$, and hence the convergence holds uniformly over all admissible collections $\mathcal A$.
\end{proof}

\begin{proposition}
	Let $(k_n)_{n\ge1}$ be any deterministic sequence with $1\le k_n\le n$ and
	$k_n/\sqrt n\to1$. Then
	\[
		2^{-3/2}n^{-3/4}\He(\tau_{n,k_n}) \convdis \sup_{0\le t\le1} e(t)
	\]
	where $e$ is a Brownian excursion of duration $1$.
\end{proposition}
\begin{proof}
	Write $k = k_n$ and $q = 2j$. By Proposition~\ref{pro:correspondence}, $\tau_{n,k}$ is obtained from a tree $T_k \in \cT_k$ by giving the binary vertices a uniform random left-right order of their two children, subdividing the $\ell$ edges of $T_k$ according to a uniform composition $Y=(Y_e)_{e\in E(T_k)}$ of $q$ into $\ell$ parts, and deleting the planted root. Here $E(T_k)$ denotes the set of edges of $T_k$. The tree $T^*_k$ obtained from $T_k$ by deleting the planted root is a uniform rooted plane tree with outdegree statistics $N_0 = k$, $N_2 = k-1$, and $N_i=0$ for $i \notin \{0,2\}$.
    \smallskip
    
	Hence $T^*_k$ has $\ell$ vertices. Let $h^*(i)$, $1 \le i \le \ell$ denote the heights of the vertices of $T^*_k$ in depth-first-search order. We extend $h^*$ to a continuous function on $[0, \ell]$ by fixing $h^*(0)=0$ with linear interpolation between values at integer points. By~\cite[Thm.~3]{zbMATH06294776} we have
	\[
		\frac{1}{2 \sqrt{\ell}} (h^*(t\ell) : 0 \le t \le 1) \convdis (e(t), 0 \le t \le 1), \qquad n \to \infty.
	\]
	Heights between $T^*_k$ and $T_k$ differ exactly by $1$. It follows that
	\begin{align}
		\label{eq:h1}
		\He(T_k) / (2 \sqrt{\ell}) \convdis \sup_{0 \le t \le 1} e(t), \qquad n \to \infty.
	\end{align}
	For each leaf $v$ of $T_k$ let $P(v) \subset E(T_k)$ denote the set of edges from the planted root to $v$. Hence the height of the corresponding leaf in $\tau_{n,k}$ is given by
	\[
		\sum_{e\in P(v)} (Y_e + 1) - 1,
	\]
	the subtraction of one accounting for the deletion of the planted root. Therefore
	\begin{align}
		\label{eq:h2}
		\He(\tau_{n,k}) + 1 = \max_{v\in L(T_k)} \sum_{e\in P(v)}(Y_e + 1)
	\end{align}
	with $L(T_k)$ denoting the set of leaves of $T_k$.
	
	For each $\epsilon_1 > 0$ it follows by~\eqref{eq:h1} that we may select $R$ large enough so that \[
		\Prb{\He(T_k) > R \sqrt{k}} \le \epsilon_1.
	\]
	On the event $\He(T_k) \le R \sqrt{k}$, the collection $\{P(v) \mid v \in L(T_k)\}$ has at most $k$ elements and $|P(v)| \le R \sqrt{k}$ for each $v \in L(T_k)$.
	By Lemma~\ref{le:blowup} it follows that for each $\epsilon > 0$
	\begin{align*}
		\Prb{\He(T_k) \le R \sqrt{k}, \max_{v\in L(T_k)} \left|\sum_{e\in P(v)}(Y_e+1) - |P(v)|\frac{2n-1}{2k-1}\right| > \epsilon \frac{n}{\sqrt k}} \to 0.
	\end{align*}
	Since $\epsilon_1>0$ may be chosen arbitrarily small, it follows that
	\begin{align}
		\label{eq:h3}
		\Prb{\max_{v\in L(T_k)} \left|\sum_{e\in P(v)}(Y_e+1) - |P(v)|\frac{2n-1}{2k-1}\right| > \epsilon \frac{n}{\sqrt k}} \to 0.
	\end{align}
	Combining~\eqref{eq:h1},~\eqref{eq:h2} and~\eqref{eq:h3} it follows that
	\begin{align*}
		\frac{2k-1}{2n-1}\frac{1}{2 \sqrt{\ell}} \He(\tau_{n,k})  \convdis \sup_{0 \le t \le 1} e(t).
	\end{align*}
	Since $k \sim \sqrt{n}$ and $\ell \sim 2k$ it follows that
	\begin{align*}
		\frac{1 }{n^{3/4} 2^{3/2}} \He(\tau_{n,k})  \convdis \sup_{0 \le t \le 1} e(t).
	\end{align*}
\end{proof}

\begin{proof}[Proof of Theorem~\ref{te:main2}]
	Let $k_n$ satisfy $1 \le k_n \le n$ and $k_n/\sqrt{n} \to 1$. We have $\tau_{n,k_n} \eqdist \tau(\mN_{k_n,j})$ for $j = n - k_n$. By~\eqref{eq:heightcompare} it follows that jointly
	\begin{align*}
		2^{-3/2} n^{-3/4} (\He^+(\mN_{k_n,j}),\He^-(\mN_{k_n,j})) \convdis \left(\sup_{0 \le t \le 1} e(t), \sup_{0 \le t \le 1} e(t)\right).
	\end{align*}
	For all $1 \le k \le n$ we have that conditionally on $K_n=k$, the random network $\mN_n$ is distributed like $\mN_{k,n-k}$. Hence by Corollary~\ref{co:ttsize} it follows that
	\begin{align*}
		2^{-3/2} n^{-3/4} (\He^+(\mN_{n}),\He^-(\mN_{n})) \convdis \left(\sup_{0 \le t \le 1} e(t), \sup_{0 \le t \le 1} e(t)\right).
	\end{align*}
    This concludes the proof. 
\end{proof}

\subsection{Convergence of the rescaled height profile}
\label{sec:profile}

Recall the maps
\[
	\phi_+(x,y):=\max(x,y) \qquad \text{and} \qquad \phi_-(x,y):=\min(x,y).
\]
Recall that $\mu_e$ denotes the occupation measure of the Brownian excursion $e$, and that
\[
	\Pi_e^\pm = (\mu_e \otimes \mu_e) \circ \phi_\pm^{-1}.
\]
For $1 \le k \le n$ define
\begin{align*}
	\gamma_{n,k}:=\frac{\sqrt{2k-1}}{2(2n-1)}.
\end{align*}
Note that if $k \sim \sqrt{n}$ we have 
\begin{align}
	\label{eq:gammascale}
	\gamma_{n,k} \sim  2^{-3/2}n^{-3/4}.
\end{align}

\begin{proposition}\label{pro:top-tree-profile}
	Let  $k=k_n$ satisfy $1 \le k \le n$ and $k/\sqrt{n} \to 1$. Let $U(\tau_{n,k})$ denote the set of unary vertices of $\tau_{n,k}$, and $V(\tau_{n,k})$ the vertex set of $\tau_{n,k}$. Set
	\[
		\mu_n := \frac{1}{2n-1}\sum_{v\in V(\tau_{n,k})} \delta_{\gamma_{n,k}\he_{\tau_{n,k}}(v)}, \qquad
		\nu_n :=\frac{1}{2(n-k)}\sum_{u\in U(\tau_{n,k})} \delta_{\gamma_{n,k}\he_{\tau_{n,k}}(u)}.
	\]
	Then, jointly for weak convergence of probability measures on $[0,\infty)$, we have
	\[
		(\mu_n,\nu_n) \convdis (\mu_e,\mu_e).
	\]
\end{proposition}
\begin{proof}
	Put
	\[
		\ell_n:=2k-1, \qquad q_n:=2(n-k), \qquad s_n:=\ell_n+q_n=2n-1.
	\]
	For a rooted plane tree $\tau$ with vertex set $V(\tau)$, and an integer $h \ge 0$  let
	\[
		z_\tau(h) := |\{v\in V(\tau) \mid \he_{\tau}(v) = h\}|, \qquad c_\tau(h):=\sum_{r=0}^h z_\tau(r).
	\]
	Define
	\[
		C_n(t) := \frac{1}{s_n} c_{\tau_{n,k}}\left(\left\lfloor\frac{s_n}{\sqrt{\ell_n}}t\right\rfloor\right) = \mu_n[0, t/2], \qquad t \ge 0.
	\]
	By definition, $\tau_{n,k}$ is uniform among rooted plane trees with degree
	statistics
	\[
		N_0^{(n)}=k, \qquad N_1^{(n)}=q_n, \qquad N_2^{(n)}=k-1, \qquad N_i^{(n)}=0\quad(i\ge 3).
	\]
	The size of the tree is $s_n \to \infty$, its maximal outdegree equals $2$. Moreover,
	\[
		\frac{1}{\ell_n} \sum_{i\ge 0}(i-1)^2 N_i^{(n)} = \frac{N_0^{(n)} + N_2^{(n)}}{\ell_n}=1.
	\]
	This allows us to apply \cite[Prop.~1, Prop.~3, Thm.~1]{angtuncio2020profiletreesgivendegree} with scaling
	sequence $\sqrt{\ell_n}$ and Vervaat transform given by Brownian excursion $e$, see~\cite{zbMATH03608918}, yielding
	\begin{align}
		\label{eq:cumulativeprofile}
		\left(C_n(t), t\ge 0\right) \convdis (\mathcal C(t), t\ge 0),
	\end{align}
	where $(\mathcal{Z},\mathcal{C})$ is the Lamperti pair associated with $e$. That is, we define first
	\[
		I_e(x) :=\int_0^x \frac{\mathrm{d}u}{e(u)}, \qquad 0\le x \le 1.
	\]
	We have  $I_e(x)< \infty$ for $0 \le x \le 1$ by~\cite[Prop.~3]{angtuncio2020profiletreesgivendegree}. $\mathcal{C}$ is given by  its right-continuous inverse,
	\[
		\mathcal C(t) := \inf\bigl\{x\in[0,1] \mid I_e(x) > t \bigr\},\qquad t\ge0,
	\]
	with the convention $\inf\emptyset=1$. $\mathcal{Z}$ is given by
	\[
		\mathcal{Z}(t) := e(\mathcal{C}(t)), \qquad t \ge 0.
	\]
	
	Since $e$ is positive on $(0,1)$, the function $I_e$ is continuous and strictly increasing from $[0,1]$ onto $[0,I_e(1)]$. Consequently,
	\[
		\mathcal{C}(t) = I_e^{-1}(t), \qquad 0 \le t \le I_e(1),
	\]
	whereas $\mathcal C(t)=1$ and $\mathcal Z(t)=0$ for $t \ge I_e(1)$. It follows in particular that $\mathcal C$ is almost surely continuous on $[0,\infty)$. Moreover, for $0<t<I_e(1)$, the inverse-function formula gives
	\[
		\mathcal{C}'(t) = \frac{1}{I_e'(\mathcal C(t))} = e(\mathcal{C}(t)) = \mathcal{Z}(t).
	\]
	By continuity, this yields
	\begin{align}
		\label{eq:ume}
		\mathcal{C}(t) = \int_0^t\mathcal{Z}(r)\,\mathrm{d}r, \qquad t\ge0.
	\end{align}
	Thus $\mathcal{C}$ is continuous, non-decreasing, and satisfies $\mathcal{C}(0)=0$ and $\lim_{t\to\infty}\mathcal C(t)=1$. Equivalently, it is the distribution function of the probability measure $\mathcal Z(t)\,\mathrm{d}t$ on $[0,\infty)$.
	
	Let $(L_e(x),x\ge 0)$ denote the occupation density of $e$, normalised by
	\[
	\int_0^1 f(e(t))\,\mathrm{d}t
	=\int_0^\infty f(x)L_e(x)\,\mathrm{d}x.
	\]
	Jeulin's identity~\cite{zbMATH03896047} in this normalisation states that
	\begin{align*}
		(\mathcal{Z}(t), t\ge 0) \eqdist \left(\frac{1}{2} L_e(t/2), t\ge 0\right),
	\end{align*}
	see the discussion following \cite[Thm.~1]{angtuncio2020profiletreesgivendegree}. Consequently, using~\eqref{eq:ume} we obtain
	\begin{align*}
		(\mathcal{C}(2x), x\ge0) &= \left(\int_0^{2x}\mathcal Z(r)\,\mathrm{d}r,x\ge0\right) \eqdist \left(\int_0^{2x}\frac12L_e(r/2)\,\mathrm{d}r,x\ge0\right)\\
		&= \left(\int_0^xL_e(y)\,\mathrm{d}y,x\ge0\right) = (\mu_e([0,x]),x\ge0).
	\end{align*}
	By~\eqref{eq:cumulativeprofile} it follows that
	\begin{align}
		\label{eq:mun}
		\mu_n \convdis \mu_e.
 	\end{align}

	It remains to pass from the profile of all vertices to the profile of the unary vertices. The tree $\tau_{n,k}$ has $q_n$ unary vertices and $\ell_n$ non-unary vertices. Hence, for every bounded continuous function $f$,
	\begin{align*}
		\mu_n(f)-\nu_n(f) &= \left(\frac{1}{s_n}-\frac{1}{q_n}\right) \sum_{u\in U(\tau_{n,k})} f\left(\gamma_{n,k}\he_{\tau_{n,k}}(u)\right)\\
		&\quad+ \frac{1}{s_n} \sum_{v\in V(\tau_{n,k})\setminus U(\tau_{n,k})} f\left(\gamma_{n,k}\he_{\tau_{n,k}}(v)\right).
	\end{align*}
	Since $s_n=q_n+\ell_n$, it follows that
	\begin{align*}
		\left|\int f(x)\,\mu_n(\mathrm{d}x)  - \int f(x)\,\nu_n(\mathrm{d}x) \right| &\le \left(1 - \frac{q_n}{s_n} \right) \|f\|_\infty + \frac{\ell_n}{s_n} \|f\|_\infty \\
		&= 2\|f\|_\infty\frac{\ell_n}{s_n} \\
		&= 2\|f\|_\infty\frac{2k-1}{2n-1} \to 0.
	\end{align*}
	By~\eqref{eq:mun} it follows that
	\[
		(\mu_n,\nu_n) \convdis (\mu_e,\mu_e).
	\]
\end{proof}

\begin{lemma}
	\label{le:matching}
	Let $q_n = 2j_n \to \infty$ with $j_n \in \ndN$, and let $x_{n,1}, \ldots, x_{n,q_n}$ be random points in $[0,\infty)$, and set
	\[
		\eta_n := \frac{1}{q_n}\sum_{i=1}^{q_n}\delta_{x_{n,i}}.
	\]
	Conditionally on these points, let $\{\{A_{n,r},B_{n,r}\} \mid 1\le r\le j_n\}$ be a uniform ordered partition of $[q_n]$ into pairs. If $\eta_n\convdis\eta$, where $\eta$ is a random probability measure, then jointly
	\[
		\frac{1}{j_n}\sum_{r=1}^{j_n} \delta_{\phi_+(x_{n,A_{n,r}},x_{n,B_{n,r}})} \convdis(\eta\otimes\eta)\circ\phi_+^{-1}
	\]
	and
	\[
		\frac{1}{j_n}\sum_{r=1}^{j_n} \delta_{\phi_-(x_{n,A_{n,r}},x_{n,B_{n,r}})} \convdis(\eta\otimes\eta)\circ\phi_-^{-1}.
	\]
\end{lemma}
\begin{proof}
	Fix one of the signs and a bounded continuous function $f$, and set
		\[
		g(x,y):=f(\phi_\pm(x,y)).
		\]
	For $a,b\in[q_n]$, write
		\[
		g_n(a,b):=g(x_{n,a},x_{n,b}),
		\]
	and, for $r\in[j_n]$, set
		\[
		X_{n,r}:=g_n(A_{n,r},B_{n,r}),
		\qquad
		A_n:=\frac{1}{j_n}\sum_{r=1}^{j_n}X_{n,r}.
		\]
	Note that $g_n$ is symmetric and $|g_n(a,b)|\le\|g\|_\infty$. We let
		\[
		\cF_n:=\sigma(x_{n,1},\ldots,x_{n,q_n}).
		\]
	Each block of a uniform ordered partition into pairs is a uniform two-element subset of $[q_n]$, and therefore
	\begin{align*}
		\Exb{A_n\mid\cF_n} = \Exb{X_{n,r}\mid\cF_n} = \frac{1}{q_n(q_n-1)} \sum_{a\ne b}g_n(a,b).
	\end{align*}
	At the same time,
	\begin{align*}
		\iint g(x,y)\,\eta_n(\mathrm{d}x)\eta_n(\mathrm{d}y)
		&= \frac{1}{q_n^2}\sum_{a,b}g_n(a,b)\\
		&= \frac{q_n-1}{q_n}\Exb{A_n\mid\cF_n} +\frac{1}{q_n^2}\sum_a g_n(a,a).
	\end{align*}
	Hence
	\begin{align}
		\label{eq:yo0}
		\left| \Exb{A_n\mid\cF_n} -\iint g(x,y)\,\eta_n(\mathrm{d}x)\eta_n(\mathrm{d}y) \right| \le\frac{2}{q_n}\|g\|_\infty.
	\end{align}
		
	For distinct $r,s\in[j_n]$, the ordered pair consisting of the $r$th and $s$th blocks is uniform among all ordered pairs of disjoint two-element subsets of $[q_n]$. Thus
		\begin{align*}
			\Exb{X_{n,r}X_{n,s}\mid\cF_n}
			&=
			\frac{1}{
				\binom{q_n}{2}\binom{q_n-2}{2}
			}
			\sum_{\substack{
					1\le a<b\le q_n,\;1\le c<d\le q_n\\
					\{a,b\}\cap\{c,d\}=\emptyset
			}}
			g_n(a,b)g_n(c,d).
		\end{align*}
		On the other hand, since each block is marginally uniform,
		\begin{align*}
			\Exb{X_{n,r}\mid\cF_n}\Exb{X_{n,s}\mid\cF_n} &= \frac{1}{\binom{q_n}{2}^2} \sum_{1\le a<b\le q_n} \sum_{1\le c<d\le q_n} g_n(a,b)g_n(c,d).
		\end{align*}
		There are
		$
			\binom{q_n}{2}\binom{q_n-2}{2}
		$
		ordered pairs of disjoint two-element subsets and exactly
		$
			\binom{q_n}{2} \left( \binom{q_n}{2}-\binom{q_n-2}{2}\right)
		$
		ordered pairs which are not disjoint. Hence
		\begin{align*}
			\left|\Cov(X_{n,r},X_{n,s}\mid\cF_n)\right| &\le \left( \frac{1}{ \binom{q_n}{2}\binom{q_n-2}{2}} - \frac{1}{\binom{q_n}{2}^2} \right) \binom{q_n}{2}\binom{q_n-2}{2} \|g\|_\infty^2\\
			&\quad+
			\frac{1}{\binom{q_n}{2}^2} \binom{q_n}{2} \left(\binom{q_n}{2}-\binom{q_n-2}{2}\right) \|g\|_\infty^2\\
			&= 2\left(1-\frac{\binom{q_n-2}{2}}{\binom{q_n}{2}}\right)\|g\|_\infty^2 =2\frac{4q_n-6}{q_n(q_n-1)}\|g\|_\infty^2 \le\frac{8\|g\|_\infty^2}{q_n}.
		\end{align*}
		It follows that
		\begin{align*}
			\Vab{A_n\mid\cF_n} &= \frac{1}{j_n^2} \sum_{r=1}^{j_n}\Vab{X_{n,r}\mid\cF_n} + \frac{2}{j_n^2} \sum_{1\le r<s\le j_n} \Cov(X_{n,r},X_{n,s}\mid\cF_n)\\
			&\le \frac{1}{j_n^2} \sum_{r=1}^{j_n}\Exb{X_{n,r}^2\mid\cF_n} + \frac{2}{j_n^2} \sum_{1\le r<s\le j_n} \left|\Cov(X_{n,r},X_{n,s}\mid\cF_n)\right|\\
			&\le \frac{\|g\|_\infty^2}{j_n} + \frac{2}{j_n^2}\binom{j_n}{2} \frac{8\|g\|_\infty^2}{q_n} = \frac{\|g\|_\infty^2}{j_n} + \frac{8(j_n-1)}{j_nq_n}\|g\|_\infty^2 \le \frac{10\|g\|_\infty^2}{q_n},
		\end{align*}
		where we used $j_n = q_n/2$. For every $\epsilon>0$, it follows from~\eqref{eq:yo0} that, for all sufficiently large $n$,
		\begin{align*}
				\left| \Exb{A_n\mid\cF_n} -\iint g(x,y)\,\eta_n(\mathrm{d}x)\eta_n(\mathrm{d}y) \right| \le \epsilon/2
		\end{align*}
		and hence
		\begin{align*}
			\Prb{ \left| A_n- \iint g(x,y)\, \eta_n(\mathrm{d}x)\eta_n(\mathrm{d}y) \right|>\epsilon} &\le
			\Prb{ \left| A_n-\Exb{A_n\mid\cF_n} \right|>\epsilon/2 }\\
			&=\Exb{ \Prb{\left|A_n-\Exb{A_n\mid\cF_n}\right|>\epsilon/2\mid\cF_n}}\\
			&\le \frac{4}{\epsilon^2} \Exb{\Vab{A_n\mid\cF_n}} \le \frac{40\|g\|_\infty^2}{\epsilon^2q_n}.
		\end{align*}
		Since $q_n \to \infty$, this yields
		\begin{align}
			\label{eq:equivalent}
			A_n -  \iint g(x,y) \eta_n(\mathrm{d}x) \eta_n(\mathrm{d}y) \convp 0.
		\end{align}
		Define
		\[
			\widehat{\eta}_n^\pm := \frac{1}{j_n} \sum_{r=1}^{j_n} \delta_{\phi_\pm(x_{n,A_{n,r}},x_{n,B_{n,r}})}.
		\]
		Equation~\eqref{eq:equivalent} states
		\begin{align}
			\label{eq:tf}
			\int f(z)\widehat{\eta}_n^\pm(\mathrm{d}z) -  \int f(z) \left((\eta_n\otimes\eta_n)\circ\phi_\pm^{-1}\right)(\mathrm{d}z) \convp 0.
		\end{align}
      %  {\color{red}
      %  Since $f\in C_b([0,\infty))$ and the sign were arbitrary, \eqref{eq:tf}
	%	holds for every bounded continuous $f$ and for both signs. Fix a countable
	%	convergence-determining family $(f_m)_{m\ge1}\subset C_b([0,\infty))$ with
	%	$0\le f_m\le1$, and metrize the weak topology on the space $\cP([0,\infty))$
	%	of probability measures on $[0,\infty)$ by
	%	\[
	%		d(\mu,\nu):=\sum_{m\ge1}2^{-m}\left|\int f_m\,\mathrm{d}\mu-\int f_m\,\mathrm{d}\nu\right|.
	%	\]
	%	Write $\nu_n^\pm:=(\eta_n\otimes\eta_n)\circ\phi_\pm^{-1}$. For every $M\ge1$,
	%	\[
	%		d(\widehat{\eta}_n^\pm,\nu_n^\pm)\le
	%		\sum_{m=1}^{M}2^{-m}\left|\int f_m\,\mathrm{d}\widehat{\eta}_n^\pm-\int f_m\,\mathrm{d}\nu_n^\pm\right|+2^{-M}.
	%	\]
	%	By \eqref{eq:tf}, for each fixed $M$ the finitely many summands with $m\le M$
	%	converge to zero in probability, whence
	%	$\limsup_{n\to\infty}\Prb{d(\widehat{\eta}_n^\pm,\nu_n^\pm)>2^{-M}+\epsilon}=0$
	%	for every $\epsilon>0$. Letting $M\to\infty$ and using $2^{-M}\to0$ yields
	%	$d(\widehat{\eta}_n^\pm,\nu_n^\pm)\convp0$, jointly in the two signs.
     %   }
		The map
		\[
			\mu \mapsto \left( (\mu\otimes\mu)\circ\phi_+^{-1}, (\mu\otimes\mu)\circ\phi_-^{-1}\right)
		\]
		is continuous for weak convergence, since $\mu_n \Rightarrow \mu$ implies $\mu_n\otimes\mu_n\Rightarrow\mu\otimes\mu$ and the maps $\phi_+$ and $\phi_-$ are continuous. Hence, by $\eta_n \convdis \eta$ and the continuous mapping theorem,
		\begin{align}
			\label{eq:tf2}
			\left( (\eta_n\otimes\eta_n)\circ\phi_+^{-1}, (\eta_n\otimes\eta_n)\circ\phi_-^{-1}\right) \convdis \left( (\eta\otimes\eta)\circ\phi_+^{-1}, (\eta\otimes\eta)\circ\phi_-^{-1}\right).
		\end{align}
		Combining~\eqref{eq:tf} and~\eqref{eq:tf2} yields
		\[
			(\widehat{\eta}_n^+, \widehat{\eta}_n^-) \convdis \left( (\eta\otimes\eta)\circ\phi_+^{-1}, (\eta\otimes\eta)\circ\phi_-^{-1}\right).
		\]
\end{proof}

\begin{proof}[Proof of Theorem~\ref{te:profile}]
	Let $k=k_n$ satisfy $k / \sqrt{n} \to 1$. Write $j:=n-k$ and $q:=2j$. Recall that by Lemma~\ref{le:conditioned-tree}, we have
	\[
		\tau(\mN_n \mid K_n=k) \eqdist \tau_{n,k}.
	\]
	In order to simplify notation we will identify the two trees, so that $\tau_{n,k}$ is constructed from $\mN_n$ conditioned on $K_n=k$. Let $u_1,\ldots,u_q$ denote the unary vertices of $\tau_{n,k}$, listed according to the canonical ordering used in Proposition~\ref{pro:tt}, and set
	\[
		\eta_{n,k} := \frac{1}{q}\sum_{i=1}^{q} \delta_{\gamma_{n,k}\he_{\tau_{n,k}}(u_i)}.
	\]
	By Proposition~\ref{pro:top-tree-profile}
	\begin{equation}
		\eta_{n,k} \convdis \mu_e.
	\end{equation}
	Let
	\[
		M = \left(\{A_r,B_r\}\right)_{1 \le r \le j}
	\]
	be the ordered partition of $[q]$ corresponding to the perfect matching of $u_1,\ldots,u_q$ in the decomposition of $\mN_n$ according to Proposition~\ref{pro:tt}. Conditionally on $\tau_{n,k}$, the partition $M$ is uniform among all ordered partitions of $[q]$ into pairs and is independent of $\tau_{n,k}$. For each $1 \le r \le j$ let $v_r$ denote the reticulation leaf associated with the pair $\{A_r,B_r\}$. The two directed paths from the planted root of $\mN_n$ to $v_r$ have lengths
	\[
		\he_{\tau_{n,k}}(u_{A_r})+3 \qquad \text{and} \qquad \he_{\tau_{n,k}}(u_{B_r})+3.
	\]
	Consequently,
	\begin{align}
		\label{eq:heq}
		\he_{\mN_n}^{\pm}(v_r) = \phi_\pm\left(\he_{\tau_{n,k}}(u_{A_r}), \he_{\tau_{n,k}}(u_{B_r}) \right)+3.
	\end{align}
	Define the empirical height measures of the leaves whose parent is a reticulation vertex, at scale
	$\gamma_{n,k}$, by
	\[
		\widehat{\Pi}_{n,k}^{\pm} := \frac{1}{j}\sum_{r=1}^{j} \delta_{\gamma_{n,k}\he_{\mN_n}^{\pm}(v_r)}.
	\]
	Applying Lemma~\ref{le:matching} to the points
	\[
		x_{n,i} := \gamma_{n,k} \he_{\tau_{n,k}}(u_i), \qquad 1 \le i \le q,
	\]
	we obtain jointly
	\begin{align}
		\label{eq:projo}
		\frac{1}{j}\sum_{r=1}^{j} \delta_{\phi_\pm(x_{n,A_r},x_{n,B_r})} \convdis (\mu_e\otimes\mu_e)\circ\phi_\pm^{-1} = \Pi_e^\pm.
	\end{align}
	Equation~\eqref{eq:heq} may be rephrased by
	\[
		\gamma_{n,k} \he_{\mN_n}^{\pm}(v_r) = \phi_\pm(x_{n,A_r}, x_{n,B_r}) + 3\gamma_{n,k}.
	\]
	Since $\gamma_{n,k} \to 0$ it follows from~\eqref{eq:projo} that
	\begin{align}
		\label{eq:whpc}
		\left(\widehat{\Pi}_{n,k}^+, \widehat{\Pi}_{n,k}^-\right) \convdis \bigl(\Pi_e^+,\Pi_e^-\bigr).
	\end{align}

	We next include the $k$ leaves whose parent is a tree vertex. Define
	\[
		\overline{\Pi}_{n,k}^{\pm} := \frac{1}{n}\sum_{v\in L(\mN_n)} \delta_{\gamma_{n,k}\he_{\mN_n}^{\pm}(v)}.
	\]
	For every bounded continuous function $f$,
	\[
		\int f(z)\, \overline{\Pi}_{n,k}^{\pm}(\mathrm{d}z) = \frac{j}{n}\int f(z)\,\widehat{\Pi}_{n,k}^{\pm}(\mathrm{d}z) + \frac{1}{n} \sum_{\substack{v\in L(\mN_n)\\
			v\text{ has a tree vertex parent}}} f\left( \gamma_{n,k} \he_{\mN_n}^{\pm}(v) \right).
	\]
	Using $j=n-k$ it follows that
	\[
		\left| \int f\,\overline{\Pi}_{n,k}^{\pm}(\mathrm{d}z) - \int f\,\widehat{\Pi}_{n,k}^{\pm}(\mathrm{d}z) \right| \le 2 \|f\|_\infty\frac{k}{n}.
	\]
	Since $k/\sqrt{n} \to 1$, we have $k/n \to 0$. By~\eqref{eq:whpc} it follows that
	\begin{align}
		\label{eq:barscale}
		\left(\overline{\Pi}_{n,k}^+,\overline{\Pi}_{n,k}^-\right) \convdis \left(\Pi_e^+,\Pi_e^-\right).
	\end{align}
	
	Next, we replace $\gamma_{n,k}$ by the scale $2^{-3/2}n^{-3/4}$. Put $c_{n,k}:= \gamma_{n,k}^{-1} 2^{-3/2}n^{-3/4}$. Since $k / \sqrt{n} \to 1$ we have $c_{n,k} \to 1$ as noted in~\eqref{eq:gammascale}. If $\sigma_c(x):=cx$, then, conditioned on
	$K_n=k$,
	\[
		\Pi_n^\pm = \overline{\Pi}_{n,k}^{\pm}\circ\sigma_{c_{n,k}}^{-1}.
	\]
	The map
	\[
		(c,\mu) \mapsto \mu\circ\sigma_c^{-1}
	\]
	is continuous for $c>0$ and $\mu$ a probability measure on $\ndR$. By~\eqref{eq:barscale} it follows that
	\begin{align}
		\label{eq:mixture}
		\left((\Pi_n^+,\Pi_n^-) \mid K_n=k_n \right) \convdis \left(\Pi_e^+,\Pi_e^-\right)
	\end{align}
	for every deterministic sequence $k_n/\sqrt n\to1$. 
	
	By Corollary~\ref{co:ttsize}  we have $K_n / \sqrt{n} \convp 1$. Using~\eqref{eq:mixture} it follows that
	\begin{align*}
		(\Pi_n^+,\Pi_n^-)  \convdis \left(\Pi_e^+,\Pi_e^-\right).
	\end{align*}
	This completes the proof.
\end{proof}

\begin{proof}[Proof of Theorem~\ref{te:main}]
	We have
	\begin{align}
		\label{eq:lmu1}
		\int_0^\infty x\,\Pi_n^\pm(\mathrm{d}x) = 2^{-3/2}n^{-7/4}\Si^\pm(\mN_n).
	\end{align}
	For any constant $R>0$ let $f_R(x) := \min(x,R)$. Theorem~\ref{te:profile} gives jointly
	\begin{align}
		\label{eq:lmu2}
		\int f_R(x)\,\Pi_n^\pm(\mathrm{d}x) \convdis \int f_R(x)\,\Pi_e^\pm(\mathrm{d}x).
	\end{align}
	The supports of both $\Pi_n^+$ and $\Pi_n^-$ are contained in $[0, 2^{-3/2}n^{-3/4}\He^+(\mN_n)]$. By Theorem~\ref{te:main2}, the upper end of this interval is stochastically bounded as $n \to \infty$. Consequently, for every $\epsilon>0$,
	\[
		\lim_{R\to\infty} \limsup_{n\to\infty} \Prb{\left| \int x\,\Pi_n^\pm(\mathrm{d}x) -\int f_R(x)\,\Pi_n^\pm(\mathrm{d}x) \right| > \epsilon} = 0.
	\]
	Likewise, since $\sup_{0\le t\le 1}e(t)<\infty$ we have
	\[
		\lim_{R\to\infty}  \Prb{\left| \int x\,\Pi_e^\pm(\mathrm{d}x) -\int f_R(x)\,\Pi_e^\pm(\mathrm{d}x) \right| > \epsilon} = 0.
	\]
	Letting $R \to \infty$ and using \eqref{eq:lmu1} and \eqref{eq:lmu2} yields jointly
	\[
		2^{-3/2}n^{-7/4}\Si^+(\mN_n) \convdis \int_0^1 \int_0^1\max(e(s),e(t))\,\mathrm{d}s\mathrm{d}t
	\]
	and
	\[
		2^{-3/2}n^{-7/4}\Si^-(\mN_n) \convdis \int_0^1 \int_0^1\min(e(s),e(t))\,\mathrm{d}s\mathrm{d}t.
	\]
	This completes the proof of Theorem~\ref{te:main}.
\end{proof}

\subsection{Moments of the limiting Sackin indices}
\label{sec:sackin-moments}

We recall necessary background on the Brownian continuum random tree~\cite{MR1085326,MR1166406,MR1207226}. For a continuous function $f:[0,1] \to [0,\infty)$ with $f(0)=f(1)=0$, define a pre-metric on $[0,1]$ by
\[
d_f(s,t) := f(s) + f(t) - 2 \inf_{u \in [\min(s,t),\max(s,t)]} f(u)
\]
for $s,t \in [0,1]$. The quotient metric space  of $[0,1]$ by the equivalence relation $d_f(s,t)=0$, equipped with the induced metric $d_{\cT_f}$, is called the rooted real tree coded by $f$.  We denote it by $\cT_f$ and let $\pi_f: [0,1] \to \cT_f$ denote the canonical projection. The root is given by $\rho_f := \pi_f(0)$. By construction,
\[
	f(s) = d_{\cT_f}(\rho_f, \pi_f(s))
\]
for all $s \in [0,1]$. The metric space $\cT_f$ has the property that, for any $x,y \in \cT_f$, there is a unique isometric map
\[
	\varphi_{x,y}: [0, d_{\cT_f}(x,y)] \to \cT_f
\]
with $\varphi_{x,y}(0) = x$ and $\varphi_{x,y}(d_{\cT_f}(x,y)) = y$. We let $[x,y]$ denote the image of $\varphi_{x,y}$. Furthermore, for any finite collection of distinct points $x_1, \ldots, x_k \in \cT_f$, we may form the reduced tree spanned by these points:
\[
	\cR(\cT_f, x_1, \ldots, x_k) = \bigcup_{i=1}^k [\rho_f, x_i],
\]
which is isometric to a discrete  rooted tree with edge-lengths. We refer the reader to~\cite[Ex. 10.6.]{art} for details on real trees that justify these standard constructions.

Let $U_1,U_2,\ldots$ be independent uniform random variables on $[0,1]$,
independent of the Brownian excursion $e$. Conditionally on~$e$, the variables $e(U_i)$ are independent with common law $\mu_e$, the occupation measure of~$e$. Hence, for every integer $p \ge 1$,
\begin{align}
	\label{eq:lpmp}
	\Ex{(L^\pm)^p} = \Exb{ \prod_{j=1}^p \phi_\pm(e(U_{2j-1}), e(U_{2j}))}.
\end{align}
Indeed, conditionally on $e$, the expectation on the right-hand side is the
$p$th power of $\int_0^1 \int_0^1 \phi_{\pm}(e(s),e(t))\,\mathrm{d}s\mathrm{d}t$.

For every $k \ge 2$ and every tree $T \in \cT_k$, fix an arbitrary ordering of its set of edges $E(T)$. The random real tree $\cT_e$ encoded by $e$ is called  the Brownian continuum random tree. The reduced tree $\cR_k$ of $\cT_e$ spanned by $\pi_e(U_1), \ldots, \pi_e(U_k)$ is almost surely isometric to a tree $T_k$ from $\cT_k$ (and hence has $2k-1$ edges) with edge lengths specified in the fixed ordering by an edge length vector $\bm{X}_k$. The joint density of $(T_k, \bm{X}_k)$ with respect to the counting  measure on $\cT_k$ and the Lebesgue measure on $[0, \infty[^{2k-1}$ is given by
\begin{align}
	\label{eq:d1}
	(T,\bm{x}) \mapsto  2^{2k} \|\bm{x}\|_1 \exp(-2\|\bm{x}\|_1^2).
\end{align}
In particular, $T_k$ is uniformly distributed on $\cT_k$. We refer the reader to~\cite[Sec. 4.3]{MR1207226} for detailed justifications. We set
\[
	S_k := \sum_{e\in E(T_k)}X_{k,e}, \qquad \bm{D}_k:=\frac{\bm X_k}{S_k}.
\]
The following fact is a standard property of the Brownian tree.
\begin{proposition}
	\label{pro:standard}
	The random variables $T_k$, $S_k$, and $\bm{D}_k$ are independent, with
	\begin{align*}
		T_k&\text{ uniform on }\cT_k,\\
		\bm D_k&\eqdist\mathrm{Dir}(1,\ldots,1)
		\quad\text{with }2k-1\text{ coordinates},\\
		2S_k&\eqdist \chi_{2k},
	\end{align*}
	where $\chi_{2k}$ denotes the chi distribution with $2k$ degrees of
	freedom. Equivalently, for every real $r>-2k$,
	\begin{align*}
		\Ex{S_k^r} = 2^{-r/2}\frac{\Gamma(k+r/2)}{\Gamma(k)}.
	\end{align*}
\end{proposition}
\begin{proof}
	Using~\eqref{eq:d1}, it follows that the joint density of
	$(T_k,S_k,\bm D_k)$, with respect to counting measure, Lebesgue measure
	$\mathrm{d}s$, and Lebesgue measure on the  simplex $\Delta_{2k-1}$, is
	\[
		(T, s, \bm{y}) \mapsto 2^{2k}s^{2k-1}e^{-2s^2}.
	\]
	Note that
	\[
		2^{2k}s^{2k-1}e^{-2s^2}	= \frac{1}{(2k-3)!!}\left( \frac{2^{k+1}}{(k-1)!}s^{2k-1}e^{-2s^2} \right) (2k-2)!.
	\]
	The simplex $\Delta_{2k-1}$ has volume $1/(2k-2)!$, hence $(2k-2)!$ is the density of the $\mathrm{Dir}(1,\ldots,1)$ distribution on $\Delta_{2k-1}$. We have $|\cT_k| = (2k-3)!!$, hence $1/(2k-3)!!$ is the density of the uniform distribution on $\cT_k$. Furthermore,
	\[
		\frac{2^{k+1}}{(k-1)!}s^{2k-1}e^{-2s^2} 
	\]
	is the density of $2^{-1}\chi_{2k}$. Finally, for $r>-2k$,
	\begin{align*}
		\Ex{S_k^r} = \frac{2^{k+1}}{(k-1)!} \int_0^\infty s^{2k+r-1}e^{-2s^2}\,\mathrm{d}s
		= 2^{-r/2}\frac{\Gamma(k+r/2)}{\Gamma(k)}.
	\end{align*}
	This completes the proof.
\end{proof}

Let $p \ge 1$ be an integer. For each $T \in \cT_{2p}$ we let $\bm{E}_T=(E_e)_{e \in E(T)}$ consist of independent exponential random variables with mean one.

\begin{lemma}
	\label{le:expshift}
	We have
	\[
		\Ex{(L^\pm)^p} = 2^{3p/2-1}\frac{\Gamma(5p/2)}{\Gamma(5p-1)} \sum_{T\in\cT_{2p}} \Ex{F_{T,p}^\pm(\bm E_T)}.
	\]
\end{lemma}
\begin{proof}
	 From~\eqref{eq:lpmp} and Proposition~\ref{pro:standard} it follows that
	\begin{align*}
		\Ex{(L^\pm)^p} 	&= \Exb{ F_{T_{2p},p}^\pm(\bm{X}_{2p})} \\
		&= \sum_{T\in\cT_{2p}} \frac{1}{|\cT_{2p}|} \Exb{ F_{T,p}^\pm(\bm{X}_{2p})} \\
		&= \sum_{T\in\cT_{2p}} \frac{1}{(4p-3)!!} \Exb{ S_{2p}^{p} F_{T,p}^\pm(\bm{X}_{2p}/S_{2p})} \\
		&= \sum_{T \in \cT_{2p}} \frac{1}{(4p-3)!!} \Exb{ S_{2p}^{p} } \Exb{ F_{T,p}^\pm(\bm{X}_{2p}/S_{2p})} \\
		&= \sum_{T \in \cT_{2p}} \frac{1}{(4p-3)!!} 2^{-p/2} \frac{\Gamma(5p/2)}{\Gamma(2p)} \Exb{ F_{T,p}^\pm(\bm{X}_{2p}/S_{2p})}.
	\end{align*}
	By Proposition~\ref{pro:standard} we know that $\bm{X}_{2p}/S_{2p}$ follows a $\mathrm{Dirichlet}(1, \ldots, 1)$ distribution on the simplex. Likewise, \(\bm{E}_T / \|\bm{E}_T\|_1\) follows a
\(\mathrm{Dirichlet}(1,\ldots,1)\) distribution and is independent of
\(\|\bm{E}_T\|_1\); see~\cite[Thm. 4.1]{zbMATH03954145}. Hence
	\begin{align*}
		\Exb{ F_{T,p}^\pm(\bm{X}_{2p}/S_{2p})} &= \Exb{ F_{T,p}^\pm(\bm{E}_T / \|\bm{E}_T\|_1)} \\
		&= \Exb{\|\bm{E}_T\|_1^{p}}^{-1} \Exb{\|\bm{E}_T\|_1^{p}} \Exb{ F_{T,p}^\pm(\bm{E}_T / \|\bm{E}_T\|_1)} \\
		&= \Exb{\|\bm{E}_T\|_1^{p}}^{-1} \Exb{ F_{T,p}^\pm(\bm{E}_T )} \\
		&= \frac{\Gamma(4p-1)}{\Gamma(5p-1)} \Exb{ F_{T,p}^\pm(\bm{E}_T )}.
	\end{align*}
	Note that
	\begin{align*}
		\frac{1}{(4p-3)!!} \frac{\Gamma(4p-1)}{\Gamma(2p)} &= \frac{1}{(4p-3)!!} (4p-2) (4p-3) \cdots (2p) \\
		&= \frac{1}{ (2p-1)!!} (4p-2) (4p-4) \cdots (2p) \\
		&= 2^{2p-1}.
	\end{align*}
	Hence we arrive at
	\[
		\Ex{(L^\pm)^p} = 2^{3p/2-1}\frac{\Gamma(5p/2)}{\Gamma(5p-1)} \sum_{T\in\cT_{2p}} \Ex{F_{T,p}^\pm(\bm E_T)}.
	\]
\end{proof}

For $p \ge 1$, set
\begin{align}
	\label{eq:apm}
	A_p^\pm := \sum_{T\in\cT_{2p}} \Ex{F_{T,p}^\pm(\bm E_T)},
\end{align}
where $\bm{E}_T=(E_e)_{e\in E(T)}$ consists of independent exponential random variables with mean one.

\begin{lemma}
	\label{le:compute}
	For each integer $p \ge 1$ the number $A_p^{\pm}$ is rational and may be calculated using a finite number of rational arithmetic operations. We have
	\begin{align*}
		A_1^+ &= \frac{5}{2} & \qquad A_1^- &= 3/2, \\
		A_2^+ &= \frac{1095}{4} & \qquad A_2^- &= \frac{395}{4} \\
		A_3^+ &= \frac{1354665}{8} & \qquad A_3^- &= \frac{294375}{8}.
	\end{align*}
\end{lemma}
\begin{proof}
	We shall use two standard facts about exponential clocks~\cite[Thm.~2.3.1, Thm~2.3.3]{zbMATH01013027}:  If $(Y_a)_{a\in\mathcal A}$ are independent exponential random variables with rate one, then
	$
	W := \min_{a \in \cA} Y_a
	$
	has exponential distribution with rate $|\cA|$ and hence mean $1 / |\cA|$. The almost surely unique minimising index is uniform on $\cA$ and independent of $\min_{a \in \cA} Y_a$. Moreover, interpreting the exponential random variables as continuous time clocks, after the first clock has rung, the residual lifetimes of the clocks which have not rung are again independent exponential random variables with rate one.
	
	Now, fix $p \ge 1$ and $T \in \cT_{2p}$. We reveal independent exponential edge lengths of $T$ through a continuous-time exploration. All edges are oriented away from the planted root. At time zero only the planted edge is active. Whenever an edge becomes active, it is assigned an independent exponential clock of rate one. When this clock rings, the edge is completed and deactivated. At this time, if the target vertex of the completed edge is an internal vertex $v$, then the two outgoing edges starting at $v$ are activated. But if instead the child endpoint is a leaf we record its label. With this construction, the completion time of leaf $i$ is distributed like $h_i(T,\bm{E}_T)$, so we may identify $h_i(T,\bm{E}_T)$ with this completion time.
	\smallskip
    
	For each $1 \le j \le p$ set
	\[
		\sigma_j^\pm := \phi_{\pm}\left( h_{2j-1}(T,\bm E_T), h_{2j}(T,\bm E_T) \right).
	\]
	We may interpret $\sigma_j^\pm$ as maximum or minimum of the completion times of $2j-1$ and $2j$. For $\cB \subset [2p]$ set
	\[
		I_+(\cB) := \{i\in[p] \,\mid\, \{2i-1,2i\} \subset \cB\},
		\qquad
		I_-(\cB) := \{i\in[p] \,\mid\, \{2i-1,2i\}\cap \cB \neq \emptyset\}.
	\]
	Hence, if at some time point  the completed leaf labels in the continuous-time exploration are equal to $\cB$, then stopping time $\sigma_j^+$ has occurred when $j\in I_+(\cB)$, whereas $\sigma_j^-$ has occurred precisely when $j\in I_-(\cB)$.
	\smallskip
    
	At a given time in the exploration let $\cA$ denote the set of active edges, let $\cB\subseteq[2p]$ be
	the set of completed leaf labels, and let $\cR \subset [p]\setminus I_\pm(\cB)$
	be the set of indices whose stopping times have not yet occurred. Starting
	from such a state, let
	\[
		\sigma_j^\pm(\cA,\cB), \qquad j\in \cR,
	\]
	denote the additional time until the stopping time $\sigma_j^\pm$
	occurs, when the process is continued from the active set $\cA$ with
	completed leaf set $\cB$ and fresh independent exponential residual clocks on
	the active edges. Define
	\[
		G_{\cA,\cB,\cR}^\pm(\bm z) := \Exb{\exp\left(-\sum_{j\in \cR}z_j\sigma_j^\pm(\cA,\cB)\right)},
		\qquad \bm z=(z_1,\ldots,z_p).
	\]
	Suppose now that $\cA \ne \emptyset$. For an active edge $a \in \cA$, let $(\cA_a,\cB_a)$ be the state obtained if $a$ is the next edge that gets completed. More explicitly, if the target vertex of $a$ is an internal vertex with outgoing edges $a_1$ and $a_2$, then
	\[
		\cA_a = (\cA\setminus\{a\})\cup\{a_1,a_2\}, \qquad \cB_a=\cB.
	\]
	If the child endpoint of $a$ is a leaf $i \in [2p]$, then
	\[
		\cA_a = \cA \setminus\{a\}, \qquad \cB_a=\cB\cup\{i\}.
	\]
	Finally set
	\[
		\cR_a^\pm:=\cR\setminus I_\pm(\cB_a).
	\]
	Let $W$ be the waiting time until the next edge gets completed and let $K$ be the active edge completed at that time. Then $W$ follows an exponential distribution with rate $|\cA|$, and the edge $K$ is uniform on $\cA$, and $W$ is independent of $K$. During the waiting time $W$, every unresolved stopping time from $\cR$ receives the same increment~$W$. Thus if $K=a$ for $a \in \cA$ and $j \in \cR \setminus \cR_{a}^\pm$, then $\sigma_j^\pm(\cA, \cB) = W$ as completing edge $a$ has triggered the stopping time for $j$.  If $j \in \cR_{a}^\pm$, then $\sigma_j^\pm(\cA, \cB)$ is distributed like $W$ plus an independent copy of $\sigma_j^{\pm}(\cA_a, \cB_a)$. Hence, using that $K$ is uniform on $\cA$ we obtain
	\begin{align*}
		G_{\cA,\cB,\cR}^\pm(\bm z) &= \sum_{a \in \cA} \frac{1}{|\cA|} \Exb{e^{-\sum_{j\in \cR}z_jW}} \Exb{\exp\left(-\sum_{j\in \cR_a^\pm}z_j\sigma_j^\pm(\cA_a,\cB_a)\right)}  \\
		&= \sum_{a\in\cA} \frac{1}{|\cA|} \Exb{e^{-\sum_{j\in \cR}z_jW}} G_{\cA_a,\cB_a,\cR_a^\pm}^\pm(\bm z).
	\end{align*}	
	Since $W$ follows an exponential distribution with rate $|\cA|$ this simplifies to
	\begin{align}
		\label{eq:gabr}
		G_{\cA,\cB,\cR}^\pm(\bm z) &= \sum_{a\in\cA} \frac{1}{|\cA|} G_{\cA_a,\cB_a,\cR_a^\pm}^\pm(\bm z)  \int_{0}^\infty |\cA| \exp(- |\cA|x) \exp\left(- \sum_{j \in \cR} z_j x\right)\,\mathrm{d}x  \\
		&= \sum_{a\in\cA}  \frac{1}{|\cA|+\sum_{j\in \cR}z_j}G_{\cA_a,\cB_a,\cR_a^\pm}^\pm(\bm z). \nonumber
	\end{align}
	With $e_0$ denoting the planted edge of $T$, 
	\[
		G_{\{e_0\}, \emptyset, [p]}^\pm(\bm z) = \Exb{\exp\left(-\sum_{j=1}^p z_j\sigma_j^\pm\right)}.
	\]
	We may determine this series explicitly by repeatedly
applying~\eqref{eq:gabr} until \(\cA=\emptyset\). At every reachable
terminal state, all stopping times have occurred and hence
\(\cR=\emptyset\), so the corresponding Laplace transform is equal to
\(1\). This happens after a bounded number of steps since the tree \(T\)
has only finitely many edges. Note that all coefficients are rational numbers.

	Since  $\sigma_j^\pm$ is stochastically bounded by the sum of $|E(T)|$ many independent exponential random variables, it follows that $\sigma_j^\pm$ has finite exponential moments and this Laplace transform $G_{\{e_0\}, \emptyset, [p]}^\pm(\bm{z})$ is analytic in a neighbourhood of $\bm{z} = \bm{0}$. Therefore its Taylor expansion gives
	\begin{align}
		\label{eq:ftp}
		\Ex{F_{T,p}^\pm(\bm E_T)} = (-1)^p[z_1\cdots z_p]\, G_{\{e_0\},\emptyset,[p]}^\pm(\bm z).
	\end{align}
	Note that for $J \subset [p]$, $r \ge 1$
	\[
	\left[\prod_{j\in J}z_j\right]\frac{1}{r+\sum_{j\in \cR}z_j} =
	\begin{cases}
		\frac{(-1)^{|J|}|J|!}{r^{|J|+1}},
		& J \subset \cR,\\[3mm]
		0,
		& J \not \subset \cR.
	\end{cases}
	\]
	Thus $\Ex{F_{T,p}^\pm(\bm E_T)} \in \ndQ$ may be determined from~\eqref{eq:ftp} and~\eqref{eq:gabr} using a finite number of rational arithmetic operations. 
	
	In order to determine~$A_p^\pm$ from~\eqref{eq:apm} we merely need to sum~\eqref{eq:ftp} over all trees $T \in \cT_{2p}$. These trees may be generated via an edge-insertion recursion: from a tree in $\cT_{m-1}$, subdivide one of its $2m-3$ edges by adding a vertex and attach the new leaf labelled $m$ to the freshly created vertex by an edge. This way every element of $\cT_m$ is generated exactly once. 
	\smallskip
    
	Evaluating the sums over these finite sets yields the stated values for $A_1^\pm$, $A_2^\pm$ and $A_3^\pm$.  The evaluation may be sped up by using the fact that 
	\[
		\Ex{F_{T,p}^\pm(\bm E_T)} = \Ex{F_{T',p}^\pm(\bm E_{T'})}
	\]
	whenever $T$ and $T'$ are isomorphic as rooted trees. 
\end{proof}

\begin{proof}[Proof of Theorem~\ref{te:smoments}]
	By Lemma~\ref{le:expshift} we have
	\[
		\Ex{(L^\pm)^p} = 2^{3p/2-1}\frac{\Gamma(5p/2)}{\Gamma(5p-1)} \sum_{T\in\cT_{2p}} \Ex{F_{T,p}^\pm(\bm E_T)}.
	\]
	
	The integral expression for $\Ex{F_{T,p}^\pm(\bm E_T)}$ follows directly
	from the fact that the coordinates of $\bm{E}_T$ are independent exponential
	random variables with rate one. 
	
	By Lemma~\ref{le:compute} we have $\sum_{T\in\cT_{2p}} \Ex{F_{T,p}^\pm(\bm E_T)} \in \ndQ$. Hence if $p$ is even, then
	\[
		2^{3p/2-1}\frac{\Gamma(5p/2)}{\Gamma(5p-1)} \in \ndQ.
	\]
	and hence $\Ex{(L^\pm)^p} \in \ndQ$. If $p$ is odd, then $k := (5p-1)/2$ is an integer and
	\[
		\Gamma(5p/2) =  \Gamma\left(k + \frac{1}{2}\right) = \frac{(2k)!}{4^k k!}\sqrt{\pi}.
	\]
	Hence 
	\[
			2^{3p/2-1}\frac{\Gamma(5p/2)}{\Gamma(5p-1)} \in \sqrt{2 \pi} \ndQ
	\]
and thus \(\Ex{(L^\pm)^p} \in \sqrt{2\pi}\,\ndQ\). Finally, substituting
the values of \(A_1^\pm\), \(A_2^\pm\), and \(A_3^\pm\) from
Lemma~\ref{le:compute} into the moment formula of
Lemma~\ref{le:expshift} yields the stated expressions for the first three
moments. This completes the proof.
\end{proof}

\section{Local limits}
\label{sec:local}

Throughout this section we set $I_n:=[\tfrac12\sqrt n,\tfrac32\sqrt n]$ and,
conditionally on $\{K_n=k\}$, realise $\mN_n$ as the blow-up
$\Lambda^{-1}(T,Y,M)$ of Proposition~\ref{pro:tt}, with $j=n-k$, $T$ uniform on
$\cT_k$, $Y$ uniform on compositions of $2j$ into $2k-1$ non-negative parts and
$M$ a uniform ordered perfect matching of $[2j]$. Here we denote by $\Lambda$ the bijection given by Proposition \ref{pro:tt}. We describe the asymptotic neighbourhood of \(\mN_n\) around four
canonical base points: the root, a uniform tree vertex, a uniform
reticulation vertex, and a uniform leaf. We then combine these limits to
obtain the limit at a uniform vertex and the corresponding quenched
statements.
All four limits are built from the single repeating unit, the \emph{cell} of
Figure~\ref{fig:limit}(a). Throughout this section $\tau_n^{*}$ denotes the top tree
component of $\mN_n$, which we regard as retaining the planted root; it is the
object from which $\tau_n = \tau(\mN_n)$ is obtained by deleting the planted root, adding
a plane order, and forgetting the leaf labels.

\subsection{Marked local weak convergence}
\label{ssec:lwc}

For a rooted directed graph $N$ and $v\in V(N)$ we write $U_R(N,v)$ for the
subgraph of $N$ induced by the vertices at undirected graph distance at most
$R$ from $v$, rooted (and marked) at $v$; when $v$ is the root of $N$ we
write $U_R(N)$. Vertices carry a \emph{mark} indicating their type
(root, tree vertex, reticulation, or leaf), and two rooted directed marked
graphs are \emph{isomorphic}, written $\cong$, if there is a bijection of
their vertex sets preserving the root, the marks, and the edge orientations.

\begin{definition}[Marked local weak convergence]\label{def:lwc-marked}
Let $(N_n,v_n)_{n\ge 1}$ be random rooted directed graphs with a
distinguished vertex $v_n$, and let $(N,v)$ be a (possibly random) such
object. We write
\[
  (N_n,v_n)\;\convdis\;(N,v)
\]
in the \emph{local weak sense} if, for every fixed $R\ge 0$ and every
finite rooted directed marked graph $G$,
\[
  \Prb{U_R(N_n,v_n)\cong G}\;\xrightarrow{n\to\infty}\;\Prb{U_R(N,v)\cong G}.
\]
When the limit is deterministic, the right-hand side equals $\one_{\{G\cong
U_R(N,v)\}}$, so the condition reduces to
$\Prb{U_R(N_n,v_n)\cong U_R(N,v)}\to 1$, for every $R\ge 0$.
This is the natural extension to rooted directed marked graphs of the
notion of Benjamini and Schramm~\cite{MR1873300}.
\end{definition}

\FloatBarrier
\subsection{Local limit at the root}
\label{ssec:root}
In this subsection we introduce the deterministic infinite network
$\hat{\mN}$ and state the first of our local limit results,
\[
  \mN_n\;\xrightarrow{\;d\;}\;\hat{\mN}
\]
in the local weak sense (at the root).

\begin{definition}[The limit network $\hat{\mN}$]\label{def:limit}
The \emph{rooted limit network} \(\hat{\mN}\) is the infinite directed network
defined as follows. It has a root \(o\), followed by a one-ended spine of
tree vertices
\(o\to v_1\to v_2\to\cdots\), and for each \(t\ge1\):
\begin{itemize}
\item $v_t$ has a reticulation child $\rho_t$, which has a leaf child $\ell_t$;
\item $\rho_t$ has a second parent $w_t$, which is a tree vertex on a
bi-infinite spine. The reticulation child associated with $w_t$ is the
already introduced vertex $\rho_t$. Every other vertex of this spine has
a reticulation child with a leaf child and a second parent on another
bi-infinite decorated spine, and the construction continues recursively.
\end{itemize}
All spines and all vertices introduced at every level are pairwise distinct.
The \emph{decorated chain} $\hat{\cL}$ is the same object with a bi-infinite
spine $\cdots \to p_{-1} \to p_0 \to p_1 \to \cdots$, rooted at $p_0$.
\end{definition}
\smallskip

\begin{center}
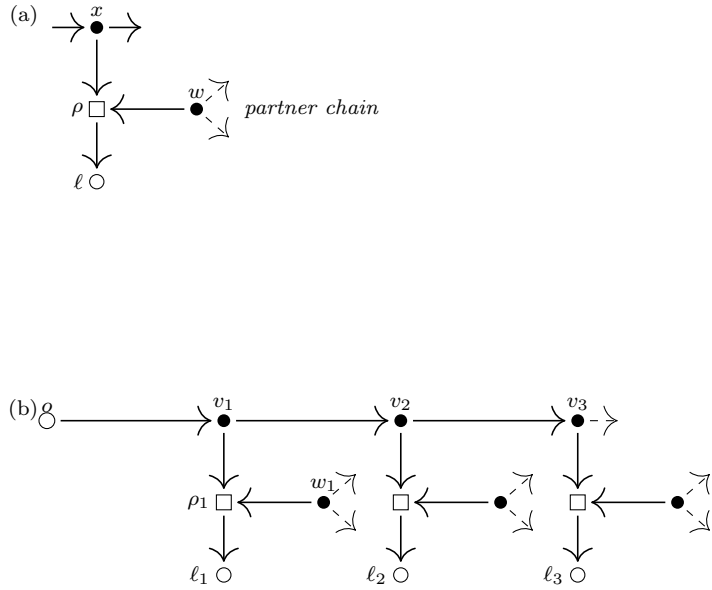

\begin{minipage}{0.95\textwidth}
\centering
\begin{tikzpicture}[x=0.60cm,y=0.60cm]
\def\dx{3.9}
\def\cellx{1.1}
%---- (a) the cell ----
\begin{scope}
  \node[slbl,left=2pt] at (0,0.25) {(a)};
  \pic (P) at (\cellx,0) {cell};
  \draw[e] ($(P-v)+(-1.1,0)$) -- (P-v);
  \draw[e] (P-v) -- ($(P-v)+(1.1,0)$);
  \node[slbl,above=3pt] at (P-v) {$x$};
  \node[slbl,left=4pt]  at (P-r) {$\rho$};
  \node[slbl,left=4pt]  at (P-l) {$\ell$};
  \node[slbl,above=3pt] at (P-w) {$w$};
  \node[note,right=3pt] at ($(P-w)+(0.8,0)$) {partner chain};
\end{scope}

%---- (b) N-hat ----
\begin{scope}[yshift=-5.2cm]
  \node[slbl,left=2pt] at (0,0.25) {(b)};
  \node[ro] (o) at (0,0) {}; \node[slbl,above=2pt] at (o) {$o$};
  \pic (A1) at (1*\dx,0) {cell};
  \pic (A2) at (2*\dx,0) {cell};
  \pic (A3) at (3*\dx,0) {cell};
  \draw[e] (o)--(A1-v);
  \draw[e] (A1-v)--(A2-v);
  \draw[e] (A2-v)--(A3-v);
  \draw[d] (A3-v)--($(A3-v)+(1.0,0)$);
  \foreach \i in {1,2,3}{
    \node[slbl,above=3pt] at (A\i-v) {$v_{\i}$};
    \node[slbl,left=4pt]  at (A\i-l) {$\ell_{\i}$};
  }
  \node[slbl,left=4pt]  at (A1-r) {$\rho_1$};
  \node[slbl,above=3pt] at (A1-w) {$w_1$};
\end{scope}
\end{tikzpicture}

\captionof{figure}{The local weak limit $\hat{\mN}$ of
Definition~\ref{def:limit}. \textbf{(a)} The repeating unit, or cell.
\textbf{(b)} The network $\hat{\mN}$: root $o$ followed by the one-ended
decorated spine.}
\label{fig:limit}
\end{minipage}
\end{center}
\smallskip

The geometric reason is a mismatch between two scales. Since
\(K_n/\sqrt n \stackrel{p}{\to}1\), we may condition on \(K_n=k\) with
\(k\asymp\sqrt n\). The \(2(n-k)\asymp 2n\) subdivision vertices are then spread
over only \(2k-1\asymp\sqrt n\) edges of the top tree, so a typical inflated
edge has length of order \(\sqrt n\). Hence a fixed-radius exploration from the
root does not see the branching vertices of the top tree. Moreover, the
matching \(M\) chooses partners among \(\asymp n\) subdivision vertices, so the
reticulations encountered locally have their other parents far away with high
probability. Thus the local limit is the one-ended decorated spine
\(\hat{\mN}\) of Definition~\ref{def:limit} and Figure~\ref{fig:limit}.
\smallskip

%\begin{theorem}\label{te:local_limit}
%	There exists an infinite (deterministic) network $\hat{\mN}$ such that
%	\[
%		\mN_n \convdis \hat{\mN}
%	\]
%	in the local weak sense at the root.
%\end{theorem}

We denote by $U_R(N)$ the subnetwork of a rooted network $N$ induced by the
vertices at undirected graph distance at most $R$ from the root. Since
$\hat{\mN}$ is deterministic, the local weak convergence
$\mN_n\convdis\hat{\mN}$ reduces to showing that, for every fixed $R\ge 0$,
\begin{equation}\label{eq:crit}
  \Prb{U_R(\mN_n)\cong U_R(\hat{\mN})}\;\xrightarrow{n\to\infty}\;1,
\end{equation}
where $\cong$ denotes isomorphism of rooted directed marked graphs. Every underlying undirected graph considered here has maximum degree at
most~$3$ (for $\hat{\mN}$ by Definition~\ref{def:limit}; for $\mN_n$ because
it is bicombining). Hence any radius-$R$ ball contains at most
$N:=3\cdot 2^R$ vertices, whether in $\mN_n$ or in $\hat{\mN}$.

\begin{proof}[Proof of Theorem~\ref{te:local_limit}]
Fix \(R\ge1\). The intuitive picture is that, with probability tending to one,
the root edge is long enough for the radius \(R\) exploration to see only
subdivision vertices at first; these are the vertices inserted by the
composition \(Y\) in the Cardona--Zhang decomposition of
Proposition~\ref{pro:tt}. The same interior of an inflated edge (an edge with
subdivisions) picture should then hold from the other parent of every
reticulation encountered. The only possible obstruction is that such an other
parent lies too close to a part of the top tree already seen, in which case the
reticulation may close a cycle inside the radius-\(R\) ball; the estimates
below show that this obstruction has probability tending to zero.

\smallskip
By Proposition~\ref{pro:tt}, conditionally on \(K_n=k\) the network \(\mN_n\)
is the blow-up of a uniform triple \((T,Y,M)\), with \(T\in\cT_k\), \(Y\)
uniform on compositions of \(2j=2(n-k)\) into \(2k-1\) parts, and \(M\) a
uniform ordered matching of \([2j]\). Let \(\tau_n^{\star}\) denote the top tree
component of \(\mN_n\). By Corollary~\ref{co:ttsize}, \(\Prb{K_n\notin
I_n}\to0\), so it suffices to bound
\[
  \Prb{D_n\mid K_n=k}
\]
uniformly for \(k\in I_n\), where
\[
  D_n=\{U_R(\mN_n)\not\cong U_R(\hat{\mN})\}.
\]

\smallskip
Fix such \(k\) and explore \(U_R(\mN_n)\) by breadth-first search from \(o\) in
the underlying undirected graph, revealing \((T,Y,M)\) on demand. By the definition of \(N\), at most \(N\) vertices are revealed. Let \(G\) be the event that
\begin{itemize}\itemsep0pt
\item the root edge \(e_\star\) of \(T\) carries at least \(R\) subdivision
  vertices;
\item whenever a matching partner \(\pi\) is revealed, it is at distance \(>R\)
  in \(\tau_n^{\star}\) from the two vertices of \(T\) joined by the edge containing
  \(\pi\), and at distance \(>2R\) in \(\tau_n^{\star}\) from all subdivision vertices
  exposed before \(\pi\) is revealed.
\end{itemize}

\smallskip
We claim that \(G\subseteq D_n^c\). On \(G\), the exploration can be compared
step by step with the construction of \(\hat{\mN}\). The root condition gives
the initial string of subdivision vertices. When a reticulation is revealed
from a previously seen subdivision vertex \(u\), its other parent
\(\pi=M(u)\) starts a new top-tree string: the first separation condition keeps
this string inside the interior of an inflated edge for all vertices visible
within radius \(R\), while the second separation condition, with the margin
\(2R\), ensures that the part of this new string visible within radius \(R\) is
disjoint in \(\tau_n^{\star}\) from the top-tree strings already exposed in the
radius-\(R\) exploration. Thus each newly revealed reticulation cell contributes
a new copy of the local cell structure, exactly as in
Definition~\ref{def:limit}. Hence \(U_R(\mN_n)\cong U_R(\hat{\mN})\), and
therefore
\[
  \Prb{D_n\mid K_n=k}\le \Prb{G^c\mid K_n=k}.
\]

\smallskip
It remains to bound \(\Prb{G^c\mid K_n=k}\). Let \(m=2j=2(n-k)\) and
\(q=2k-1\). Since \(Y\) is uniform over weak compositions of \(m\) into \(q\)
parts, stars and bars gives, for \(a\ge0\),
\[
  \Prb{Y_{e_\star}=a\mid K_n=k}
  =
  \frac{\binom{m-a+q-2}{q-2}}{\binom{m+q-1}{q-1}} .
\]
The right-hand side is decreasing in \(a\), and therefore
\begin{align*}
  \Prb{Y_{e_\star}<R\mid K_n=k}
  \le
  R\, \cdot \Prb{Y_{e_\star}=0\mid K_n=k}
  &=
  R\, \cdot \frac{\binom{m+q-2}{q-2}}{\binom{m+q-1}{q-1}}  \\
  &=
  R\,\frac{q-1}{m+q-1}
  =
  O_R(k/n).
\end{align*}
This bound holds uniformly for \(k\in I_n\). Since \(k\in I_n\), it is
\(O_R(n^{-1/2})\).

\smallskip
Next suppose that the exploration has reached a subdivision vertex \(u\) whose
matching partner has not yet been revealed. Let \(\pi=M(u)\) be this partner,
and let \(\mathcal F_\pi\) be the exploration history just before revealing
\(\pi\). Let \(\cA_\pi\) be the set of subdivision vertices which are
still available as possible values of \(\pi\). Since \(M\) is a uniform
matching, conditionally on \(K_n=k\) and on \(\mathcal F_\pi\), the vertex
\(\pi\) is uniform on \(\cA_\pi\). Moreover, before revealing \(\pi\),
at most \(2N\) subdivision vertices have been ruled out by the matching
information already exposed. Hence
\[
  |\cA_\pi|\ge 2j-2N .
\]

\smallskip
Let \(S_\pi\) be the set of subdivision vertices exposed before \(\pi\) is
revealed, and let \(B_\pi\) be the event that \(\pi\) fails one of the two
separation conditions in the definition of \(G\). If \(\mathcal B_\pi\subseteq
\cA_\pi\) denotes the set of available choices of \(\pi\) for which
\(B_\pi\) occurs, then
\[
  \Prb{B_\pi\mid K_n=k,\mathcal F_\pi}
  =
  \frac{|\mathcal B_\pi|}{|\cA_\pi|}.
\]

\smallskip
We now bound \(|\mathcal B_\pi|\) from above. A bad choice is of one of two
types. First, it may lie within distance \(R\), in \(\tau_n^{\star}\), of one of the
two vertices of \(T\) joined by the edge containing it. Since \(T\) has
\(2k-1\) edges, and each edge contributes at most \(2R\) such subdivision
vertices, the number of choices of this first type is at most \(2R(2k-1)\).
Second, it may lie within distance \(2R\), in \(\tau_n^{\star}\), of \(S_\pi\). Since
\(|S_\pi|\le N\) and \(\tau_n^{\star}\) has maximum degree at most \(3\), the
\(2R\)-neighbourhood of \(S_\pi\) in \(\tau_n^{\star}\) contains at most
\(C_R \leq N(3\cdot 2^{2R}-2)\) subdivision vertices. Therefore
\[
  |\mathcal B_\pi|
  \le
  2R(2k-1)+C_R .
\]
Combining the numerator and denominator bounds, with the equality $j=n-k$, gives
\[
  \Prb{B_\pi\mid K_n=k,\mathcal F_\pi}
  \le
  \frac{2R(2k-1)+C_R}{2j-2N}
  =
  O_R(k/n),
\]
uniformly for \(k\in I_n\).

\smallskip
At most \(N\) matching partners are revealed during the exploration. Thus,
by taking expectations and using a union bound over these partners, and then
adding the preceding estimate for the root edge, we obtain
\[
  \Prb{G^c\mid K_n=k}
  =
  O_R(k/n)
  =
  O_R(n^{-1/2}),
\]
uniformly for \(k\in I_n\). Since \(G\subseteq D_n^c\), it follows that
\[
  \Prb{D_n\mid K_n=k}=O_R(n^{-1/2})
\]
uniformly for \(k\in I_n\).

\smallskip
Finally, removing the conditioning gives
\begin{align*}
  \Prb{D_n}
  &=
  \sum_k \Prb{K_n=k}\Prb{D_n\mid K_n=k} \\
  &\le
  \Prb{K_n\notin I_n}
  +
  \sum_{k\in I_n}\Prb{K_n=k}\,O_R(n^{-1/2}) \\
  &=
  \Prb{K_n\notin I_n}
  +
  O_R(n^{-1/2})
  \xrightarrow[n\to\infty]{}0.
\end{align*}
Since \(R\) was arbitrary, \(\mN_n\convdis\hat{\mN}\) in the local weak sense.
\end{proof}
\subsection{Local limit at a uniform tree vertex}
\label{ssec:unary}

The decorated chain $\hat{\cL}$ was introduced in Definition~\ref{def:limit}:
it is the bi-infinite analogue of $\hat{\mN}$, a two-sided spine
$\cdots\to p_{-1}\to p_0\to p_1\to\cdots$ of cells rooted at $p_0$.

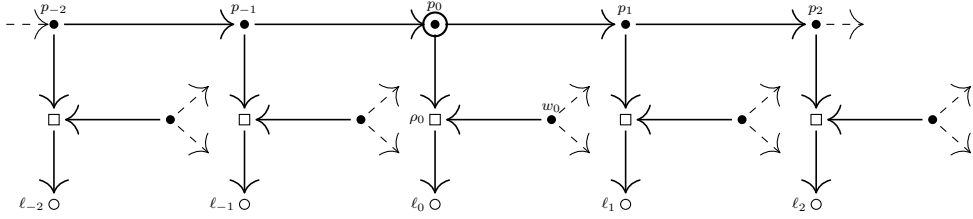
\begin{figure}[!htbp]
\centering
\begin{tikzpicture}[x=1cm,y=1cm,scale=0.7,transform shape]
\def\dx{3.6}
\draw[d] ($(-1*\dx,0)+(-1.0,0)$)--(-1*\dx,0);
\pic (Am2) at (-1*\dx,0) {cell};
\pic (Am1) at (0,0)      {cell};
\pic (A0)  at (1*\dx,0)  {cellM};
\pic (Ap1) at (2*\dx,0)  {cell};
\pic (Ap2) at (3*\dx,0)  {cell};
\draw[e] (Am2-v)--(Am1-v);
\draw[e] (Am1-v)--(A0-v);
\draw[e] (A0-v)--(Ap1-v);
\draw[e] (Ap1-v)--(Ap2-v);
\draw[d] (Ap2-v)--($(Ap2-v)+(1.0,0)$);
\node[slbl,above=4pt] at (Am2-v) {$p_{-2}$};
\node[slbl,above=4pt] at (Am1-v) {$p_{-1}$};
\node[slbl,above=6pt] at (A0-v)  {$p_0$};
\node[slbl,above=4pt] at (Ap1-v) {$p_1$};
\node[slbl,above=4pt] at (Ap2-v) {$p_2$};
\foreach \nm/\sub in {Am2/{-2},Am1/{-1},A0/0,Ap1/1,Ap2/2}{
  \node[slbl,left=4pt] at (\nm-l) {$\ell_{\sub}$};
}
\node[slbl,left=4pt]  at (A0-r) {$\rho_0$};
\node[slbl,above=3pt] at (A0-w) {$w_0$};
\end{tikzpicture}
\caption{The decorated chain $\hat{\cL}$: a bi-infinite spine of cells
$\cdots\to p_{-1}\to p_0\to p_1\to\cdots$, rooted at the marked vertex
$p_0$ (highlighted by a ring).}
\label{fig:Lhat}
\end{figure}

The idea behind this local limit is again a mismatch of scales.
Conditionally on $K_n=k$ with $k\asymp\sqrt n$, the $2(n-k)\asymp 2n$
subdivision tree vertices are distributed along the $2k-1\asymp\sqrt n$
edges of $T$, so a uniform subdivision vertex $u_n$ lies in the bulk of an
edge of $T$ rather than near one of its endpoints with high probability:
at undirected distance $\ge R$ from any branching vertex of $T$, with
probability $1-O(R/\sqrt n)$. Locally around $u_n$ the network therefore
looks like a bi-infinite chain of subdivision vertices, each carrying a
reticulation with a pendant leaf, with partners pushed to infinity by the
same union-bound argument as before. This is the structure of the
decorated chain $\hat{\cL}$ rooted at $p_0$.

%\begin{theorem}\label{te:loc-unary}
%On the event that $\mN_n$ has at least one unary tree vertex, let $u_n$ be
%chosen uniformly at random among the unary tree vertices of $\mN_n$.
%Then $(\mN_n,u_n)\convdis(\hat{\cL},p_0)$ in the local weak sense.
%\end{theorem}

\begin{proof}[Proof of Theorem~\ref{te:loc-unary}]
Let \(u_n\) be chosen uniformly among all tree vertices of \(\mN_n\).
The network contains \(2(n-K_n)\) subdivision tree vertices, whereas the
number of remaining tree vertices is \(O(K_n)\). Since
\(K_n/\sqrt n\convp1\), it follows that
\[
\Prb{u_n\text{ is not a subdivision tree vertex}}\longrightarrow0.
\]
Consequently, it suffices to prove the result when \(u_n\) is chosen
uniformly among the subdivision tree vertices.
\smallskip

Fix \(R\ge1\). The intuitive picture is that a uniformly chosen subdivision
tree vertex lies, with probability tending to one, deep inside an inflated edge
of the top tree. Thus the radius-\(R\) exploration should first see a two-sided
string of subdivision vertices. The difference from the root case is that this
initial top-tree string is now two-sided; after this initial change, the
reticulation partners are controlled by the same separation estimate as before.

\smallskip
Write
\[
  D_n=\{U_R(\mN_n,u_n)\not\cong U_R(\hat{\cL},p_0)\}.
\]
By Proposition~\ref{pro:tt}, conditionally on \(K_n=k\), the network
\(\mN_n\) is the blow-up of a uniform triple \((T,Y,M)\), with
\(T\in\cT_k\), the random variable \(Y\) is uniform on compositions of \(2j=2(n-k)\) into \(2k-1\)
parts, and \(M\) is a uniform ordered matching of \([2j]\). Let \(\tau_n^{\star}\)
denote the top tree component of \(\mN_n\). Conditionally on \(K_n=k\),
\(u_n\) is uniform among the \(2j\) subdivision vertices. By
Corollary~\ref{co:ttsize}, \(\Prb{K_n\notin I_n}\to0\), so it suffices to
bound \(\Prb{D_n\mid K_n=k}\) uniformly for \(k\in I_n\).

\smallskip
Fix such \(k\), and explore \(U_R(\mN_n,u_n)\) by breadth-first search,
revealing \((T,Y,M)\) on demand. At most \(N\) vertices are revealed. Let
\(G\) be the event that
\begin{itemize}\itemsep0pt
\item \(u_n\) is at distance \(>R\) in \(\tau_n^{\star}\) from the two vertices of
  \(T\) joined by the edge containing \(u_n\);
\item whenever a matching partner \(\pi\) is revealed, it is at distance \(>R\)
  in \(\tau_n^{\star}\) from the two vertices of \(T\) joined by the edge containing
  \(\pi\), and at distance \(>2R\) in \(\tau_n^{\star}\) from all subdivision vertices
  exposed before \(\pi\) is revealed.
\end{itemize}
On \(G\), the exploration sees the decorated bi-infinite chain locally, hence
\(G\subseteq D_n^c\).

\smallskip
The number of subdivision vertices within distance \(R\), in \(\tau_n^{\star}\), of
one of the two vertices of \(T\) joined by their containing edge is at most
\(2R(2k-1)\). Since \(u_n\) is uniform among the \(2j\) subdivision vertices,
\[
  \Prb{u_n\text{ fails the first condition of }G\mid K_n=k}
  \le
  \frac{2R(2k-1)}{2j}
  =
  O_R(k/n)
  =
  O_R(n^{-1/2})
\]
uniformly for \(k\in I_n\).

\smallskip
For the later matching partners, the same estimate as in the proof of
Theorem~\ref{te:local_limit} gives
\[
  \Prb{B_\pi\mid K_n=k,\mathcal F_\pi}
  \le
  \frac{2R(2k-1)+C_R}{2j-2N}
  =
  O_R(k/n),
\]
uniformly for \(k\in I_n\). A union bound over the at most \(N\) revealed
partners gives
\[
  \Prb{G^c\mid K_n=k}=O_R(k/n)=O_R(n^{-1/2}).
\]
Thus \(\Prb{D_n\mid K_n=k}=O_R(n^{-1/2})\) uniformly for \(k\in I_n\), and
removing the conditioning gives
\[
  \Prb{D_n}\le \Prb{K_n\notin I_n}+O_R(n^{-1/2})\to0.
\]
Since \(R\) was arbitrary, \((\mN_n,u_n)\convdis(\hat{\cL},p_0)\).
\end{proof}

\subsection{Local limit at a uniform reticulation}
\label{ssec:ret}

We now root the network at a uniformly chosen reticulation. The argument is the
same exploration argument as before, except that the initial neighbourhood has
two top-tree directions, corresponding to the two parents of the marked
reticulation.

\begin{definition}[The reticulation limit $\hat{\cR}$]\label{def:Rhat}
The \emph{reticulation limit} $\hat{\cR}$ is the infinite directed network
rooted at a marked reticulation $\rho_0$, defined as follows. It has two
disjoint bi-infinite spines of tree vertices,
\[
  \cdots\to p_{-1}\to p_0\to p_1\to\cdots,
  \qquad
  \cdots\to q_{-1}\to q_0\to q_1\to\cdots,
\]
and a marked reticulation $\rho_0$ whose two parents are $p_0$ and $q_0$
and whose only child is a leaf $\ell_0$. Every spine vertex other than
$p_0$ and $q_0$ carries a cell as in Definition~\ref{def:limit}: a
reticulation child with a leaf child, whose second parent lies on yet
another bi-infinite decorated spine. The marked vertices $p_0$ and $q_0$
have no cell of their own; their reticulation child is the shared
$\rho_0$.
\end{definition}

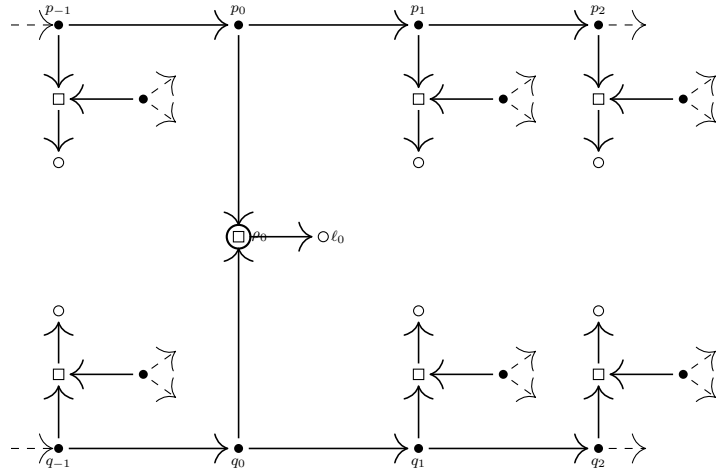
\begin{figure}[!htbp]
\centering
\begin{tikzpicture}[x=1cm,y=1cm,scale=0.7,transform shape]
\def\dx{3.4}
\def\ytop{2.5}
\def\ybot{-5.5}
% Top spine
\draw[d] ($(-1*\dx,\ytop)+(-1.0,0)$)--(-1*\dx,\ytop);
\foreach \i/\nm in {-1/Tm1,0/T0,1/Tp1,2/Tp2}{
  \node[tv] (\nm) at (\i*\dx,\ytop) {};
}
\node[slbl,above=4pt] at (Tm1) {$p_{-1}$};
\node[slbl,above=4pt] at (T0)  {$p_0$};
\node[slbl,above=4pt] at (Tp1) {$p_1$};
\node[slbl,above=4pt] at (Tp2) {$p_2$};
\draw[e] (Tm1)--(T0); \draw[e] (T0)--(Tp1); \draw[e] (Tp1)--(Tp2);
\draw[d] (Tp2)--($(Tp2)+(1.0,0)$);
\foreach \nm in {Tm1,Tp1,Tp2}{
  \node[rt] (\nm-r) at ($(\nm)+(0,-1.4)$) {};
  \node[lf] (\nm-l) at ($(\nm)+(0,-2.6)$) {};
  \node[tv] (\nm-w) at ($(\nm)+(1.6,-1.4)$) {};
  \draw[e] (\nm)--(\nm-r);
  \draw[e] (\nm-w)--(\nm-r);
  \draw[e] (\nm-r)--(\nm-l);
  \draw[d] (\nm-w)--($(\nm-w)+(0.7, 0.5)$);
  \draw[d] (\nm-w)--($(\nm-w)+(0.7,-0.5)$);
}
\draw[d] ($(-1*\dx,\ybot)+(-1.0,0)$)--(-1*\dx,\ybot);
\foreach \i/\nm in {-1/Bm1,0/B0,1/Bp1,2/Bp2}{
  \node[tv] (\nm) at (\i*\dx,\ybot) {};
}
\node[slbl,below=4pt] at (Bm1) {$q_{-1}$};
\node[slbl,below=4pt] at (B0)  {$q_0$};
\node[slbl,below=4pt] at (Bp1) {$q_1$};
\node[slbl,below=4pt] at (Bp2) {$q_2$};
\draw[e] (Bm1)--(B0); \draw[e] (B0)--(Bp1); \draw[e] (Bp1)--(Bp2);
\draw[d] (Bp2)--($(Bp2)+(1.0,0)$);
\foreach \nm in {Bm1,Bp1,Bp2}{
  \node[rt] (\nm-r) at ($(\nm)+(0,1.4)$) {};
  \node[lf] (\nm-l) at ($(\nm)+(0,2.6)$) {};
  \node[tv] (\nm-w) at ($(\nm)+(1.6,1.4)$) {};
  \draw[e] (\nm)--(\nm-r);
  \draw[e] (\nm-w)--(\nm-r);
  \draw[e] (\nm-r)--(\nm-l);
  \draw[d] (\nm-w)--($(\nm-w)+(0.7, 0.5)$);
  \draw[d] (\nm-w)--($(\nm-w)+(0.7,-0.5)$);
}
\node[rt] (rho0) at (0,-1.5) {};
\node[mkring] at (rho0) {};
\node[slbl,right=6pt] at (rho0) {$\rho_0$};
\node[lf] (l0) at (1.6,-1.5) {};
\node[slbl,right=3pt] at (l0) {$\ell_0$};
\draw[e] (T0)--(rho0);
\draw[e] (B0)--(rho0);
\draw[e] (rho0)--(l0);
\end{tikzpicture}
\caption{The reticulation limit $\hat{\cR}$: two disjoint decorated spines
meet at the marked reticulation $\rho_0$ (ring) through its two
parents $p_0$ and $q_0$; $\rho_0$ carries a pendant leaf $\ell_0$.}
\label{fig:Rhat}
\end{figure}

The phenomenon is the same as before, now seen from a reticulation: the
$\asymp n$ subdivision vertices are spread over only $\asymp\sqrt n$ edges
of $T$. Conditionally on $K_n=k$ with $k\asymp\sqrt n$, the network
$\mN_n$ has $j=n-k\asymp n$ reticulations, each attached to two
subdivision vertices of $T$ chosen uniformly by $M$. A uniformly chosen
reticulation $r_n$ therefore has two parents $p,q$ which, by
exchangeability of $M$, sit at uniform positions among the $2j$
subdivision vertices and, with probability $1-O(1/\sqrt n)$, lie deep
inside their containing edges of $T$ and at $\tau_n^{\star}$-distance larger than
$2R$ from each other. Locally around $r_n$ the network therefore looks
like a reticulation with a pendant leaf and two parents lying on two
disjoint bi-infinite chains of subdivision vertices, each decorated as in
the previous cases. This is the structure of $\hat{\cR}$ rooted at the reticulation node
$\rho_0$.

%\begin{theorem}\label{te:loc-ret}
%On the event that $\mN_n$ has at least one reticulation vertex, let $r_n$ be
%chosen uniformly at random among them. Then %$(\mN_n,r_n)\convdis(\hat{\cR},\rho_0)$
%in the local weak sense.
%\end{theorem}

\begin{proof}[Proof of Theorem~\ref{te:loc-ret}]
Fix \(R\ge1\). The intuitive picture is now the same, except that the
exploration starts from a reticulation. Thus we first have to check that its
two parents lie deep inside inflated edges of the top tree and far from each
other. Once this holds, every further reticulation encountered is handled by
the same separation estimate as before.

Write
\[
  D_n=\{U_R(\mN_n,r_n)\not\cong U_R(\hat{\cR},\rho_0)\}.
\]
Conditionally on \(K_n=k\), realise \(\mN_n\) as the blow-up of \((T,Y,M)\).
We realise \(r_n\) by choosing one pair of \(M\) uniformly. Let \(p\) and \(q\)
be the two subdivision vertices in this pair, and let \(\tau_n^{\star}\) denote the top
tree component. By Corollary~\ref{co:ttsize}, it suffices to work uniformly on
\(k\in I_n\).

\smallskip
Fix such \(k\), and explore \(U_R(\mN_n,r_n)\). At most \(N\) vertices are
revealed. Let \(G\) be the event that
\begin{itemize}\itemsep0pt
\item \(p\) and \(q\) are each at distance \(>R\) in \(\tau_n^{\star}\) from the two
  vertices of \(T\) joined by their containing edge;
\item \(d_{\tau_n^{\star}}(p,q)>2R\);
\item whenever, after the central pair \(\{p,q\}\), a matching partner \(\pi\)
  is revealed, it satisfies the same two separation conditions as in the proof
  of Theorem~\ref{te:local_limit}.
\end{itemize}
On \(G\), the two parents \(p\) and \(q\) of \(r_n\) start two disjoint
decorated spines, sharing the marked reticulation \(r_n\) as their common
child. All later reticulations open new disjoint decorated spines. Hence
\(G\subseteq D_n^c\).

\smallskip
Let \(A\) be the set of subdivision vertices lying within distance \(R\), in
\(\tau_n^{\star}\), of one of the two vertices of \(T\) joined by their containing
edge. Then
\[
  |A|\le 2R(2k-1).
\]
Since \(\{p,q\}\) is a uniformly chosen pair of subdivision vertices,
\[
  \Prb{p\in A\text{ or }q\in A\mid K_n=k}=O_R(k/n).
\]
Moreover, conditionally on \(p\), the vertex \(q\) is uniform among the
remaining \(2j-1\) subdivision vertices, and the \(2R\)-neighbourhood of \(p\)
in \(\tau_n^{\star}\) contains at most \(C_R\) subdivision vertices. Therefore
\[
  \Prb{d_{\tau_n^{\star}}(p,q)\le2R\mid K_n=k}=O_R(1/n).
\]

\smallskip
For all later partners, the same estimate as in Theorem~\ref{te:local_limit}
gives
\[
  \Prb{B_\pi\mid K_n=k,\mathcal F_\pi}
  \le
  \frac{2R(2k-1)+C_R}{2j-2N}
  =
  O_R(k/n),
\]
uniformly for \(k\in I_n\). A union bound over at most \(N\) revealed
partners gives
\[
  \Prb{G^c\mid K_n=k}=O_R(k/n)=O_R(n^{-1/2}).
\]
Thus \(\Prb{D_n\mid K_n=k}=O_R(n^{-1/2})\) uniformly for \(k\in I_n\), and
removing the conditioning gives
\[
  \Prb{D_n}\le \Prb{K_n\notin I_n}+O_R(n^{-1/2})\to0.
\]
Since \(R\) was arbitrary, \((\mN_n,r_n)\convdis(\hat{\cR},\rho_0)\).
\end{proof}

\subsection{Local limit at a uniform leaf}\label{ssec:leaf}

We next root the network at a uniformly chosen leaf. Since almost all leaves are
reticulation-leaves, this limit is obtained from the reticulation limit by
moving the root one step down to the pendant leaf.

\begin{definition}[The leaf limit $\hat{\mN}^{\mathrm{lf}}$]\label{def:Lfhat}
The \emph{leaf limit} $\hat{\mN}^{\mathrm{lf}}$ is the infinite directed
network rooted at a marked leaf $\ell_0$, obtained from $\hat{\cR}$ by
re-rooting at $\ell_0$, the leaf child of the marked reticulation $\rho_0$.
\end{definition}

\begin{figure}[!htbp]
\centering
\begin{tikzpicture}[x=1cm,y=1cm,scale=0.7,transform shape]
\def\dx{3.4}
\def\ytop{2.5}
\def\ybot{-5.5}
% Top spine
\draw[d] ($(-1*\dx,\ytop)+(-1.0,0)$)--(-1*\dx,\ytop);
\foreach \i/\nm in {-1/Tm1,0/T0,1/Tp1,2/Tp2}{
  \node[tv] (\nm) at (\i*\dx,\ytop) {};
}
\node[slbl,above=4pt] at (Tm1) {$p_{-1}$};
\node[slbl,above=4pt] at (T0)  {$p_0$};
\node[slbl,above=4pt] at (Tp1) {$p_1$};
\node[slbl,above=4pt] at (Tp2) {$p_2$};
\draw[e] (Tm1)--(T0); \draw[e] (T0)--(Tp1); \draw[e] (Tp1)--(Tp2);
\draw[d] (Tp2)--($(Tp2)+(1.0,0)$);
\foreach \nm in {Tm1,Tp1,Tp2}{
  \node[rt] (\nm-r) at ($(\nm)+(0,-1.4)$) {};
  \node[lf] (\nm-l) at ($(\nm)+(0,-2.6)$) {};
  \node[tv] (\nm-w) at ($(\nm)+(1.6,-1.4)$) {};
  \draw[e] (\nm)--(\nm-r);
  \draw[e] (\nm-w)--(\nm-r);
  \draw[e] (\nm-r)--(\nm-l);
  \draw[d] (\nm-w)--($(\nm-w)+(0.7, 0.5)$);
  \draw[d] (\nm-w)--($(\nm-w)+(0.7,-0.5)$);
}
% Bottom spine
\draw[d] ($(-1*\dx,\ybot)+(-1.0,0)$)--(-1*\dx,\ybot);
\foreach \i/\nm in {-1/Bm1,0/B0,1/Bp1,2/Bp2}{
  \node[tv] (\nm) at (\i*\dx,\ybot) {};
}
\node[slbl,below=4pt] at (Bm1) {$q_{-1}$};
\node[slbl,below=4pt] at (B0)  {$q_0$};
\node[slbl,below=4pt] at (Bp1) {$q_1$};
\node[slbl,below=4pt] at (Bp2) {$q_2$};
\draw[e] (Bm1)--(B0); \draw[e] (B0)--(Bp1); \draw[e] (Bp1)--(Bp2);
\draw[d] (Bp2)--($(Bp2)+(1.0,0)$);
\foreach \nm in {Bm1,Bp1,Bp2}{
  \node[rt] (\nm-r) at ($(\nm)+(0,1.4)$) {};
  \node[lf] (\nm-l) at ($(\nm)+(0,2.6)$) {};
  \node[tv] (\nm-w) at ($(\nm)+(1.6,1.4)$) {};
  \draw[e] (\nm)--(\nm-r);
  \draw[e] (\nm-w)--(\nm-r);
  \draw[e] (\nm-r)--(\nm-l);
  \draw[d] (\nm-w)--($(\nm-w)+(0.7, 0.5)$);
  \draw[d] (\nm-w)--($(\nm-w)+(0.7,-0.5)$);
}
% Central reticulation rho_0 with MARKED leaf ell_0
\node[rt] (rho0) at (0,-1.5) {};
\node[slbl,left=3pt] at (rho0) {$\rho_0$};
\node[lf] (l0) at (1.6,-1.5) {};
\node[mkring] at (l0) {};
\node[slbl,right=6pt] at (l0) {$\ell_0$};
\draw[e] (T0)--(rho0);
\draw[e] (B0)--(rho0);
\draw[e] (rho0)--(l0);
\end{tikzpicture}
\caption{The leaf limit $\hat{\mN}^{\mathrm{lf}}$: the same underlying graph as
$\hat{\cR}$, but rooted at the marked leaf $\ell_0$ (ring) instead of at
$\rho_0$.}
\label{fig:Lfhat}
\end{figure}
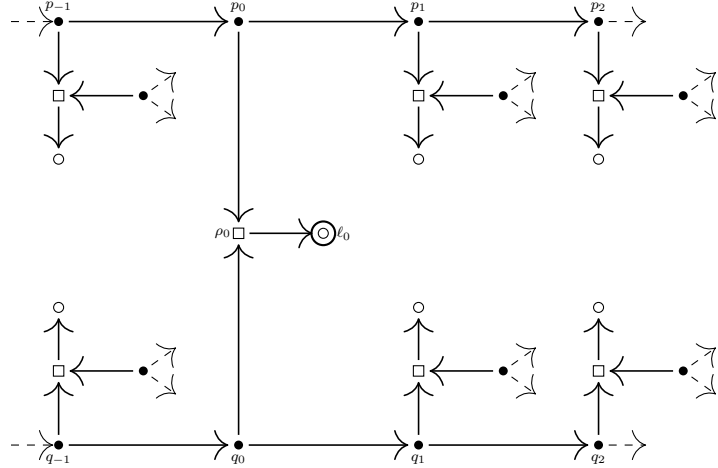

The picture is that of the reticulation case (Theorem~\ref{te:loc-ret}),
one step further down. The $K_n\asymp\sqrt n$ tree-leaves represent a
vanishing fraction of all $n$ leaves, so a uniformly chosen leaf $w_n$ is
a reticulation-leaf with probability $1-O(1/\sqrt n)$. On this event, its
parent $\rho_n$ is a uniform reticulation of $\mN_n$, whose two parents
lie on two disjoint decorated bi-infinite chains by
Theorem~\ref{te:loc-ret}. This is the structure of
$\hat{\mN}^{\mathrm{lf}}$ rooted at the leaf $\ell_0$, where $\rho_0$ is
the marked reticulation of $\hat{\cR}$ and $p_0,q_0$ are its two parents.

%\begin{theorem}\label{te:loc-leaf}
%Let $w_n$ be a vertex of $\mN_n$ chosen uniformly at random among its leaves.
%Then $(\mN_n,w_n)\convdis(\hat{\mN}^{\mathrm{lf}},\ell_0)$ in the local weak
%sense.
%\end{theorem}
\begin{proof}[Proof of Theorem~\ref{te:loc-leaf}]
Fix \(R\ge1\). The intuitive picture is that a uniformly chosen leaf is,
with probability tending to one, not a tree-leaf but the leaf child of a
reticulation. Thus the leaf limit is the reticulation limit viewed one step
below the marked reticulation.
\smallskip
Write
\[
  D_n=\{U_R(\mN_n,w_n)\not\cong
  U_R(\hat{\mN}^{\mathrm{lf}},\ell_0)\},
  \qquad
  E_n=\{w_n\text{ is a reticulation-leaf}\}.
\]
Conditionally on \(K_n=k\), the network \(\mN_n\) has \(k\) tree-leaves and
\(j=n-k\) reticulation-leaves. Hence, uniformly for \(k\in I_n\),
\[
  \Prb{E_n^c\mid K_n=k}
  =
  \frac{k}{k+j}
  =
  \frac{k}{n}
  =
  O(n^{-1/2}).
\]
Combining with $\Prb{K_n\notin I_n}\to 0$ gives $\Prb{E_n^c}=o(1)$.
Thus, with probability tending to one, \(w_n\) is
the leaf child of a reticulation. On this event, its parent is a uniformly
chosen reticulation vertex (see below). Therefore, by
Theorem~\ref{te:loc-ret}, the radius-\(R\) neighbourhood of \(w_n\) converges
to the radius-\(R\) neighbourhood of the leaf child \(\ell_0\) of the marked
reticulation in \(\hat{\cR}\), namely \((\hat{\mN}^{\mathrm{lf}},\ell_0)\).
\smallskip

More explicitly, on $E_n$, let $\rho_n$ denote the unique parent of $w_n$.
Since each of the $j$ reticulations of $\mN_n$ has exactly one leaf child,
the map $\rho\mapsto\ell(\rho)$ is a bijection between the $j$ reticulations
and the $j$ reticulation-leaves; consequently, conditionally on $E_n$,
$\rho_n$ is uniform on the reticulation vertices of $\mN_n$. Moreover,
the rooted neighbourhood \(U_R(\mN_n,w_n)\) is obtained from
\(U_{R-1}(\mN_n,\rho_n)\) by adjoining the leaf \(w_n\) and the edge
\(\rho_n\to w_n\), and then re-rooting at \(w_n\). Let \(Q_n\) denote the proportion of reticulation vertices \(\rho\) of
\(\mN_n\) for which
\[
U_{R-1}(\mN_n,\rho)\not\cong
U_{R-1}(\hat{\cR},\rho_0).
\]
If \(\mN_n\) has no reticulation vertices, we set \(Q_n:=0\); on that event
\(n-K_n=0\), so this convention leaves the identities below unchanged. Conditionally on \(\mN_n\), we have
\[
\Prb{\left.
E_n\cap
\{U_{R-1}(\mN_n,\rho_n)\not\cong
U_{R-1}(\hat{\cR},\rho_0)\}
\,\right|\,\mN_n}
=
\frac{n-K_n}{n}\,Q_n
\le Q_n.
\]
Let \(r_n\) be a uniformly chosen reticulation vertex of \(\mN_n\).
Conditionally on \(\mN_n\), the uniform choice of \(r_n\) gives
\[
\Prb{\left.
U_{R-1}(\mN_n,r_n)\not\cong
U_{R-1}(\hat{\cR},\rho_0)
\,\right|\,\mN_n}
=
Q_n.
\]

Hence, by taking
expectations,
\[
\Ex{Q_n}
=
\Prb{
U_{R-1}(\mN_n,r_n)\not\cong
U_{R-1}(\hat{\cR},\rho_0)
}.
\]
Theorem~\ref{te:loc-ret} states that
\((\mN_n,r_n)\convdis(\hat{\cR},\rho_0)\) in the local weak sense.
Consequently, for the fixed radius \(R-1\), the probability on the
right-hand side tends to \(0\), and therefore
\[
\Ex{Q_n}\longrightarrow0.
\] In particular this implies that
\[
\Prb{D_n}
\le
\Prb{E_n^c}+\Ex{Q_n}
\longrightarrow0.
\] Since $R$ was arbitrary,
$(\mN_n,w_n)\convdis(\hat{\mN}^{\mathrm{lf}},\ell_0)$.
\end{proof}

\subsection{Local limit at a uniform vertex}\label{ssec:vertex}
We now combine the local limits obtained for the three vertex types.
The numbers of subdivision tree vertices, reticulations, and leaves have
asymptotic proportions \(1/2\), \(1/4\), and \(1/4\), respectively, among
all vertices of \(\mN_n\). Consequently, the local limit around a uniformly
chosen vertex is the corresponding mixture of the three type-specific
limits. We now prove Theorem \ref{te:loc-vertex}.

%\begin{theorem}\label{te:loc-vertex}
%Let $v_n$ be a vertex of $\mN_n$ chosen uniformly at random among all its
%vertices. Then $(\mN_n,v_n)\convdis(N^*,v^*)$, where
%\[
%  (N^*,v^*)\;\eqdist\;\begin{cases}
%    (\hat{\cL},\,p_0) & \text{with probability } 1/2,\\[2pt]
%    (\hat{\cR},\,\rho_0) & \text{with probability } 1/4,\\[2pt]
%    (\hat{\mN}^{\mathrm{lf}},\,\ell_0) & \text{with probability } 1/4.
%  \end{cases}
%\]
%The weights are the asymptotic proportions of subdivision tree vertices,
%reticulations and leaves among the $4n-2K_n$ vertices of $\mN_n$; the planted
%root and the $K_n-1$ binary tree vertices have total proportion
%$O(K_n/n)=O(n^{-1/2})$.
%\end{theorem}

\begin{proof}[Proof of Theorem \ref{te:loc-vertex}]
The result follows by conditioning on the type of the uniformly chosen vertex.
Conditionally on \(K_n=k\), the network has \(j=n-k\) reticulations and \(2j\)
subdivision tree vertices. It also has \(n\) leaves in total, split into \(k\)
tree-leaves and \(n-k\) reticulation-leaves. The remaining vertices are the
\(k-1\) binary tree vertices of the reduced tree \(T\), and the planted root.
Thus the vertex set decomposes as
\[
\begin{array}{c|c}
\text{type of vertex} & \text{number of such vertices} \\ \hline
\text{subdivision tree vertices} & 2(n-k) = 2j \\
\text{reticulation vertices} & n-k \\
\text{leaves} & n \\
\text{binary tree vertices of }T & k-1 \\
\text{planted root} & 1
\end{array}
\]
and the total number of vertices is
\[
  2(n-k)+(n-k)+n+(k-1)+1=4n-2k.
\]
\smallskip
Since \(K_n/\sqrt n\to 1\) in probability (Corollary~\ref{co:ttsize}), in
particular \(K_n/n\to 0\) in probability. By the continuous mapping theorem,
the proportions of the first three classes converge in probability to
\[
  \frac{2(n-K_n)}{4n-2K_n}\to \frac12,\qquad
  \frac{n-K_n}{4n-2K_n}\to \frac14,\qquad
  \frac{n}{4n-2K_n}\to \frac14,
\]
whereas
\[
  \frac{K_n-1}{4n-2K_n}\to0,
  \qquad
  \frac{1}{4n-2K_n}\to0.
\]
Since these proportions are bounded by $1$, convergence in probability
upgrades to convergence in $L^1$ by bounded convergence.

\smallskip
By Definition~\ref{def:lwc-marked}, local weak convergence amounts to showing
that, for every fixed \(R\ge 0\) and every finite rooted directed marked
graph $G$,
\[
  \Prb{U_R(\mN_n,v_n)\cong G}\;\longrightarrow\;\Prb{U_R(N^*,v^*)\cong G},
\]
with \((N^*,v^*)\) the announced mixture. Decomposing according to the type of
\(v_n\) and using exchangeability within each of the first three classes
(with the tree-leaf vs reticulation-leaf subdivision in the leaf class
absorbed by Theorem~\ref{te:loc-leaf}), the conditional probability
\(\Prb{U_R(\mN_n,v_n)\cong G\mid v_n\text{ is a subdivision vertex}}\) converges
to \(\Prb{U_R(\hat{\cL},p_0)\cong G}\) by Theorem~\ref{te:loc-unary}, and
analogously for the reticulation and leaf cases via
Theorems~\ref{te:loc-ret} and~\ref{te:loc-leaf} (the leaf case absorbing the
\(O(1/\sqrt n)\) tree-leaf contribution into its error term). The contribution
from the remaining types (binary tree vertices of \(T\) and the planted root)
is bounded by
\begin{align*}
  \Prb{v_n\text{ is a binary tree vertex or the planted root}}
  &=\Exb{\frac{K_n}{4n-2K_n}}
  \\ &\le\frac{\Ex{K_n}}{2n} 
  =O(n^{-1/2}),
\end{align*}
since $\Ex{K_n}=O(\sqrt n)$ by Corollary~\ref{co:Knmoments}. Combining with the
proportion limits above yields
\begin{multline*}
  \Prb{U_R(\mN_n,v_n)\cong G}\;\longrightarrow\;
  \tfrac12\Prb{U_R(\hat{\cL},p_0)\cong G}\\
  +\tfrac14\Prb{U_R(\hat{\cR},\rho_0)\cong G}
  +\tfrac14\Prb{U_R(\hat{\mN}^{\mathrm{lf}},\ell_0)\cong G},
\end{multline*}
so the local weak limit is the stated mixture:
\[
  (N^*,v^*)\;\eqdist\;\begin{cases}
    (\hat{\cL},\,p_0) & \text{with probability } 1/2,\\[2pt]
    (\hat{\cR},\,\rho_0) & \text{with probability } 1/4,\\[2pt]
    (\hat{\mN}^{\mathrm{lf}},\,\ell_0) & \text{with probability } 1/4.
  \end{cases}
\]
This concludes the proof.
\end{proof}

\subsection{Quenched local convergence}
\label{ssec:quenched}

The annealed limits above upgrade to quenched local convergence in the sense of Theorem \ref{te:loc-quenched}. We now provide the proof of this result. 

%\begin{theorem}\label{te:loc-quenched}
%For every fixed $R\ge0$ and every finite rooted directed marked graph $G$,
%\[
%  \frac{\#\{u\in V(\mN_n):\,u\text{ unary tree vertex},\,
%        U_R(\mN_n,u)\cong G\}}
%       {\#\{\text{unary tree vertices of }\mN_n\}}
%  \;\convp\;\Prb{U_R(\hat{\cL},p_0)\cong G},
%\]
%In each ratio, if the denominator is zero, the ratio is defined to be $0$. %The analogous statements hold for reticulations (with limit $\hat{\cR}$) and
%for leaves (with limit $\hat{\mN}^{\mathrm{lf}}$).
%\end{theorem}

\begin{proof}[Proof of Theorem \ref{te:loc-quenched}]
We prove the result for subdivision tree vertices; the reticulation and leaf
cases are entirely analogous. Fix \(R\ge0\) and a finite rooted directed marked
graph \(G\), and write
\[
  p:=\Prb{U_R(\hat{\cL},p_0)\cong G}.
\]
Let \(U_n\subseteq V(\mN_n)\) be the set of subdivision tree vertices, so that
\(|U_n|=2(n-K_n)\). On the event \(\{|U_n|=0\}\), whose probability tends to
zero, we define the quantity below arbitrarily, say as \(0\). Set
\[
  X_n:=
  \begin{cases}
  \displaystyle
  \frac{1}{|U_n|}
  \sum_{u\in U_n}
  \mathbf{1}_{\{U_R(\mN_n,u)\cong G\}}, & |U_n|>0,\\[2ex]
  0, & |U_n|=0.
  \end{cases}
\]
We shall prove that
\[
  \Ex{X_n}\to p
  \qquad\text{and}\qquad
  \Va{X_n}\to0.
\]
The desired convergence in probability will then follow from Chebyshev's
inequality.
\smallskip

We first prove the convergence of the mean. Let \(u_n\) be a uniformly chosen
element of \(U_n\), conditionally on \(\mN_n\), whenever \(U_n\ne\varnothing\);
on \(\{U_n=\varnothing\}\), choose \(u_n\) arbitrarily. Since
\(\Prb{U_n=\varnothing}\to0\), the convention on this event is irrelevant. By
the definition of \(X_n\) and the uniform sampling of \(u_n\),
\begin{align*}
  \Ex{X_n}
  &=
  \Prb{U_R(\mN_n,u_n)\cong G\mid U_n\ne\varnothing}\,\Prb{U_n\ne\varnothing}
  +
  0\cdot\Prb{U_n=\varnothing}\\
  &=
  \Prb{U_R(\mN_n,u_n)\cong G}+o(1),
\end{align*}
where the $o(1)$ accounts for the event $\{U_n=\varnothing\}$, which has
probability tending to zero. The right-hand side converges to \(p\) by
Theorem~\ref{te:loc-unary}. Hence \(\Ex{X_n}\to p\).
\smallskip

We now turn to the second moment. Let \(u_n,u_n'\) be two vertices drawn independently and uniformly from \(U_n\),
conditionally on \(\mN_n\), whenever \(U_n\ne\varnothing\), and chosen
arbitrarily otherwise. By the definition of \(X_n\) and the iid sampling of
\(u_n,u_n'\),
\[
  \Ex{X_n^2}
  =
  \Prb{
    U_R(\mN_n,u_n)\cong G,\;
    U_R(\mN_n,u_n')\cong G
  }
  +o(1).
\]
The diagonal contribution is negligible: since \(|U_n|=2(n-K_n)\to\infty\) in
probability,
\[
  \Prb{u_n=u_n'}
  =
  \Exb{\frac{1}{|U_n|}\,;\,|U_n|>0}
  +\Prb{|U_n|=0}
  \;\longrightarrow\;0.
\]
It is therefore enough to study the case in which the two sampled vertices are
distinct.
\smallskip

Run two simultaneous BFS explorations, one from \(u_n\) and one from \(u_n'\),
revealing the underlying data \((T,Y,M)\) only when needed. Let \(J_n\) be the
event that
\begin{itemize}\itemsep0pt
\item each of \(u_n\) and \(u_n'\) is at distance \(>R\) in \(\tau_n^{\star}\) from
  the two vertices of \(T\) joined by its containing edge;
\item \(d_{\tau_n^{\star}}(u_n,u_n')>2R\);
\item whenever a matching partner \(\pi\) is revealed by either exploration,
  \(\pi\) is at distance \(>R\) in \(\tau_n^{\star}\) from the two vertices of \(T\)
  joined by the edge containing \(\pi\), and at distance \(>2R\) in
  \(\tau_n^{\star}\) from all subdivision vertices exposed by either exploration
  before \(\pi\) is revealed.
\end{itemize}
On \(J_n\) the two \(R\)-neighbourhood explorations do not interact in
\(\tau_n^{\star}\). The same bad-event estimate used in the proof of
Theorem~\ref{te:loc-unary}, now with at most \(2N\) revealed vertices
instead of \(N\), together with
\(\Prb{d_{\tau_n^{\star}}(u_n,u_n')\le2R}=O_R(n^{-1})\), gives
\[
  \Prb{J_n^c}\le \Prb{K_n\notin I_n}+O_R(n^{-1/2})\to0.
\]
Moreover, on \(J_n\), the deterministic comparison from the proof of
Theorem~\ref{te:loc-unary} applies simultaneously to the two explorations. Since
\(U_R(\hat{\cL},p_0)\) is deterministic, we have
\[
  p=\mathbf 1_{\{U_R(\hat{\cL},p_0)\cong G\}}.
\]
Thus, on \(J_n\),
\[
  \mathbf 1_{\{U_R(\mN_n,u_n)\cong G\}}
  \mathbf 1_{\{U_R(\mN_n,u_n')\cong G\}}
  =p^2.
\]
Since both sides are bounded by \(1\), it follows that
\[
  \left|
  \Prb{
    U_R(\mN_n,u_n)\cong G,\;
    U_R(\mN_n,u_n')\cong G
  }
  -p^2
  \right|
  \le \Prb{J_n^c}\to0.
\]
Consequently \(\Ex{X_n^2}\to p^2\).
\smallskip

Combining the first and second moment estimates,
\[
  \Va{X_n}
  =
  \Ex{X_n^2}-\Ex{X_n}^2
  \;\longrightarrow\;p^2-p^2=0.
\]
Thus, for every \(\varepsilon>0\),
\[
  \Prb{|X_n-p|>\varepsilon}
  \le
  \Prb{|X_n-\Ex{X_n}|>\varepsilon/2}
  +
  \mathbf{1}_{\{|\Ex{X_n}-p|>\varepsilon/2\}}
  \;\longrightarrow\;0.
\]
Hence \(X_n\convp p\), as claimed.
\smallskip

As in the proof of Theorem~\ref{te:loc-unary}, the \(K_n-1\) binary tree
vertices form an asymptotically negligible proportion of all tree
vertices. More precisely, the difference between the empirical
proportions over all tree vertices and over subdivision tree vertices is
bounded by
\[
\frac{K_n-1}{2(n-K_n)+(K_n-1)}
=
\frac{K_n-1}{2n-K_n-1}
\convp0.
\]
Hence the same limit holds when the empirical proportion is taken over all
tree vertices.
\smallskip

The reticulation and reticulation-leaf cases follow by the same argument,
replacing \(U_R(\hat{\cL},p_0)\) by \(U_R(\hat{\cR},\rho_0)\) and
\(U_R(\hat{\mN}^{\mathrm{lf}},\ell_0)\), respectively, and using
Theorems~\ref{te:loc-ret} and~\ref{te:loc-leaf} in place of
Theorem~\ref{te:loc-unary}. Finally, the \(K_n\) tree leaves form an
asymptotically negligible proportion of all \(n\) leaves, since
\[
\frac{K_n}{n}\convp0.
\]
Hence the same quenched limit holds when the empirical proportion is taken
over all leaves.
\end{proof}

We conclude this subsection by proving the quenched local convergence
result for the empirical neighbourhood distribution over all vertices, this is Corollary \ref{co:loc-quenched-vertex}.

\begin{proof}[Proof of Corollary~\ref{co:loc-quenched-vertex}]
Decompose the empirical neighbourhood distribution into the contributions
from tree vertices, reticulation vertices, and leaves. The proportions of these three vertex types converge in probability
to \(1/2\), \(1/4\), and \(1/4\), respectively, see the proof of Theorem \ref{te:loc-vertex}. Hence, by
Theorem~\ref{te:loc-quenched},
\[
\begin{aligned}
\frac{\#\{v\in V(\mN_n):U_R(\mN_n,v)\cong G\}}
     {|V(\mN_n)|}
\;\convp\;&
\frac12\Prb{U_R(\hat{\cL},p_0)\cong G}\\
&+\frac14\Prb{U_R(\hat{\cR},\rho_0)\cong G}\\
&+\frac14\Prb{U_R(\hat{\mN}^{\mathrm{lf}},\ell_0)\cong G}.
\end{aligned}
\]
By the definition of the mixture \((N^*,v^*)\) in
Theorem~\ref{te:loc-vertex}, the expression on the right equals
\[
\Prb{U_R(N^*,v^*)\cong G}.
\]
This proves the result.
\end{proof}

\section{Acknowledgements}

This research was funded in part by the Austrian Science Fund (FWF) 10.55776/PAT6732623 and 10.55776/PAT1382025. For open access purposes, the authors have applied a CC BY public copyright license to any author-accepted manuscript version arising from this submission. Research of V.~J.\ Maci\'a was partially funded by grants PID2021-123151NB-I00, PID2023-148028NB-I00 and PID2025-167726NB-I00 from the Spanish Government.

\bibliographystyle{abbrv}
\bibliography{phylo2}

\end{document}